\documentclass[11pt]{amsart}

\usepackage{hyperref}
\usepackage{mathrsfs}
\usepackage{mathtools}

\newcounter{my_enumerate_counter}
\newcommand{\pushcounter}{\setcounter{my_enumerate_counter}{\value{enumi}}}
\newcommand{\popcounter}{\setcounter{enumi}{\value{my_enumerate_counter}}}

\newcommand{\cstu}{\mathrm{C}^*_u}
\newcommand{\roeq}{\mathrm{Q}^*_u}
\DeclareMathOperator{\Ker}{Ker}
\DeclareMathOperator{\Fin}{Fin}

\DeclareMathOperator{\End}{End}
\DeclareMathOperator{\ind}{ind}

\newcommand{\bfB}{\mathbf B}
\newcommand{\bfD}{\mathbf D} 
\newcommand{\bfE}{\mathbf E} 
 
\newcommand{\bbP}{\mathbb P}

\newcommand{\Clop}{\mathrm{Clop}}

\newcommand{\bfF}{\mathbf F} 
\DeclareMathOperator{\Ext}{Ext}
\DeclareMathOperator{\Ad}{Ad}

\newcommand{\norm}[1]{\left\lVert #1 \right\rVert}

\DeclareMathOperator{\id}{id}
\newcommand{\cU}{\mathcal U}

\newcommand{\forces}{\Vdash}

\newcommand{\bbF}{{\mathbb F}}

\newcommand{\bbZ}{{\mathbb Z}}
\newcommand{\bbT}{\mathbb T}

\newcommand{\bbN}{{\mathbb N}}
\newcommand{\bbC}{\mathbb C}
\newcommand{\bbQ}{\mathbb Q}
\newcommand{\bbR}{\mathbb R}
\newcommand{\cJ}{{\mathcal J}}
\newcommand{\cI}{{\mathcal I}}

\newcommand{\cM}{{\mathcal M}}
\newcommand{\cN}{{\mathcal N}}

\newcommand{\calL}{\mathcal L}

\newcommand{\cX}{{\mathcal X}}

\newcommand{\cY}{{\mathcal Y}}
\newcommand{\cZ}{{\mathcal Z}}
\newcommand{\cA}{{\mathcal A}}

\newcommand{\fc}{\mathfrak c} 
\newcommand{\fd}{\mathfrak d} 

\newcommand{\fA}{\mathfrak A}

\newcommand{\rs}{\restriction}

\newcommand{\cC}{\mathcal C}
\newcommand{\cF}{\mathcal F}

\newcommand{\cV}{\mathcal V}
\newcommand{\cG}{\mathcal G}

\newcommand{\cB}{\mathcal B}
\newcommand{\cK}{\mathcal K}
\newcommand{\cQ}{\mathcal Q}

\newcommand{\calD}{\mathcal D}
\newtheorem{thm}{Theorem}[section]
\newtheorem{theorem}[thm]{Theorem}

\newtheorem{corollary}[thm]{Corollary}

\newtheorem{conjecture}[thm]{Conjecture}
\newtheorem{metaconj}[thm]{Conjecture Template}

\newtheorem{question}[thm]{Question}
\newtheorem{claim}[thm]{Claim}

\newtheorem{lemma}[thm]{Lemma}

\newtheorem{proposition}[thm]{Proposition}
\newtheorem{convention}[thm]{Convention}

\DeclareMathOperator{\Aut}{Aut}

\DeclareMathOperator{\Homeo}{Homeo}
\DeclareMathOperator{\propg}{prop}

\theoremstyle{definition}

\newtheorem{definition}[thm]{Definition}
\newtheorem{problem}[thm]{Problem}
\newtheorem{example}[thm]{Example}

\newcommand{\cP}{\mathcal P} 
 
\DeclareMathOperator{\Th}{Th}

\newcommand{\cstar}{$\mathrm{C}^*$}

\DeclareMathOperator{\dom}{dom}

\newcommand{\sfF}{\mathsf F}
\newcommand{\sfG}{\mathsf G}
\newcommand{\CH}{\mathrm{CH}}
\newcommand{\MA}{\mathrm{MA}}
\newcommand{\MAsigmalinked}{\mathrm{MA_{\sigma-linked}}}
\newcommand{\OCA}{\mathrm{OCA_T}}
\newcommand{\OCAsharp}{\mathrm{OCA^\#}}
\newcommand{\OCAi}{\mathrm{OCA_\infty}}
\newcommand{\ZFC}{\mathrm{ZFC}}
\newcommand{\PFA}{\mathrm{PFA}}
\newcommand{\bSigma}{\mathbf \Sigma}
\newcommand{\bPi}{\mathbf \Pi}

\DeclareMathOperator{\Null}{Null}

 \numberwithin{equation}{section}

\renewcommand{\phi}{\varphi}

\DeclareMathOperator{\Part}{Part_{\bbN}}
\DeclareMathOperator{\Exh}{Exh}
\DeclareMathOperator{\NWD}{NWD}
\DeclareMathOperator{\NULL}{NULL}

\newcommand{\cZlog}{\cZ_{\log}}

\newcommand{\META}{Conjecture Template~\ref{Meta.1}}

\DeclareMathOperator{\even}{even}\DeclareMathOperator{\odd}{odd}

\newcommand{\twolo}{\{0,1\}^{<\bbN}}\newcommand{\twoo}{\{0,1\}^{\bbN}}

\newcommand{\ucup}{\underline\cup}

\newcommand{\sfD}{\mathsf D}

\usepackage{enumitem}
\newcommand{\enumthree}{\setenumerate[1]{label=(\thesubsection.\arabic*)}}
\newcommand{\enumtwo}{\setenumerate[1]{label=(\thesection.\arabic*)}}
\newcommand{\enumone}{\setenumerate[1]{label=(\arabic*)}}

\setenumerate[2]{label*=(\arabic*)}

\usepackage{amsmath, amssymb}
\usepackage{tikz}
\usepackage{tikz-cd}
\usepackage{enumitem}

\date{\today}

\title{Corona rigidity}

\author[I. Farah]{Ilijas Farah}

\address{Department of Mathematics and Statistics\\
York University\\
4700 Keele Street\\
North York, Ontario\\ Canada, M3J
1P3 and Matemati\c vki Institut SANU, Kneza Mihaila 36, Belgrade 11001, Serbia}

\urladdr{https://ifarah.mathstats.yorku.ca}

\email{ifarah@yorku.ca}

\author[S. Ghasemi]{Saeed Ghasemi}
\address{Department of Mathematics and Statistics\\
York University\\
4700 Keele Street\\
North York, Ontario\\ Canada, M3J 1P3 and Department of Mathematical Sciences, Lakehead University\\
	955 Oliver Road,
	Thunder Bay, Ontario\\
	Canada P7B 5E1 (current address)}
\urladdr{https://www.lakeheadu.ca/users/G/sghasem2/node/201808}
\email{sghasem2@lakeheadu.ca}

\author[A. Vaccaro]{Andrea Vaccaro}
\address{Mathematisches Institut, Fachbereich Mathematik und Informatik der
Universit\"at M\"unster, Einsteinstrasse 62, 48149 M\"unster, Germany.}
\email{avaccaro@uni-muenster.de}
\urladdr{https://sites.google.com/view/avaccaro}

\author[A. Vignati]{Alessandro Vignati}
\address{Institut de Math\'ematiques de Jussieu (IMJ-PRG)\\
Universit\'e Paris Cit\'e and Institut Universitaire de France\\
B\^atiment Sophie Germain\\
8 Place Aur\'elie Nemours \\ 75013 Paris, France}
\email{vignati@imj-prg.fr}
\urladdr{https://www.automorph.net/avignati}

\begin{document}

\maketitle

\begin{abstract}
We give a unified overview of the study of the effects of additional set theoretic axioms on quotient structures. Our focus is on rigidity, measured in terms of existence (or rather non-existence) of suitably \emph{non-trivial} automorphisms of the quotients in question. A textbook example for the study of this topic is the Boolean algebra $\cP(\bbN)/\Fin$, whose behavior is the template around which this survey revolves: Forcing axioms imply that all of its automorphisms are \emph{trivial}, in the sense that they are induced by almost permutations of $\bbN$, while under the Continuum Hypothesis this rigidity fails and $\cP(\bbN)/\Fin$ admits uncountably many non-trivial automorphisms. We consider far-reaching generalisations of this phenomenon and present a wide variety of situations where analogous patterns persist, focusing mainly (but not exclusively) on the categories of Boolean algebras, \v Cech--Stone remainders, and \cstar-algebras. We survey the state of the art and the future prospects of this field, discussing the major open problems and outlining the main ideas of the proofs whenever possible.

\end{abstract}
\begin{quote}
	\emph{Dedicated to Saharon Shelah, without whom the topic of this survey would not have existed.} 
\end{quote}

\setcounter{secnumdepth}{4}
\setcounter{tocdepth}{2}

\section{Introduction}

%\emph{The} motivating example for the study of this topic is the quotient of the Boolean algebra $\mathcal P(\bbN)$ by the Fr\'echet ideal $\Fin$. 
%In this case, we say that an automorphism of $\cP(\bbN)/\Fin$ is \emph{trivial}
%if it is induced by an {almost permutation} of $\bbN$, namely a bijection between two cofinite subsets of $\bbN$. Equivalently, an automorphism is trivial if it can be lifted
%to a Boolean algebra homomorphism on $\cP(\bbN)$.

In 1956 W. Rudin proved that if the Continuum Hypothesis (CH) holds then $\cP(\bbN)/\Fin$, the quotient of the Boolean algebra of $\cP(\bbN)$ by the Fr\'echet ideal (i.e., the ideal of finite subsets of $\bbN$), $\Fin$, has $2^\mathfrak{c}$ automorphisms\footnote{This is not the whole truth, see \S\ref{S.Abel}.}, with $\mathfrak{c} = | \bbR |$ being the size of the continuum (\cite{Ru}). In 1979 Shelah described a forcing extension of the universe where all automorphisms of $\cP(\bbN)/\Fin$ are induced by an {almost permutation} of $\bbN$, namely a bijection between two cofinite subsets of $\bbN$ (\cite{Sh:Proper}). We refer to these automorphisms as \emph{trivial}.
Later on, in the work of several authors, Shelah’s conclusion was proved to follow from forcing axioms (\cite{ShSte:PFA, Ve:OCA}).
Since there are only $\fc$ almost permutations of $\bbN$, Shelah's theorem contradicts Rudin's, while Rudin's result shows the failure of the rigidity of $\cP(\bbN)/\Fin$ under CH.

Rudin's result is, by today's standards, trivial: The algebra $\cP(\bbN)$ is countably saturated\footnote{I.e., $\aleph_1$-saturated.} (in model-theoretic sense; this is  not to be confused with the set-theoretic notion of a saturated ideal, also not to be confused with the notion of countable saturation as defined in \cite[\S 2.3]{ChaKe}), hence $\CH$ implies that it is saturated. A routine back-and-forth argument produces a complete binary tree of height $\aleph_1=\fc$ whose branches are distinct automorphisms. The fact that the theory of atomless Boolean algebras admits elimination of quantifiers facilitates the construction by implying that any partial isomorphism between countable subalgebras of $\cP(\bbN)/\Fin$ can be extended to an automorphism. This is, however, a convenience rather than a necessity. 
On the other hand, Shelah's construction of an oracle-cc forcing extension of the universe in which every automorphism of $\cP(\bbN)/\Fin$ is given by an almost permutation is, unlike most of the 1970s memorabilia, still as formidable as when it first appeared and its ramifications reverberate throughout the subject this survey is about. 

Fast-forwarding 30 years, the same pattern was isolated again in a different quotient structure, this time originating in the context of algebras of operators on Hilbert
spaces (see \S\ref{ss:pre} for the definitions): The Calkin algebra. The Calkin algebra $\cQ(H)$ is the quotient of the \cstar-algebra of all bounded linear operators on a complex, infinite-dimensional, separable Hilbert space~$H$, by the ideal of all compact operators. In this case, we say that an automorphism of $\cQ(H)$ is \emph{inner} if it is implemented by a unitary in $\cQ(H)$ (equivalently, it is lifted by a conjugation by an isometry between two closed subspaces of $H$ with finite co-dimension).
In the early 20th century Weyl and von Neumann initiated the study of unitary equivalence of self-adjoint operators modulo compact perturbations. Their elegant characterisation of this equivalence was extended to normal operators (\cite{berg, sikonia}) and then to essentially normal operators, those whose images in $\mathcal Q(H)$ are normal (\cite{BrDoFi:Unitary}). While two normal operators are unitarily equivalent modulo compact perturbations if and only if their images in the Calkin algebra are conjugate by an automorphism, it was not clear whether this conclusion could be extended to essentially normal operators. 
 In 1973 (a few years before Shelah’s result) Brown, Douglas, and Fillmore asked whether $\cQ(H)$ has an \emph{outer} (i.e. not inner) automorphism. If one replaces `trivial automorphism’ with `inner automorphism’ then the exact analogs of Rudin’s and Shelah’s results hold for $\cQ(H)$: Phillips and Weaver (\cite{PhWe:Calkin}) showed that CH implies that $\cQ(H)$ admits $2^\fc$
 automorphisms (most of them outer), while the first author (\cite{Fa:All}) proved that under forcing axioms all of its automorphisms are inner.%\footnote{For some comedic relief, the reader may want to compare the titles of these two papers. Both results are correct.} 
 This implies that, consistently with ZFC, two essentially normal operators are unitarily equivalent modulo compact perturbations if and only if their images in the Calkin algebra are conjugate by an automorphism. It remains open whether this is a theorem of ZFC. 

These theorems taken together form only the tip of the iceberg and this survey is largely motivated by the remarkable fact that many quotient structures follow this
dichotomous paradigm. In \S\ref{S.Intro} we will introduce a general framework of Borel quotient structures, leading to the following question.

\begin{question}\label{Q.main}
	Under what assumptions is it true that every isomorphism (homomorphism) between Borel quotient structures $\cM/E$ and $\cN/F$ has a Borel lifting?
\end{question}

The assumptions referred to in this question come in three varieties: 
\begin{enumerate}
	\item \label{1.item.1} The assumptions on the original structures $\cM$ and $\cN$. 
	\item \label{1.item.2} The assumptions on the congruence relations $E$ and $F$. 
	\item \label{1.item.3} The additional set-theoretic assumptions. 
\end{enumerate}
In terms of~\eqref{1.item.1}, we will mostly concentrate on the categories in which we have something nontrivial to say: Boolean algebras, \cstar-algebras, fields, linear orderings, trees, and sufficiently random graphs (see also questions in \S\ref{S.Absoluteness}). Regarding~\eqref{1.item.2}, we will assume that $E$ and $F$ are Borel (but see a short paragraph in \S\ref{S.other} on ultrapowers). The two long and central sections, \S\ref{S.Independence} and \S\ref{S.Independence2}, are devoted to~\eqref{1.item.3}.

In this survey we endeavour to revisit the most influential results and discuss the new developments in this field, while also emphasizing the analogies and pointing to the differences
in the methods used in these rigidity proofs. Our ultimate goal is to give a broad exposure to the emerging unifying theory of this subject, and to provide a unique meta-mathematical framework capable of describing, using a common language, the concept of rigidity of Borel quotients across different categories.

\subsubsection*{Summary}
This manuscript was written by four authors, and each has brought their view of the subject as well as some idiosyncrasies. It is composed of twelve sections. Each section is as self-contained and autonomous as possible (accordingly, some redundancies can be noted here and there), yet most sections rely on \S\ref{S.Intro}. One can safely skip any number of sections as there is only a minimal amount of references to other sections. The survey is also sprinkled with conjectures and open problems.

In \S\ref{S.Intro} we present the abstract rigidity problem.~\S\ref{S.Abel} serves as a warm-up. Here we present some basic examples of quotient structures from the worlds of Boolean algebras and of topological spaces, in the form of \v{C}ech--Stone remainders. In \S\ref{3a.BDF} we introduce \cstar-algebras as well as the framework which eventually motivated and ignited the study of rigidity questions in the noncommutative setting. \S\ref{S.Ulam} is devoted to Ulam-stability and to the problem of obtaining algebraically trivial liftings from topologically trivial ones.\footnote{An isomorphism is called topologically trivial if it has a Borel lifting, see Definition~\ref{Def.Trivial}.} This topic, as interesting and as necessary for understanding quotient rigidity, has little to do with logic and readers interested only in logical aspects may want to skip it. \S\ref{S.Independence} and \S\ref{S.Independence2} are the main body of this survey, and they portray two faces of the same coin. The former describes various constructions of isomorphisms of Borel quotient structures which are not topologically trivial%(i.e., are without a Borel lifting)
, under set-theoretic assumptions such as CH and some of its weakenings. The latter, on the other hand, presents rigidity results, either obtained via forcing or as a consequence of Forcing Axioms, and instances where all automorphisms of certain quotient structures are topologically trivial. In \S\ref{S.endo} we explore rigidity results for maps more general than isomorphisms and automorphisms, such as endomorphisms of the Calkin algebra. \S\ref{S.large} is devoted to large quotients of the form $\cP(\kappa)/\Fin$ and $\cQ(\ell_2(\kappa))$ for uncountable $\kappa$, while \S\ref{S.Roe} focuses on rigidity results on uniform Roe and Higson coronas, \cstar-algebras arising from coarse metric spaces. In \S\ref{S.other} we provide a quick overview of a selection of topics loosely related to the contents of this survey. Finally \S\ref{S.Absoluteness} is reserved for remarks of metamathematical nature on the rigidity problems discussed in the rest of the survey.

\subsection*{Acknowledgments} We would like to thank Jim Broadbent, Juris Stepr\= ans, and Jacek Tryba for pointing to inaccuracies and typos in the original draft of this survey.
We are also thankful to the anonymous referee for their valuable suggestions on the exposition of the material in this manuscript.

\subsection*{Funding}
IF was supported by NSERC.
AVa was supported by the Deutsche Forschungsgemeinschaft (DFG, German Research Foundation) under Germany’s Excellence Strategy EXC 2044 –390685587, Mathematics M\"unster: Dynamics–Geometry–Structure, through SFB 1442, by the ERC Advanced Grant 834267 - AMAREC, and by the European Union’s Horizon 2020 research and innovation program under the Marie Sk\l odowska-Curie grant agreement No. 891709. AVi was supported by the Institut Universitaire de France.

\tableofcontents

%introduction 
 \section{The general rigidity question}\label{S.Intro}
 \enumtwo 
In this section we introduce a unifying framework for describing all the categories and rigidity instances considered in the rest of the survey. Readers interested in less abstract results are encouraged to skip ahead and refer to this section as needed. Two critically
important notions are that of \emph{topologically trivial} and \emph{algebraically trivial} automorphisms (respectively Definition \ref{Def.Trivial} and \ref{Def.AlgebraicallyTrivial}),
before which we isolate the class of quotients we aim to study, with a definition that generalises both the Boolean algebra $\cP(\bbN)/\Fin$ and the Calkin algebra $\cQ(H)$.
\begin{definition}\label{Def.Borel}
A \emph{Borel structure} $\cM$ in a signature $\calL$ is a structure whose universe~$M$ is a Polish space and such that the interpretations of all functions and relations are Borel (a general study of Borel structures was initiated by H. Friedman, see~\cite{Ste:Borel}). A \emph{Borel quotient structure} is obtained from $\cM$ by specifying a congruence $E$ on the universe $M$ that is a Borel equivalence relation; it is denoted $\cM/E$.\footnote{An equivalence relation on a Polish space $M$ is said to be \emph{Borel} if it is Borel when identified with a subset of $M^2$.} The quotient map is denoted~$\pi_E$. 
\end{definition}

Suppose that $\Phi\colon \cM/E\to \cN/F$ is a homomorphism between Borel quotient structures. A \emph{lifting} of $\Phi$ is a function $\Phi_*\colon \cM\to \cN$ such that the diagram on Fig.~\ref{Fig.lifting} commutes. 

\begin{figure}[h]
\begin{tikzpicture}
 \matrix[row sep=1cm,column sep=1.5cm] 
 {
& & \node (M1) {$\cM$}; & \node (M2) {$\cN$};&
\\
& & \node (Q1) {$\cM/E$}; & \node (Q2) {$\cN/F$} ;
\\
};
\draw (M1) edge [->] node [above] {$\Phi_*$} (M2);
\draw (Q1) edge [->] node [above] {$\Phi$} (Q2);
\draw (M1) edge [->] node [left] {$\pi_E$} (Q1);
\draw (M2) edge [->] node [left] {$\pi_F$} (Q2);
\end{tikzpicture}
\caption{A lifting $\Phi_*$ of $\Phi$.}\label{Fig.lifting}
\end{figure}

The Axiom of Choice (assumed throughout) implies that every homomorphism $\Phi$ between quotient structures has a lifting. Imposing the existence of a lifting with additional requirements results in more interesting assertions. 

\begin{definition} \label{Def.Trivial} An isomorphism $\Phi$ as in Fig.~\ref{Fig.lifting} is \emph{topologically trivial} if it has a Borel-measurable lifting. 
\end{definition}

As pointed out by many (not least the referee), a better terminology would be `Borel liftable’. Regrettably, this terminology has already taken root and it may be too late to change it now. The choice of Borel-measurable instead of (for example) continuous may appear to be arbitrary, in some situations the existence of a Borel-measurable lifting already implies the existence of a continuous one.

The (even more desirable) notion of an \emph{algebraically trivial} isomorphism is a bit trickier to define.
The obvious attempt, to require $\Phi_*$ to be an isomorphism, is too restrictive---defining an automorphism, even topologically trivial, of $\cP(\bbN)/\Fin$ with no lifting that is an automorphism is quite easy. (It is defined by its lifting, the map that sends $A\subseteq \bbN$ to $\{n+1\mid n\in A\}$.) In the case of Boolean algebras of the form $\cP(\bbN)/\cI$ for a Borel ideal $\cI$, a lifting is \emph{algebraically trivial} if it is a Boolean algebra homomorphism (this definition, and that of a topologically trivial lifting, comes from~\cite{Fa:Rigidity}). The existing theorems justify this definition as optimal, as we will see in \S\ref{S.Independence2}. However, in the case of more complicated objects, such as the Calkin algebra $\mathcal Q(H)$, it is not difficult to define a topologically trivial automorphism that cannot be lifted by a $^*$-homomorphism (see \S\ref{3a.BDF}; the obstruction is similar to that in the case of Boolean algebras---the automorphism is conjugation by a unitary with a nontrivial Fredholm index). In short, the definition of algebraically trivial depends on the category of interest. To keep a unifying approach, we import the following Gordian definition from~\cite{vignati2018rigidity}. 

\begin{definition} \label{Def.AlgebraicallyTrivial} An isomorphism $\Phi$ as in Fig.~\ref{Fig.lifting} is \emph{algebraically trivial} if it has a lifting that preserves as much of the algebraic structure as possible. 
\end{definition}

We move on to examples of trivial isomorphisms outside the abovementioned setting of Borel quotients of $\mathcal P(\bbN)$. The second part of the following example gives some justification for the current form of Definition~\ref{Def.AlgebraicallyTrivial}.

\begin{example} \label{Ex.inner} Suppose that $\cM=\cN$ is a Polish group and $E=F$ is the coset relation associated with a normal Borel subgroup $\cG\subseteq\cM$. If an automorphism~$\Phi$ of the quotient group $\cM/\cG$ is \emph{inner} (i.e., the conjugation $f\mapsto g f g^{-1}$ for some group element $g$) then $\Phi$ has a lifting that is both algebraically and topologically trivial.

In case of groups, every element defines an inner automorphism hence every inner automorphism of the quotient lifts to an inner automorphism. In more complicated categories, such as \cstar-algebras, this is no longer true. Inner automorphisms correspond to conjugation by unitary elements and not every unitary $u$ in the quotient can be lifted to a unitary. Nevertheless, if $a$ lifts $u$ then conjugation by $a$ is Borel-measurable and it preserves as much of the algebraic structure as possible as in Definition~\ref{Def.AlgebraicallyTrivial}. The appropriate definition of algebraically trivial isomorphism in this setting has to therefore accommodate for these conjugations.
\end{example}

An analogous statement to that in Example~\ref{Ex.inner} holds for \cstar-algebras (see \S\ref{3a.BDF}) and in any other category that has naturally defined inner automorphisms.\footnote{This does not obviously apply to the abstract definition of an inner automorphism studied in the context of abstract functorial classification in \cite{Ell:Towards}.} 

\begin{example}
If the signature of Borel structures $\cM$ and $\cN$ is empty and $E$, $F$ are (not necessarily Borel) equivalence relations, then the existence of an injective morphism from $\cM/E$ into $\cN/F$ with a Borel lifting is, by definition, equivalent to the assertion that the equivalence relation~$E$ is Borel-reducible to $F$ (e.g.,~\cite{Hj:Borel}). 
\end{example}

%\begin{example} \label{Ex.Borel-ctns}
%When both $\cM$ and $\cN$ are equal to the Boolean algebra~$\cP(\bbN)$ (considered with the Cantor set topology) and both $E$~and~$F$ are the equality modulo finite, then the an automorphism $\Phi$ has a topologically trivial lifting if and only if it has a lifting that is continuous \emph{and} a Boolean algebra homomorphism (see \S\ref{S.Ulam}). 
%\end{example}

The question of whether it is possible to modify a topologically trivial lifting to an algebraically trivial one will be treated in \S\ref{S.Ulam}. 
Our focus in \S\ref{S.Independence}~and~\S\ref{S.Independence2} will be on the other side of the coin, Question~\ref{Q.main} from the introduction: 	Under what assumptions is it true that every isomorphism (homomorphism) between Borel quotient structures has a Borel lifting?
The remaining part of this survey revolves around the following meta-conjecture related to this question. 

\begin{metaconj} \label{Meta.1} For Borel quotient structures $\cM/E$ and $\cN/F$ we have the following statements. 
\enumone
\begin{enumerate}
 \item The Continuum Hypothesis (CH) implies that $\cM/E$ has $2^{\fc}$ automorphisms (and therefore $2^{\fc}$ topologically nontrivial automorphisms). It also implies that $\cM/E$ and $\cN/F$ are isomorphic, unless there is an obvious obstruction for this.\footnote{A reader familiar with Woodin's $\Sigma^2_1$ absoluteness theorem will recognise its relevance; more on this in \S\ref{S.Absoluteness}.} 
 \item \label{item2:conjecture} Forcing Axioms imply that all automorphisms of $\cM/E$ are topologically trivial. They also imply that every isomorphism between $\cM/E$ and $\cN/F$ is topologically trivial.\footnote{A heuristic argument for the assertion that there are no nontrivial isomorphisms has been given by K. P. Hart: If there were one, we would have thought of it. All work on Conjecture~\ref{Meta.1}.\ref{item2:conjecture} can be viewed as partial results towards proving Hart’s apothegm.}
 \end{enumerate}
\end{metaconj}

The results by Rudin and Shelah in \cite{Ru} and \cite{Sh:Proper} and by Phillips--Weaver and the first author in \cite{PhWe:Calkin} and \cite{ Fa:All} that were discussed in the introduction, precisely demonstrate that both instances of Conjecture Template \ref{Meta.1} turn into theorems for $\cP(\bbN)/\Fin$ and $\cQ(H)$. More recently, it was shown in \cite{de2023trivial} that \ref{item2:conjecture} can be extended to reduced products of fields, linear orderings, trees, and sufficiently random graphs. 
The present survey covers a remarkable number of Borel quotient structures that follow the same pattern, and we find the possibility that these results are instances of a general theorem taunting, to say the least. 

The following easy example from~\cite{Fa:Liftings} shows that there are situations in which the second part of \META{} fails, and this should not come as a surprise. 

\begin{example} \label{Ex.F2} Consider the group $G:=(\bbZ/2\bbZ)^\bbN$ with respect to the product topology. If $E$ is the coset equivalence relation associated with a subgroup of $G$, then the isomorphism type of the quotient $G/E$ depends only on its cardinality, $\kappa$. Moreover, if $\kappa$ is infinite then $G/E$ has $2^\kappa$ automorphisms. In particular, if $E$ is associated to the subgroup\footnote{$\forall^\infty$ stands for the quantifier `for all but finitely many'.} $\{a\in G \mid (\forall^\infty n)a_n=0\}$ (or any other countable subgroup), then $G/E$ has an automorphism with no Borel lifting. 

The reason for the first assertion is that $G/E$ is a vector space over the two-elements field $F_2$. The second follows from the fact that $2^{\fc}$ is strictly greater than $\fc$, the number of Borel functions on a Polish space. 

A model-theorist will see an alternative explanation: The theory of $G$ is $\aleph_0$-stable (thus all of its uncountable models are saturated), admits elimination of quantifiers, and is preserved under quotients. An even better explanation is given below in Theorem~\ref{T.dividing-line} below, taken from \cite{debondt2023saturation}. 
\end{example} 

The following incomparably deeper example, giving a scenario (still within the confines of group theory) in which the first part of \META{} fails, comes from~\cite{truss1996recovering}. 
 
\begin{example}\label{Ex.AlCoMac} 
 Consider $S_\infty$, the group of all permutations of $\bbN$, with its unique Polish group topology (see e.g.,~\cite{Ke:Classical}). Let $H$ be the subgroup of finitely supported permutations (i.e., those $f\in S_\infty$ that fix all but finitely many $n$). Then every automorphism $\Phi$ of $S_\infty/H$ has a continuous lifting. Moreover, there exists a bijection $h$ between cofinite subsets of $\bbN$ such that~$\Phi$ is implemented by conjugation with $h$: $\Phi(f/H)=h\circ f \circ h^{-1}/H$.

Here is an alternative formulation. Consider the semigroup of bijections between cofinite subsets of $\bbN$ (they are called \emph{almost permutations} of $\bbN$), equipped with composition. The quotient of this semigroup by the finitely supported permutations is a group, and by~\cite{truss1996recovering} all automorphisms of this group are inner. 
 \end{example}

A simpler proof of Truss's theorem used in the previous example was announced in~\cite{AlCoMac}. There is unfortunately a gap in this proof; see~\cite[Remark~11.A.6]{cornulier2019near}.

Going back to Conjecture Template \ref{Meta.1}, while its first part is mainly developed under the assumption of CH (and some of weaker forms of it), the second part relies on the use of \emph{Forcing Axioms}.
{Forcing axioms} are extensions of the Baire Category Theorem, asserting that in certain compact Hausdorff spaces $K$ the intersection of $\aleph_1$ dense open subsets is dense. The strength of a forcing axiom is given by the category of compact Hausdorff spaces to which it applies. Already the case when $K$ is the unit interval contradicts $\CH$ (for every $x\in [0,1]$, the set $[0,x)\cup (x,1]$ is dense and open). From Martin's Axiom, via the Proper Forcing Axiom, to Martin's Maximum (see e.g.,~\cite{Ku:Book,Mo:PFA}), the study of forcing axioms and their applications has had, and still has, a special place in set theory.

 \section{The abelian case: Boolean algebras and \v{C}ech--Stone remainders}\label{S.Abel}
\enumtwo 

Historically the first instance of \META{} comes from the study of  Boolean algebras  $\cP(\bbN)/\Fin$ and $\cP(\bbN)/\cI$ for other analytic ideals on $\bbN$ and of \v{C}ech--Stone remainders (also known as coronas) of locally compact spaces.
In this section,  we will formulate a specific instance of \META{} (Conjecture \ref{conj:boolean}) and give a rigorous notion of \emph{algebraically trivial}
automorphism (Definition \ref{def:bool_trivial}) for these families of quotients, but we shall refrain from stating any specific independent result, whose discussion is postponed to \S\ref{S.Independence} and~\S\ref{S.Independence2}.

As the title of the section suggests, we refer to these as \emph{abelian} or \emph{commutative} quotients, implicitly hinting at their correspondence, via the Gelfand duality, with abelian \cstar-algebras. A more detailed explanation of this connection, as well as the introduction of the more general framework of noncommutative quotients, is deferred to \S \ref{3a.BDF}.
 
A good vantage point for our analysis is given by the Stone duality. The category of Boolean algebras and homomorphisms is equivalent to the category of Stone spaces (i.e. compact, totally disconnected Hausdorff spaces) and continuous maps between them. To a Stone space one associates the algebra of its clopen sets, and to a Boolean algebra one associates its Stone space, that is the space of all ultrafilters on it. 

This (contravariant) correspondence associates $\cP(\bbN)$ to $\beta\bbN$, the \v Cech--Stone compactification of $\bbN$ (taken with the discrete topology), and $\cP(\bbN)/\Fin$ to the \v Cech--Stone remainder $\beta\bbN\setminus\bbN$. While the Boolean algebra $\cP(\bbN)/\Fin$ is homogeneous (i.e., its automorphism group acts transitively on the elements other than $[\emptyset]$ and $[\bbN]$), it is not obvious whether $\beta\bbN\setminus \bbN$ is a homogeneous compact Hausdorff space. 

In his proof that $\CH$ implies that $\beta\bbN\setminus\bbN$ is not homogeneous, Rudin (\cite{Ru}) showed that, under CH, $\cP(\bbN)/\Fin$ has $2^{\fc}$ automorphisms. The family of almost permutations of $\bbN$ has size $\fc$, hence Rudin's result provided the historically first confirmation of \META. The nonhomogeneity of $\beta\bbN\setminus \bbN$ in $\ZFC$ was later proven by Frol\`ik (\cite{Frolik}), and we refer to \cite{vM:Introduction} for more on the peculiar dependence of $\beta\bbN\setminus \bbN$ on the choice of set-theoretic axioms. 

We now turn to two classes of Borel quotient structures that generalise $\mathcal P(\bbN)/\Fin$ in different directions:
\begin{itemize}
\item \S \ref{ss.borel_quotients} Quotients of the form $\mathcal P(\bbN)/\mathcal I$ where $\mathcal I$ is a Borel ideal, %(primarily as Boolean algebras, occasionally as groups), 
\item \S \ref{ss.remainders} \v{C}ech--Stone remainders of locally compact spaces.
\end{itemize}

\subsection{Borel ideals in $\mathcal P(\bbN)$.} \label{ss.borel_quotients}
We identify sets of natural numbers with their characteristic function in $2^\bbN$ and equip $\cP(\bbN)$ with the topology resulting from this identification with the Cantor space. Alternatively, this is the compact topology induced by the metric $d(A,B) = 1/(\min(A \Delta B) +1)$.
An ideal $\cI$ of the Boolean algebra $\cP(\bbN)$ is \emph{Borel} (\emph{analytic}, $F_\sigma$, etc.) if it is Borel (\emph{analytic}, $F_\sigma$, etc.) with respect to this topology.

Every Borel ideal $\cI$ in the Boolean algebra $\cP(\bbN)$ provides us with a textbook example of a Borel quotient structure, $\cP(\bbN)/\cI$.
Reformulating \META{} in this setting leads to the following problem, which in its most general form is still open (more on this in \S\ref{S.Borel_quotients_of_P(N)} and \S\ref{S.FABoole}).

\begin{conjecture}\label{conj:boolean}
Let $\mathcal I$ and $\mathcal J$ be a Borel ideals in $\mathcal P(\bbN)$. Then 
\begin{itemize}
\item $\CH$ implies that if there is an isomorphism between $\mathcal P(\bbN)/\mathcal I$ and $\mathcal P(\bbN)/\mathcal J$, then there is an isomorphism between them which is not topologically trivial;
\item Forcing Axioms imply that all isomorphisms between $\mathcal P(\bbN)/\mathcal I$ and $\mathcal P(\bbN)/\mathcal J$ are topologically trivial.
\end{itemize}
\end{conjecture} 
The following is the proper definition of algebraically trivial automorphisms for Borel quotients of $\cP(\bbN)$.
\begin{definition} \label{def:bool_trivial}
Let $\mathcal I$ and $\mathcal J$ be Borel ideals in $\mathcal P(\bbN)$. A homomorphism 
\[
\Phi\colon\mathcal P(\bbN)/\mathcal I\to\mathcal P(\bbN)/\mathcal J
\] 
is said to be \emph{algebraically trivial} if it has a lifting $\Phi_* \colon\mathcal P(\bbN)\to\mathcal P(\bbN)$ which is a Boolean algebra homomorphism. 
\end{definition}

Rigidity problems for such quotients can be traced back at least to~\cite[p. 38-39]{Er:Scottish}, where Erd\"os and Ulam asked whether the quotients over the two ideals, known as the \emph{asymptotic density zero ideal} $\cZ_0$ and the \emph{logarithmic density zero ideal} $\cZlog$, 
 are isomorphic, where
\begin{align}
\label{eq.Z0} \cZ_0 &= \left\{ A \subseteq \bbN \mid \lim_{n \to \infty} \frac{| A \cap n |}{n} = 0\right\},\\
\label{eq.Zlog} \cZlog &= \left\{ A \subseteq \bbN \mid \lim_{n \to \infty} \frac{\sum_{i \in A \cap n } 1 /i}{\sum_{i < n} 1/i} = 0 \right\}.
\end{align}
The unexpected (at the time) answer to this question will be discussed in \S \ref{S.Borel_quotients_of_P(N)} and \S\ref{S.FABoole}, where we will have something interesting to
say also for the classes of analytic \emph{P-ideals} and of \emph{generalised density ideals}, to which both $\cZ_0$ and $\cZ_{\log}$ belong, and which we briefly anticipate here.

\begin{definition} An ideal $\cI$ is a \emph{P-ideal} if for any countable collection $\{ A_n\}_{n \in \bbN}\subseteq \cI$, for $n\in \bbN$, there exists $A\in \cI$ such that $A_n\setminus A\in \Fin$ for all $n$. 
\end{definition}

Verifying that $\cZ_0$ and $\cZ_{\log}$ are Borel P-ideals is a nice exercise, to which the following definition provides a hint.

\begin{definition}\label{Def.Submeasure}
A \emph{submeasure} on $\bbN$ is a function $\mu\colon \cP(\bbN)\to [0,\infty]$ such that, for every $X,Y \subseteq \bbN$
\begin{enumerate}
\item $\mu(\emptyset)=0$,
\item $\mu(X)<\infty$ when $X$ is finite,
\item $\mu(X)\leq \mu(Y)$ when $X\subseteq Y$
\item $\mu(X\cup Y)\leq \mu(X)+\mu(Y)$.
\end{enumerate} 
 A submeasure is \emph{lower semicontinuous} if $\mu(X)=\lim_{n \to \infty} \mu(X\cap n)$.
\end{definition}

For a lower semicontinuous submeasure $\mu$ on~$\bbN$ let 
\[
\Fin(\mu)=\{X\subseteq \bbN\mid \mu(X)<\infty\},
\]
\[
\Exh(\mu)=\{X\subseteq \bbN\mid \lim_{n \to \infty} \mu(X\setminus n)=0\}. 
\]
The requirement that $\mu$ is finite on finite sets gives that both $\Fin(\mu)$ and $\Exh(\mu)$ contain the ideal $\Fin$. 

\begin{definition}\label{Def.GeneralisedDensity}
Let $\bbN=\bigsqcup_m I_m $ be a partition of $\bbN$ into finite sets and $\mu_m$ be a submeasure concentrated on $I_m$ for all $m$. Consider the measure $\mu(X)=\sup_m \mu_m (X\cap I_m)$. An ideal is a \emph{generalised density ideal} if it is of the form $\Exh(\mu)$ for a $\mu$ arising this way. If each $\mu_n$ is a measure, then this is a \emph{density ideal}. 
\end{definition}

Density ideals were introduced in \cite[Definition~1.13.1]{Fa:AQ}, and it is not difficult to see that both $\cZ_0$ and $\cZ_{\log}$ are density ideals (\cite[Theorem~1.13.3~(a)]{Fa:AQ}; for a fix to part (b) of this theorem see \cite{tryba2021different}). We also have the following results of Mazur (\cite{Maz:F-sigma}) and Solecki (\cite{Sol:Analytic}).

\begin{theorem}\label{T.Solecki} 
An ideal is $F_\sigma$ if and only if it is of the form $\Fin(\mu)$ for a lower semicontinuous submeasure $\mu$. An analytic ideal is a P-ideal if and only if it is of the form $\Exh(\mu)$ for a lower semicontinuous submeasure $\mu$. In this case, 
	\begin{equation*}%\label{eq.dmu}
		\textstyle d_\mu(A,B)=\lim_n \mu((A\Delta B)\setminus n)
	\end{equation*}
	is a complete metric on $\cP(\bbN)/\cI$.
\end{theorem}

(The last line is the easy \cite[Lemma~1.2.3]{Fa:AQ}.) Note that if $\cI$ is not an $F_\sigma$ ideal, then there is a decreasing sequence of $\cI$-positive sets $A_n$ such that $\lim_n \mu(A_n)=0$. This shows a failure of the countable saturation of $\cP(\bbN)/\cI$ as a discrete structure, and we shall see why this is relevant to us in \S\ref{S.Borel_quotients_of_P(N)}.

%A side question asks whether the statement `all isomorphism between Borel quotients of $\mathcal P(\bbN)$ are topologically trivial' is consistent with $\ZFC$; see \S\ref{S.FABoole} for this. 

An important problem in the framework of Borel quotients of $\cP(\bbN)$ is whether topologically trivial homomorphisms are automatically algebraically trivial.
Submeasures from Definition \ref{Def.Submeasure} will be used in \S\ref{S.Ulam} to give partial answers to Question~\ref{Q.UlamAB} below, which in its most general form is still open (see Theorem~\ref{T.idealiso} for partial results).

\begin{question}\label{Q.UlamAB}
Let $\mathcal I$ and $\mathcal J$ be Borel ideals in $\mathcal \cP(\bbN)$ containing $\Fin$. Are all topologically trivial homomorphisms $\Phi\colon\mathcal \cP(\bbN)/\mathcal I\to\mathcal \cP(\bbN)/\mathcal J$ algebraically trivial?
\end{question}

\subsection{\v{C}ech--Stone remainders} \label{ss.remainders}
Given a locally compact space $X$, its \v{C}ech--Stone compactification $\beta X$ is the `largest' compact space in which $X$ sits densely as an open set. $\beta X$ is defined as the space with the universal property that if $K$ is a compact Hausdorff space and $f\colon X\to K$ is continuous, one can find a continuous extension $\beta f\colon\beta X\to K$.\footnote{Functional analyst's definition of $\beta X$: It is the Gelfand spectrum of the \cstar-algebra $C_b(X)$ of all continuous, bounded, complex-valued functions on $X$ (see \S\ref{ss:pre} and \S\ref{ss:mcalg}).}
The compact space $\beta X\setminus X$ is the \emph{\v{C}ech--Stone remainder} (or the \emph{corona}) of~$X$. \v{C}ech--Stone remainders of zero-dimensional spaces and their homeomorphisms from a set theoretic point of view were analysed in~\cite{Fa:AQ} with the aid of Boolean algebraic methods.

We already mentioned how the Stone Duality relates $\cP(\bbN)/\Fin$ to $\beta \bbN \setminus \bbN$.
To study homeomorphisms of \v{C}ech--Stone remainders of other spaces, we need to take a closer look at the duality between homeomorphisms of remainders and homomorphisms of quotient algebras. 

 Fix a locally compact noncompact space $X$. When $X$ is zero-dimensional, autohomeomorphisms of $\beta X\setminus X$ are completely determined by the image of clopen subsets. By identifying clopen sets with their characteristic functions, one endows the Boolean algebra of clopen subsets of $X$, $\Clop(X)$, with the topology of pointwise convergence. If $X$ is in addition second countable, this is a Polish topology. Since the Boolean algebra $\Clop(\beta X\setminus X)$ is isomorphic to the quotient of $\Clop(X)$ by the ideal of compact-open sets, $\Clop_c(X)$, a continuous map from $\beta X\setminus X$ to $\beta Y\setminus Y$ corresponds to a Boolean algebra homomorphism from the Borel quotient structure $\Clop(Y)/\Clop_c(Y)$ to $\Clop(X)/\Clop_c(X)$. As homeomorphisms of \v{C}ech--Stone remainders correspond to isomorphisms of Boolean algebra quotients, we have a notion of topologically trivial homeomorphism (where the induced isomorphism admits a Borel lifting), and one of algebraically trivial (admitting a lifting which is a Boolean algebra homomorphism). 

For non zero-dimensional spaces, one cannot rely on any Boolean algebraic underlying structure, even though the remainder of $\mathbb R$ was studied using topological tools in~\cite{DoHa:Introduction} and~\cite{Ha:Cech-Stone}. To study these objects, we enter the realm of \cstar-algebras. The appropriate versions of \META{} and Question~\ref{Q.UlamAB} will be stated in \S\ref{ss:mcalg}.

\section{The origins of rigidity questions in \cstar-algebras} \label{3a.BDF}
%\input{3.aBDF}
%\section{The Origins of Rigidity Questions in \cstar-algebras} \label{3a.BDF}
\enumthree
The interest in rigidity phenomena in the noncommutative setting began with the study of the existence of outer automorphisms of the Calkin algebra in~\cite{PhWe:Calkin} and~\cite{Fa:All} (see \S\ref{S.Cohomology} and \S\ref{6bii.PFAAVi}), which lead to one of the deepest interplays between set theory and operator algebras to date. The methods developed therein have become a benchmark for the successive investigations on rigidity phenomena of more general massive quotients \cstar-algebras (see \S\ref{6bii.PFAAVi}).

Similar to other categories we discuss in this survey, the rigidity phenomena in the category of \cstar-algebras are deeply influenced by set-theoretic axioms. Nevertheless, the initial motivations for these studies were of purely \cstar-algebraic nature, and they originate from the influential works~\cite{BrDoFi:Unitary} and~\cite{BrDoFi}.
These motivations are discussed in \S\ref{BDF}. In \S\ref{ss:pre} we provide an introduction to the necessary material to carry on with our discussion, and we refer the reader to the textbooks~\cite{Black:Operator} and~\cite{Fa:STCstar} for a more extensive introduction to the subject and all the omitted details. Readers familiar with these topics can safely skip to \S\ref{ss:mcalg} where, in addition to recalling the classical notions of multiplier and corona algebras, we establish a suitable definition of algebraically trivial isomorphism for the category of \cstar-algebras (Definition~\ref{def:trivialmap}).

%\subsection{Operators}
%Throughout this section, let $H $ be the (unique up to isomorphism) complex infinite-dimensional
%Hilbert space generated by a countable orthonormal basis,
%equipped with the scalar product $\langle \xi, \eta \rangle$ and norm $\| \xi \| :=
%\langle \xi, \xi \rangle^{1/2}$.

%Let $\cB(H)$ be the set of all linear operators from $H$ to itself which have finite \emph{operator
%norm}
%\[
%\| T \|:= \sup_{\| \xi \| \le 1} \| T\xi \|.
%\]

%\subsection{\cstar-algebras}
\subsection{Preliminaries} \label{ss:pre}
There are essentially two ways to present \cstar-algebras, an abstract one and a concrete one. We start with the latter.

Let $H$ be a complex Hilbert space, that is, a complex Banach space equipped with a sesquilinear inner product $\langle \cdot , \cdot \rangle$ and a norm given by $\lVert \xi \rVert := \langle \xi, \xi \rangle^{1/2}$. Note that the theory of Hilbert spaces is $\kappa$-categorical (in the context of continuous logic) for every infinite cardinal $\kappa$, indeed every infinite-dimensional Hilbert space $H$ with density character $\kappa$ is isometric to the space
\[
\ell_2(\kappa) := \{ (z_j) \in \bbC^\kappa\mid \textstyle\sum_j | z_j |^2 < \infty \}.
\]
Let $\cB(H)$ denote the set of all linear bounded operators from $H$ to itself. It is equipped with the structure $(\cB(H), \| \cdot \|, +, \{\lambda_z \}_{z \in \bbC},\cdot, ^*)$, where $+$ is the usual sum, $\lambda_z$ is the scalar multiplication by $z$, $\cdot$ is the composition of operators, $^*$ is the adjoint, and the \emph{operator norm} $\| \cdot \|$ for $a \in \cB(H)$ is defined as
\[
\| a \| := \sup_{\xi \in H} \frac{\| a(\xi) \|}{\| \xi \|} = \sup_{\| \xi \| \le 1} \|a(\xi) \|.
\]
A \emph{concrete \cstar-algebra} is a norm-closed complex subalgebra of $\cB(H)$ closed under $^*$.

Alternatively, \cstar-algebras can be defined by providing an axiomatisation in the following way. A \emph{$^*$-Banach algebra} is a tuple $(A, \| \cdot \|, +, \cdot, ^*)$ such that $(A, +, \cdot)$ is a $\bbC$-algebra, the tuple $(A, \| \cdot \|, +)$ is a complex Banach space, $^*$ is a linear involution, and the following conditions hold:
\begin{enumerate}
\item \label{i:cstar1} $(ab)^* = b^* a^*$ for all $a,b \in A$,
\item $\| ab \| \le \| a \| \| b \|$ for all $a,b \in A$,
\item \label{i:cstar3} $\| a^* \| = \| a \|$ for all $a \in A$.
\end{enumerate}
A \emph{(abstract) \cstar-algebra} is a $^*$-Banach algebra satisfying the following additional condition, known as the \emph{\cstar-equality}.
\begin{enumerate} \setcounter{enumi}{3}
\item \label{i:cstareq}$ \| a^*a \| = \| a \|^2$ for all $a \in A$.
\end{enumerate}

 Albeit seemingly mild, the \cstar-equality is a powerful condition that tightly connects the topological nature of \cstar-algebras with their algebraic properties. It implies, for instance, that in a \cstar-algebra the norm is completely determined by its algebraic structure, and therefore it is unique (\cite[Corollary 1.3.3]{Fa:STCstar}). It moreover entails that $^*$-homomorphisms between \cstar-algebras, namely linear multiplicative maps preserving the involution $^*$ (the relevant arrows in this category), are automatically continuous (\cite[Lemma 1.2.10]{Fa:STCstar}).

By the Gelfand--Naimark--Segal (GNS) construction (see~\cite[\S1.10]{Fa:STCstar}) every \cstar-algebra $A$ is isometrically isomorphic to a concrete \cstar-algebra in some $\cB(H)$. Equivalently, the GNS-construction shows that concrete \cstar-algebras are axiomatised by the conditions listed above, allowing us to identify the two definitions.

%%%%%%%%

\subsubsection{Examples I: Abelian \cstar-algebras} %\label{Ex.Abelian}

One of the cornerstones in the theory of \cstar-algebras is the Gelfand Representation Theorem (\cite[Theorem 1.3.1]{Fa:STCstar}), which states that every abelian \cstar-algebra is isomorphic to one of the form $C_0(X)$, for some locally compact $X$, where
\[
C_0(X)=\{f\colon X\to\mathbb C\mid f\text{ is continuous and vanishes at infinity}\}
\]
is considered with the supremum norm and pointwise operations. The Gelfand Representation Theorem establishes a natural equivalence between the category of unital abelian \cstar-algebras and the category of compact Hausdorff spaces (e.g. \cite[Theorem~1.3.2]{Fa:STCstar}). This duality, usually referred to as the \emph{Gelfand--Naimark Duality}, motivates the leading philosophy in the subject, namely that \cstar-algebras provide the noncommutative analogue of topological spaces. Indeed, many definitions, techniques, and results in \cstar-algebras arise as adaptations of ideas that stem from topology. %This connection is particularly significant for the study of rigidity phenomena as well. As a matter of fact, the methods developed to explore these topics in the context \cstar-algebras often originate from previous works on \v{C}ech--Stone remainders of topological spaces or on quotient boolean algebras, which are themselves connected to topology (and thus to \cstar-algebras) via the Stone Duality (see~\cite[Chapter~9]{Fa:STCstar}).

In the special case where $X$ is moreover compact zero-dimensional, the Gelfand--Naimark Duality can be combined with the Stone Duality, allowing to interpret Boolean algebras as a particular instance of abelian \cstar-algebras, as hinted at the beginning of \S\ref{S.Abel}.

If $X$ is a locally compact space, consider the \cstar-algebras
\[
C_b(X)=\{f\colon X\to\mathbb C\mid f\text{ is continuous and bounded}\}
\] 
 with the sup norm and pointwise operations. Then, $C_0(X)$ is a norm-closed ideal in $C_b(X)$; furthermore, $C_b(X)$ is isomorphic to $C(\beta X)$ and the \cstar-algebra quotient $C_b(X)/C_0(X)$ is isomorphic to $C(\beta X\setminus X)$.

\subsubsection{Examples II: General \cstar-algebras}
The following are some of the fundamental examples of noncommutative \cstar-algebras. 

\begin{enumerate}[resume]
\item The *-algebra $\cB(H)$ of all bounded operators on $H$ is a \cstar-algebra. In the case when $H$ has a finite dimension equal to $n$, it is the \cstar-algebra of all $n\times n$ matrices with complex entries $M_n(\bbC)$.
\item The set of all operators in $\cB(H)$ that have finite-dimensional range is a $\bbC$-subalgebra of $\cB(H)$ closed under $^*$. Its norm closure is the set of all \emph{compact operators}, denoted by $\cK(H)$, a \cstar-algebra that is moreover a two-sided ideal in $\cB(H)$.
\item Given a separable, infinite-dimensional Hilbert space $H$, the \emph{Calkin algebra} is the quotient $\cQ(H) := \cB(H) / \cK(H)$. Being the quotient of a \cstar-algebra by a two-sided norm-closed ideal, it is itself a \cstar-algebra (this is nontrivial---see~\cite[Lemma 2.5.2]{Fa:STCstar}) with the operations and norm induced by $\cB(H)$.
\end{enumerate}
\cstar-algebras are preserved by standard operations such as direct sums and products (with pointwise operations and the supremum norm), quotients by two-sided closed ideals, and inductive limits. Many important classes of \cstar-algebras are constructed this way. For instance, \emph{uniformly hyperfinite} (UHF) algebras are unital inductive limits of matrix algebras, and \emph{approximately finite} (AF) algebras are inductive limits of finite-dimensional \cstar-algebras.

Given the aforementioned correspondence between \cstar-algebras and algebras of operators on Hilbert spaces, much of the standard terminology used to denote classes of operators in $\cB(H)$ is extended to abstract \cstar-algebras. In particular, given an element $a$ in a \cstar-algebra $A$ we say that
\begin{enumerate}[resume]
\item $a$ is \emph{normal} if $a^*a = aa^*$,
\item $a$ is \emph{self-adjoint} if $a^* = a$,
\item $a$ is \emph{positive} if $a = b^*b$ for some $b \in A$,
\item if $A$ is unital, $a$ is a \emph{unitary} if $a^*a = aa^* = 1$,
\item if $A$ is unital, $\sigma(a) := \{ \lambda \in \bbC\mid a- \lambda 1 \text{ is invertible} \}$,\footnote{The spectrum of a normal element $a$ of a non-unital \cstar-algebra $A$ is defined as the spectrum of $a$ in the unitisation of $A$.}
is the \emph{spectrum} of $a$,
\item $a$ is a \emph{projection} if $a = a^* = a^2$, 
\item \label{3.isometry} $a$ is an \emph{isometry} if $a^*a=1$, 
\item \label{3.partial.isometry} $a$ is a \emph{partial isometry} if $a^*a$ is a projection (equivalently, if both $a^*a$ and $aa^*$ are projections). 
\end{enumerate}

%Other fundamental examples of \cstar-algebras come from topological spaces and are defined as
%follows.
%Given a locally compact Hausdorff space $X$, let $C_0(X)$ be the set of all continuous functions
%from $X$ in $\bbC$ which vanish at infinity,
%equipped with scalar multiplication, pointwise addition and multiplication,
%the uniform norm $\| \cdot \|_\infty$ and the involution defined as $f^*(x) := \overline{f(x)}$.
%This structure makes $C_0(X)$ an abelian \cstar-algebra (i.e. the multiplication
%is a commutative operation) which has a unit for the multiplication
%if and only if $X$ is compact. One can also see $C_0(X)$ as a concrete
%\cstar-algebras on the Hilbert space $K := L^2(X, \mu)$, where $\mu$ is some positive Radon measure on $X$,
%by identifying each $f \in C_0(X)$ with the operator $M_f$ acting as $M_f(g) = fg$ for every $g \in K$.

\subsection{Multiplier and corona algebras} \label{ss:mcalg}
Multiplier algebras provide an abstract generalisation of $\cB(H)$ and a noncommutative analogue of the \v{C}ech--Stone compactification, while corona algebras similarly generalise the Calkin algebra and \v{C}ech--Stone remainders of topological spaces (Example~\ref{ex:coronas}).

An ideal $J$ in a \cstar-algebra $A$ is \emph{essential} if its annihilator, 
\[
J^\perp=\{a\in A\mid ab=ba=0\text{ for all }b\in J\}
\]
is trivial. For example, $C_0(X)$ is an essential ideal in $C_b(X)$, and a direct summand is an inessential ideal in a direct sum of (nontrivial) \cstar-algebras. 

There are many ways in which a non-unital \cstar-algebra $A$ embeds as an essential ideal into a unital one. In the abelian case of $A = C_0(X)$, by the Gelfand--Naimark duality these correspond to the ways a locally compact Hausdorff space $X$ can be embedded as a dense open set in a compact Hausdorff space. Among these, the \v{C}ech--Stone compactification $\beta X$ is the maximal compactification of $X$, in the sense that every other compactification $\gamma X$ is the continuous image of $\beta X$ via the continuous extension of the identity map on $X$. %Its \cstar-algebraic analogue is the multiplier algebra.

Given a \cstar-algebra $A$, its \emph{multiplier algebra} $\cM(A)$ is the unital \cstar-algebra such that when a unital \cstar-algebra $B$ contains $A$ as an essential ideal, then the identity map on $A$ extends uniquely to a $^*$-homomorphism from $B$ to $\cM(A)$ (\cite[II.7.3.1]{Black:Operator}).
The correspondence with the \v{C}ech--Stone compactification is due to the fact that when $A \cong C_0(X)$ then $\cM(C_0(X)) \cong C(\beta X)$. We refer to~\cite[II.7.3]{Black:Operator} and~\cite[\S 13]{Fa:STCstar} for a rigorous presentation and equivalent definitions of $\cM(A)$. The following one is illuminating (\cite[Corollary~13.2.2]{Fa:STCstar}).

\begin{lemma} \label{L.Idealiser}
For a non-unital \cstar-algebra $A$, and any nondegenerate\footnote{A representation is nondegenerate if $\Phi[A]^\perp=\{0\}$.} faithful representation $\Phi\colon A\to \cB(H)$ there is an isomorphism between $\cM(A)$ and the idealiser 
\[
\{b\in \cB(H)\mid b\Phi[A]\cup \Phi[A]b\subseteq \Phi[A]\}. 
\]
\end{lemma}

The multiplier algebra of a non-unital \cstar-algebra is always nonseparable in the norm topology.
It is however possible to make up for this for separable \cstar-algebras by considering the strict topology.
\begin{definition}
Let $A$ be a non-unital \cstar-algebra. The \emph{strict topology} on $\cM(A)$ is the topology
generated by the seminorms
\[
\lambda_a(b) := \lVert ab \rVert \text{ and } \rho_a(b) := \lVert ba \rVert, \ \forall a \in A, b \in \cM(A).
\]
\end{definition}

The multiplier algebra $\cM(A)$ is the completion of $A$ with respect to the uniformity associated with the strict topology
(\cite[Lemma 13.1.5]{Fa:STCstar}).
Moreover, if $A$ is separable, the strict topology is Polish on the unit ball of $\cM(A)$, giving it the structure of a standard Borel space (\cite[Lemma 13.1.8]{Fa:STCstar}). If $X$ is a locally compact second countable topological space, this is the topology on $C(\beta X)$ given by uniform convergence on compact subsets of $X$.

\begin{definition} \label{def:corona}
The \emph{corona} of a non-unital \cstar-algebra $A$ is the quotient $\cQ(A) := \cM(A) /A$.
We let $\pi_A \colon \cM(A) \to \cQ(A)$ be the quotient map.
\end{definition}

Corona algebras constitute the main family of noncommutative quotient structures considered in this survey. We list some examples that are discussed in the following sections.
 
\begin{example} \label{ex:coronas}
\enumone
\begin{enumerate}
\item Coronas are the noncommutative analogues of \v{C}ech--Stone remainders, since $\cM(C_0(X)) \cong C(\beta X)$, and thus $Q(C_0(X)) \cong C(\beta X \setminus X)$.
\item \label{ex:calkin} Given a Hilbert space $H$, the multiplier algebra of the algebra of compact operators $\cK(H)$ is $\cB(H)$. Hence, when $H$ is separable and infinite dimensional, the corona of $\cK(H)$ is the Calkin algebra $\cQ(H)$. We sometimes use the notation $\dot a$ to denote the class $\pi(a) \in \cQ(H)$ of an operator $a \in \cB(H)$.
\item \label{ex.reduced} Given a countable family of unital \cstar-algebras $(A_n)_{n \in \bbN}$, the direct sum $\bigoplus_n A_n$ is the set of all sequences $(a_n)_{n \in \bbN}$ such that $a_n \in A_n$ and $\lim_{n \to \infty} \| a_n \| = 0$. Its multiplier algebra is the direct product $\prod_n A_n$, the set of all bounded sequences whose $n$-th entry is in $A_n$. Its corona, $\prod_n A_n /\bigoplus_n A_n$, is referred to as the \emph{reduced product of $(A_{n})$ over the Fr\'echet ideal $\Fin$}, often denoted $\prod_n A_n /\bigoplus_{\Fin} A_n$.
\item\label{ex.reduced over I} Given a countable family of unital \cstar-algebras $(A_n)_{n \in \bbN}$, and an ideal $\mathcal I$ on $\bbN$, let
\[
\textstyle \bigoplus_{\mathcal I}A_{n}= \{(a_{n})\in \prod_{n}A_{n} \mid \lim_{n\rightarrow \mathcal I^+}\|a_{n}\|=0\},
\]
where
\[
\lim_{n \to \cI} \| a_n \| = \inf_{X\in \cI} \sup_{n\notin X} \|a_n \|.
 \]
Then $\bigoplus_{\mathcal{I}}A_{n}$ is a closed ideal of $\prod_{n}A_{n}$ and its multiplier algebra is again $\prod_{n}A_{n}$. The corona algebra $\prod_{n}A_{n}/\bigoplus_{\mathcal I}A_{n}$ is usually called the \emph{reduced product of $(A_{n})$ over the ideal $\mathcal I$}. 
\end{enumerate}
\end{example}

As mentioned before, if $A$ is separable the unit ball of $\mathcal M(A)$ is equipped with a natural Polish topology (the strict topology) containing the unit ball of $A$ as a Borel subset where all considered operations are Borel. This endows the unit ball of the corona of $A$ with a Borel quotient structure as in Definition~\ref{Def.Borel}. 
\begin{definition}\label{def:toptrivgeneral}
Let $A$ and $B$ be separable \cstar-algebra. A homomorphism $\Phi\colon\mathcal Q(A)\to\mathcal Q(B)$ is topologically trivial if there is a map $\Phi_*\colon\cM(A)\to\cM(B)$ that is Borel when restricted to the unit ball of $\mathcal M(A)$ considered with the strict topology, and makes the following diagram commute.
\[
\begin{tikzpicture}
 \matrix[row sep=1cm,column sep=1.5cm] 
 {
& & \node (M1) {$\cM(A)$}; & \node (M2) {$\cM(B)$};&
\\
& & \node (Q1) {$\cQ(A)$}; & \node (Q2) {$\cQ(B)$} ;
\\
};
\draw (M1) edge [->] node [above] {$\Phi_*$} (M2);
\draw (Q1) edge [->] node [above] {$\Phi$} (Q2);
\draw (M1) edge [->] node [left] {$\pi_A$} (Q1);
\draw (M2) edge [->] node [left] {$\pi_B$} (Q2);
\end{tikzpicture}
\]
\end{definition}

The following, stated in~\cite{CoFa:Automorphisms}, is an instance of \META{} in the current setting.

\begin{conjecture}\label{conj:CFgeneral}
Let $A$ and $B$ be separable non-unital \cstar-algebras. Then
\begin{itemize}
\item $\CH$ implies that if there is an isomorphism between $\mathcal Q(A)$ and $\mathcal Q(B)$ then there is an isomorphism between them which is not topologically trivial;
\item Forcing Axioms imply that all isomorphisms between $\mathcal Q(A)$ and $\mathcal Q(B)$ are topologically trivial.
\end{itemize}
\end{conjecture}

\subsubsection{Algebraically trivial automorphisms}
We contextualise the discussion in \S\ref{S.Intro} to the category of \cstar-algebras, aiming to introduce a suitable definition of algebraically trivial isomorphism (Definition~\ref{Def.AlgebraicallyTrivial}) for corona algebras. A natural notion of triviality for automorphisms of \cstar-algebras is given by their innerness.

\begin{definition} \label{def:inner}
Let $A$ be a unital \cstar-algebra. Every unitary $u \in A$ induces an automorphism $\Ad(u)$ on $A$, defined as $\Ad(u)(a) = uau^*$ for every $a \in A$. An automorphism of this form is said to be \emph{inner}. An automorphism that is not inner is called \emph{outer}.
\end{definition}

%\subsubsection{Algebraically trivial isomorphisms}

Inner automorphisms of coronas of separable \cstar-algebras (considered with the Borel structure induced by the strict topology) are topologically trivial. Adopting the notion of inner automorphisms as algebraically trivial ones is however futile in the abelian setting since the only inner automorphism of an abelian \cstar-algebra is the identity. The definition of algebraically trivial is thus given by generalising the concept of `almost permutation' (appearing in ~\cite{FaSh:Rigidity}) as follows.

\begin{definition}\label{defin:trivialtop}
Let $X$ and $Y$ be a second countable locally compact noncompact spaces. A homeomorphism $\Phi\colon\beta X\setminus X\to\beta Y\setminus Y$ is said to be \emph{algebraically trivial} if there are open sets with compact closure $K_X\subseteq X$ and $K_Y\subseteq Y$ and a homeomorphism $\Phi_*\colon X\setminus K_X\to Y\setminus K_Y$ such that $\beta\Phi_*=\Phi$ on $\beta X\setminus X$. 
\end{definition}

The following definition provides a suitable generalisation of Definition~\ref{defin:trivialtop} to all \cstar-algebras.

\begin{definition}[{\cite[Definition 2.3]{vignati2018rigidity}}] \label{def:trivialmap}
An isomorphism $\Phi \colon \cQ(A) \to \cQ(B)$ is \emph{algebraically trivial} if there exist positive elements
$a\in \cM(A)$, $b \in \cM(B)$ and an isomorphism $\Phi_* \colon \overline{aAa} \to \overline{bBb}$ such that
\enumone 
\begin{enumerate}
\item $1-a \in A$, $1 - b \in B$,
\item $\Phi_*$ extends to a strictly continuous $^*$-homomorphism $\bar \Phi_* \colon \overline{a \cM(A) a} \to \overline{b \cM(B) b}$ such that the following diagram commutes.
\[
\begin{tikzpicture}
 \matrix[row sep=1cm,column sep=1.5cm] 
 {
& & \node (M1) {$\overline{a\cM(A)a}$}; & \node (M2) {$\overline{b\cM(B)b}$};&
\\
& & \node (Q1) {$\cQ(A)$}; & \node (Q2) {$\cQ(B)$} ;
\\
};
\draw (M1) edge [->] node [above] {$\bar\Phi_*$} (M2);
\draw (Q1) edge [->] node [above] {$\Phi$} (Q2);
\draw (M1) edge [->] node [left] {$\pi_A$} (Q1);
\draw (M2) edge [->] node [left] {$\pi_B$} (Q2);
\end{tikzpicture}
\]
\end{enumerate}
\end{definition}

In many examples, the notion of algebraically trivial automorphism corresponds to the automorphisms of coronas that are constructed by hand. For abelian \cstar-algebras, an algebraically trivial isomorphism between $C(\beta X\setminus X)$ and $C(\beta Y\setminus Y)$ is dual to an algebraically trivial homeomorphism as in Definition~\ref{defin:trivialtop} (\cite[Proposition 2.7]{vignati2018rigidity}). 
Another natural description of algebraically trivial isomorphisms can be made for reduced products of unital separable \cstar-algebras that do not have central projections (see~\cite[Proposition 5.3]{vignati2018rigidity}). In case of the Calkin algebra, algebraic triviality coincides with innerness.

\begin{theorem}[{\cite[Lemma 7.2]{CoFa:Automorphisms}}] \label{thm:autocalk}
An automorphism of the Calkin algebra is inner if and only if it is algebraically trivial.
\end{theorem}

Algebraically trivial isomorphisms are topologically trivial if $A$ and $B$ are separable (\cite[Theorem 2.5]{vignati2018rigidity}). The converse is not known (at least not unconditionally on $A$ and $B$), and it is tightly related to Ulam-stability phenomena. The following is the absolute (see \S\ref{S.Absoluteness}) portion of~\cite[Conjecture 5.1]{vignati2018rigidity}, partially answered in Theorem~\ref{thm:UlamAB} below. 

\begin{question}\label{ques:ulamgeneral}
Let $A$ and $B$ be separable \cstar-algebras. Are all topologically trivial isomorphisms between $\mathcal Q(A)$ and $\mathcal Q(B)$ algebraically trivial?
\end{question}

\subsection{The BDF Question} \label{BDF}
The interest in the existence of outer automorphisms of the Calkin algebra was sparked in the groundbreaking paper~\cite{BrDoFi}, which was the apex of a long series of efforts aiming to classify elements of $\cB(H)$ up to essential unitary equivalence (we refer to~\cite{david:ess} for an excellent survey on the subject).
\begin{definition}
Two operators $a, b \in \cB(H)$ are \emph{essentially unitarily equivalent} (or \emph{unitarily equivalent up to a compact perturbation}) if there exists a unitary $u \in \cB(H)$ such that $uau^* - b \in \cK(H)$. We abbreviate this relation with the notation $a \sim_{\cK(H)} b$.
\end{definition}

The search for an invariant able to completely classify elements of $\cB(H)$ up to essential unitary equivalence goes back to the early investigations on operator algebras by Weyl and von Neumann, where the relation $\sim_{\cK(H)}$ naturally came out in the study of the stability of pseudodifferential operators under compact perturbations. In 1909 Weyl proved in~\cite{von1909beschrankte} that every self-adjoint operator $a$ on a separable Hilbert space can be written as $a = d + c$, where $d$ is a diagonalisable operator and $c \in \cK(H)$. Weyl also observed that if $a \sim_{\cK(H)} b$ then the essential spectra of $a$ and $b$ are the same. Given an operator $a \in \cB(H)$, its \emph{essential spectrum} $\sigma_e(a)$ is the subset of $\bbC$ consisting of all the elements in the spectrum of $a$ which are either eigenvalues of infinite multiplicity or accumulation points. The set $\sigma_e(a)$ is known to be equal to $\sigma(\dot a)$, the spectrum of $\dot a$ in the Calkin algebra (\cite[Proposition~2.2.2]{higson2000analytic}).

More than 20 years after Weyl's work, in~\cite{von1935charakterisierung} von Neumann proved the converse, and much harder, implication: if two self-adjoint operators $a$ and $b$ have the same essential spectrum, then $a \sim_{\cK(H)} b$. %This shows that the essential spectrum is a complete invariant for the relation of essential unitary equivalence on self-adjoint operators.

In Problem 4 of his \emph{Ten Problems in Hilbert Space}~\cite{halmos10p}, Halmos asked whether every normal operator is the sum of a diagonalisable operator and a compact one, pushing towards an extension of Weyl's studies. A positive answer to his question was found shortly after by Berg~\cite{berg} and Sikonia~\cite{sikonia}, who independently generalised Weyl and von Neumann's results to normal operators. This collective effort is nowadays summarised by the so called Weyl--von Neumann--Berg--Sikonia Theorem. 
\begin{theorem}[Weyl--von Neumann--Berg--Sikonia Theorem]
Any two normal operators in $\cB(H)$ are unitarily equivalent up to compact perturbation if and only if their essential spectra are equal.
\end{theorem}

This latter characterisation also applies to all compact perturbations of normal operators in $\cB(H)$.
%Furthermore, the Calkin algebra provides a natural framework for this topic, since classifying normal operators in $\cB(H)$ up to essential unitary equivalence is the same as classifying their classes in the Calkin algebra up to unitary equivalence\footnote{In the Calkin algebra it is possible to define two notions of unitary equivalence. Two elements $t_1, t_2 \in \cQ(H)$ are \emph{strongly unitarily equivalent} if there is a unitary $U \in \cB(H)$ such that $q(U) t_1 q(U)^* = t_2$, while they are called \emph{weakly unitarily equivalent} if there is a unitary $v \in \cQ(H)$ such that $vt_1v^* = t_2$. We will not stress too much this distinction since on normal elements of $\cQ(H)$ the two relations are the same.}. 
In~\cite{BrDoFi:Unitary} Brown, Douglas and Fillmore further extended the scope of this classification to all \emph{essentially normal operators}. These are all the elements $a \in \cB(H)$ such that $aa^*- a^*a \in \cK(H)$ or, equivalently, such that $\dot a$ is normal in $\cQ(H)$.

Note that the family of essentially normal operators is strictly larger than the set of compact perturbations of normal operators. The most notorious witness of this fact (and possibly {`the most important single operator'} in $\cB(H)$,~\cite[Example I.2.4.3(ii)]{Black:Operator}) is the unilateral shift~$s$.
\begin{definition} \label{def:shift} 
Given an orthonormal basis $\{\xi_n\}_{n \in \bbN}$ of $H$, the unilateral shift $s \in \cB(H)$ is defined on the basis by $s \xi_n = \xi_{n+1}$ for every~$n \in \bbN$.
\end{definition}
The operator $s^*s - ss^*$ is a projection whose range is 1-dimensional, hence compact. Therefore $s$ is essentially normal, in particular $\dot s$ is a unitary in $\cQ(H)$. We remark that, although the definition of the unilateral shift depends on the choice of a basis, it is unique up to unitary equivalence.

The Fredholm index can be used to check that $s$ is not a compact perturbation of a normal operator. The \emph{Fredholm index} of an operator $a \in \cB(H)$ is defined as
\[
\ind(a) := \dim(\Ker(a)) - \dim(\Ker(a^*)),
\]
whenever the two quantities on the right-hand side are finite, in which case we say that $a$ is \emph{Fredholm}. Atkinson's Theorem (\cite[Theorem I.8.3.6]{Black:Operator}) states that an operator $a$ is Fredholm if and only if $\dot a$ is invertible in $\cQ(H)$. Hence the unilateral shift $s$ is Fredholm.

The Fredholm index is an invariant for the relation of essential unitary equivalence,
%Indeed, if $T\in \cB(H)$ is a Fredholm operator, then by Atkinson's Theorem it is easy to see that both $UTU^*$ and $T + K$ are Fredholm operators whenever $U \in \cU(H)$ and $K \in \cK(H)$. Moreover, it can be proved that
%\begin{equation} \label{eq:ind} 
%\ind(T) = \ind(UTU^*)= \ind(T + K).
%\end{equation}
and direct computations show that it is always zero on Fredholm normal operators. Since $\ind(s) = -1$, it follows that $s$ cannot be a compact perturbation of a normal operator. Furthermore, the Fredholm index shows that the essential spectrum alone is not sufficiently fine to completely classify essentially normal operators up to essential unitary equivalence. Indeed there exist unitaries in $\cB(H)$ whose essential spectrum is equal to $\sigma_e(s) = \bbT$ (one example is the bilateral shift, which is defined as a shift on some orthonormal basis of $H$ indexed over $\bbZ$), but none of these operators is essentially unitarily equivalent to $s$, as their Fredholm index is equal to 0.

Brown, Douglas and Fillmore gave a full characterisation of unitary equivalence up to compact perturbation on essentially normal operators in~\cite{BrDoFi:Unitary}
%(extendedin~\cite{BrDoFi} to a more general analysis of extensions of abelian, unital, separable \cstar-algebras),
by adding the Fredholm index to the invariant.
\begin{theorem}[{\cite[Theorem 11.1]{BrDoFi:Unitary}}]
Given two essentially normal operators $a,b \in \cB(H)$, then $a \sim_{\cK(H)} b$ holds if and only if $\sigma_e(a) = \sigma_e(b)$ and $\ind(a - \lambda \id) = \ind(b - \lambda \id)$ for all $\lambda \notin \sigma_e(a)$.
\end{theorem}

If moreover $a$ and $b$ are such that $\dot a$ and $\dot b$ are unitaries in $\cQ(H)$, then the classification is even cleaner, as the two operators are equivalent if and only if $\sigma_e(a) = \sigma_e(b)$ and $\ind(a) = \ind(b)$ (\cite[Theorem 3.1]{BrDoFi:Unitary}).

At the very end of~\cite{BrDoFi:Unitary} the authors introduce a weakening of essential unitary equivalence. Given $a,b \in \cB(H)$ we say that $a \approx b$ if and only if there exists an automorphism $\alpha \in \Aut(\cQ(H))$ such that $\alpha(\dot a) = \dot b$. When restricting to essentially normal operators, the relation $\sim_{\cK(H)}$ is obtained from $\approx$ by requiring that $\alpha$ is inner\footnote{In the Calkin algebra it is possible to define two notions of unitary equivalence. Two elements $a,b \in \cQ(H)$ are \emph{weakly unitarily equivalent} if there is a unitary $v \in \cQ(H)$ such that $\Ad(v) a = b$, and they are \emph{(strongly) unitarily equivalent} if furthermore $v$ can be chosen to be equal to $\dot u$, for some unitary $u \in \cB(H)$. We will not stress too much this distinction since for normal elements these two relations coincide.}. In~\cite{BrDoFi:Unitary} it is asked whether $\approx$ and $\sim_{\cK(H)}$ coincide on such operators. Note that $a \approx b$ implies $\sigma_e(a) = \sigma_e(b)$, thus the two relations are different only if there is an automorphism of $\cQ(H)$ which moves an element of $\cQ(H)$ to another one with a different Fredholm index. Since a unitary of Fredholm index $n\neq 0$ has an $m$-th root if and only if $m$ divides $n$, every automorphism of $\cQ(H)$ induces a group automorphism of $\bbZ$ (the range of the Fredholm indices) and this question boils down to the following.

\begin{question} \label{ques:bdf}
Is there an automorphism of the Calkin algebra that moves the unilateral shift $\dot s$ to its adjoint $\dot s^*$?
\end{question}

Since inner automorphisms preserve the Fredholm index, an automorphism witnessing a positive answer to Question~\ref{ques:bdf} cannot be inner. As a test question, in~\cite[\S 1.6(ii)]{BrDoFi} the authors asked whether outer automorphisms of the Calkin algebra exist at all.

It is unlikely that Brown, Douglas and Fillmore expected that the answer to their question might involve set theory, although a few years after Shelah started a subject that would eventually lead to the resolution of the second problem (\S\ref{S.Shelah}). The question of the existence of outer automorphisms of the Calkin algebra will be thoroughly discussed in \S\ref{S.Cohomology} and \S\ref{6bii.PFAAVi}. 

We remark that it is still not known whether Question~\ref{ques:bdf} has a positive answer consistently with $\ZFC$. Indeed, the restriction of every known outer automorphism of $\cQ(H)$ to a separable \cstar-subalgebra is implemented by a unitary, and therefore preserves the Fredholm index.

The methods developed by Brown, Douglas and Fillmore in~\cite{BrDoFi:Unitary} had much deeper implications than simply classifying normal operators in $\cQ(H)$ up to unitary equivalence. Their work led to the introduction of
a new \cstar-algebraic invariant $\Ext$, which is defined as follows
(we refer to~\cite[Chapter VII]{blackadar1998k} for a textbook on the topic).

\begin{definition} \label{def:ext}
A \emph{unital extension} of a given unital \cstar-algebra $A$ is a unital injective $^*$-homomorphism $\Phi\colon A \to \cQ(H)$. Two extensions $\Phi_1, \Phi_2$ are \emph{unitarily equivalent} if there is a unitary $u \in \cB(H)$ such that $\Ad(\dot u) \circ \Phi_1 = \Phi_2$. We define $\Ext(A)$ as the set of all unital extensions of $A$ modulo unitary equivalence.
Exploiting the fact that $\cQ(H) \oplus \cQ(H) \subseteq M_2(\cQ(H)) \cong \cQ(H)$, one defines the direct sum of two extensions, unique up to unitary equivalence, thus endowing $\Ext(A)$ with a structure of an abelian semigroup.
\end{definition}

With the language of extensions, the classification problem we are discussing can be investigated from a much broader perspective. Note first that, given an essentially normal element $a \in \cB(H)$, by the Gelfand--Naimark Duality the unital \cstar-algebra generated by $\dot a$ in $\cQ(H)$ is isomorphic to the abelian \cstar-algebra $C(\sigma_e(a))$. The isomorphism sends the identity map on $\sigma_e(a)$ to $\dot a$, and can be thought as a unital extension of $C(\sigma_e(a))$. Given two essentially normal operators $a,b \in \cB(H)$ such that $X := \sigma_e(a) = \sigma_e(b)$, asking whether $a \sim_{\cK(H)} b$ translates into verifying that the two unital extensions of $C(X)$ by $\dot a$ and $\dot b$ defined above are unitarily equivalent. 
%\footnote{In the Calkin algebra it is possible to define another notion of unitary equivalence between extensions. Two unitalextensions $\varphi_1, \varphi_2$ of a unital \cstar-algebra $A$ are \emph{weakly unitarily equivalent} if there is a unitary $v \in \cQ(H)$ such that $\Ad(v) \circ \varphi_1 = \varphi_2$. This relation gives rise to a different invariant, $\Ext^w$, but we will not stress too much this distinction since for commutative \cstar-algebras (and thus for normal elements of $\cQ(H)$), unitary equivalence and weak unitary equivalence are equal.}.
Hence, the classification of essentially normal operators up to $\sim_{\cK(H)}$ is reduced to the study of $\Ext(C(X))$ (usually simply denoted by $\Ext(X)$), in the particular case where $X$ is a compact subset of $\mathbb{C}$.

In ~\cite{BrDoFi} the authors extend their study to $\Ext(X)$ for $X$ a general compact metrisable space, which by the Gelfand--Naimark Duality corresponds to the analysis of $\Ext$ for separable, abelian, unital \cstar-algebras. They prove that $\Ext(X)$ is always an abelian group, and that it in this setting $\Ext$ is a homotopy invariant, covariant functor, which is moreover a $K$-homology. Their work effectively led to the birth of a brand new subject combining \cstar-algebras and algebraic topology, known as extension theory (standard references for this subject are~\cite{blackadar1998k} and \cite{higson2000analytic}).
%Among the other foundational works in this field, we recall Voiculescu's striking results in~\cite{voiculescu1976non}, where it is proved that the class of \emph{trivial extensions} of a unital separable \cstar-algebra $A$, namely those extension that lift to unital $^*$-homomorphisms to $\cB(H)$, always represent the zero element of the semigroup $\Ext(A)$. In \cite{choi1976completely} Choi and Effros show that $\Ext(A)$ is a group whenever $A$ is a separable, amenable \cstar-algebra. The fact that there exist separable \cstar-algebras $A$ such that $\Ext(A)$ is not a group was obtained in~\cite{anderson:ext}, and later in~\cite{haagerup2005new} it was shown that the $\Ext$ of the reduced \cstar-algebra of the free group on two generators $C^*_{\rm{red}}(\mathbb{F}_2)$ is not a group.

Note that the group $\Aut(\cQ(H))$ acts by composition on $\Ext(A)$, for every unital separable \cstar-algebra $A$. From this perspective, Question~\ref{ques:bdf} interrogates about the triviality of the action of $\Aut(\cQ(H))$ on $\Ext(\bbT)$, since $\sigma_e(s) =\sigma_e(s^*) = \bbT$. More specifically, the function that maps each (class of an) extension $\varphi \in \Ext(\bbT)$ to the Fredholm index $\ind(\varphi(\id_\bbT))$, is an isomorphism between $\Ext(\bbT)$ and $\bbZ$. The extensions of $C(\bbT)$ which are mapped by this isomorphism to $-1$ and $1$ are exactly those that send $\id_\bbT$ to (something which is unitarily equivalent to) $\dot s$ and $\dot s^*$, respectively. Since the only nontrivial automorphism of the group $\bbZ$ is the one moving $1$ to $-1$, understanding whether the action of $\Aut(\cQ(H))$ on $\Ext(\bbT)$ is trivial reduces to answering Question~\ref{ques:bdf}. We conclude this section with a more general problem suggested by Nigel Higson.

\begin{question} \label{ques:ext}
Are there unital, separable \cstar-algebras $A$ such that, consistently with ZFC, $\Aut(\cQ(H))$ acts nontrivially on $\Ext(A)$?
\end{question}

In the case when $A = C(\bbT)$, Higson's question reduces to Question~\ref{ques:bdf}. In general, every inner automorphism acts trivially on $\Ext(A)$, and both questions are $\Sigma^2_1$ statements hence Woodin's theorem (see \S\ref{S.Absoluteness}) suggests that CH provides an optimal setting for a possible positive answer to either of the two questions.

\subsection{Sakai's Question}\label{S.Sakai}

In this section we introduce an old open problem which originated from a series of papers due to Sakai (\cite{sakai1968derivationsI},~\cite{sakai1971derivationsII} and~\cite{sakai1971derivations}), and which has been abundantly overlooked over the last decades. Similarly to some other topics discussed in the current section, this question has a purely \cstar-algebraic flavour and was motivated by the investigations on inner derivations of simple \cstar-algebras contained in Sakai's aforementioned works.

\begin{question} \label{ques:sakai}
Let $A$ and $B$ be two separable, simple non-unital \cstar-algebras, and suppose that $\cQ(A) \cong \cQ(B)$. Are $A$ and $B$ isomorphic?
\end{question}

Both simplicity and separability are necessary in Question~\ref{ques:sakai}. The former is due to the elementary example $B = A \oplus C$, where $A$ is any non-unital \cstar-algebra and $C$ is any unital \cstar-algebra, which gives $\cQ(A) \cong \cQ(B)$. 
 %More sophisticated non-simple examples arise under CH, as witnessed for instance by Theorem~\ref{thm:eeiso} or by Parovi\v{c}enko's Theorem (\cite{Pa:Universal}). The latter states that, in presence of CH, if $X$ is a zero-dimensional locally compact second countable space, $\beta X\setminus X$ is homeomorphic to $\beta\bbN\setminus\bbN$, which entails $C(\beta X\setminus X) \cong C(\beta\bbN\setminus\bbN)$.

For what concerns separability, in~\cite{sakai1971derivations} Sakai shows that there exist multiple nonseparable simple \cstar-algebras whose corona is equal to any prescribed finite-dimensional \cstar-algebra (see also~\cite{ghasemi2016extension} for examples of non-simple nonseparable \cstar-algebras whose corona is one-dimensional).

Question~\ref{ques:sakai} was attributed to Sakai and was explicitly stated for the first time in Elliott's work~\cite{Ell:Derivations}, where (a strengthening of) the following theorem is proved. Recall that a \cstar-algebra is UHF if it can be written as the direct limit of matrix algebras.
\begin{theorem}[{\cite[Theorem 1]{Ell:Derivations}}]
Let $A$ and $B$ be separable unital UHF algebras. If $\cQ(A\otimes\mathcal K(H)) \cong \cQ(B\otimes \mathcal K(H))$ then $A\otimes\mathcal K(H) \cong B\otimes\mathcal K(H)$.
\end{theorem}

The methods used to prove this result involve the study of the connected components of the unitary group of $\cQ(A)$, which are shown to be directly related to the $K_0$-group $K_0(A)$, which in turn characterises $A\otimes\mathcal K(H)$ up to isomorphism. We refer to~\cite[Chapter 1]{rorstor} for further details and proper context on this foundational classification result, and to \cite{white2023abstract} for an up-to-date overview of the Elliott classification program. 
 In \cite[Corollary~D]{farah2022calkin} it was proven that $\cQ(H)$ is not isomorphic to $\cQ(A\otimes\cK(H))$ for any Elliott-classifiable \cstar-algebra $A$, thus giving support to the conjecture that $\cQ(H)$ is not isomorphic to $\cQ(A)$ for a simple $A$ not isomorphic to the algebra of compact operators.  This result was extended in \cite[Corollary~5.3]{farah2024coronas}  where it was proven that there are $2^{\aleph_0}$ non-isomorphic coronas of simple separable \cstar-algebras.

We remark that the analogue of Question~\ref{ques:sakai} for multiplier algebras has been fully answered in~\cite{brown:MA}, where it is proved that if $A$ and $B$ are separable \cstar-algebras such that $\cM(A) \cong \cM(B)$, then $A \cong B$.

It is not clear whether set theory could play a role in further progress on Question~\ref{ques:sakai}. Nevertheless, this problem highlights how little we know about the general structure theory of corona algebras. The paper~\cite{sakai1971derivations} ends with a series of open problems aiming at a better understanding on this matter; we include a sample. 

\begin{question}[{\cite[Problems 5-7]{sakai1971derivations}}]\label{Q.prescribedCalkin}
Given any prescribed unital abelian, or simple, or even arbitrary \cstar-algebra $C$, does there exist a simple \cstar-algebra $A$ such that $\cQ(A) \cong C$?
\end{question}

For an arbitrary \cstar-algebra $C$ there exists a \cstar-algebra $A$ such that $\cQ(A)\cong C$ (see \cite[Proposition 3.2]{farah2023obstructions}), but it is not known whether such~$A$ can be chosen to be simple. 
As we said, it is possible to have $\cQ(A)\cong \bbC$ for a simple~$A$ (\cite[Theorem~1]{sakai1971derivations}), and also for an approximately finite (but not simple)~$A$ (\cite{ghasemi2016extension}). Since for a finite-dimensional \cstar-algebra $F$ and any $A$ we have $F\otimes \cQ(A)\cong \cQ(F\otimes A)$, the answer to Question~\ref{Q.prescribedCalkin} is positive for a finite-dimensional $C$. Also, for every compact metric space $K$ there is a separable Banach space~$X$ such that the `Calkin algebra' $\cB(X)/\cK(X)$ is isomorphic to $C(K)$ (\cite{motakis2021separable}; the case when $K$ is a singleton was a major open problem until~\cite{argyros2011hereditarily}).

\section{ Ulam-stability}\label{S.Ulam}
\enumthree
The present section is of central importance to the subject of corona rigidity, but it has little to do with logic. Readers interested only in logic may want to skip this section and refer to it as necessary. 
In (\cite[\S VI.1]{Ul:Problems}) Ulam asked a series of questions about stability of mathematical theorems and solutions to functional equations. In particular, he asked whether `approximate' morphisms in certain categories (such as metric groups or vector spaces) are close to morphisms. The subjects inspired by Ulam's questions are independently studied and related results have been proved in numerous contexts (see the introduction of~\cite{BPRX.UlamBook} for an overview). As we shall see in this chapter, these stability-type questions naturally arise in the study of the rigidity of Borel quotients. In fact, Ulam-stability questions are often tightly connected to whether topologically trivial isomorphisms between quotient structures are algebraically trivial.

We first review the connection of Ulam-stability to rigidity in the categories of Boolean algebras and reduced products of finite structures, where the `asymptotically additive liftings' act as a bridge between the topological and algebraic trivialities. Later, we briefly discuss Ulam-stability in the categories of \cstar-algebras, its effect on the rigidity of maps between their reduced products, and its connection with the Kadison--Kastler conjecture.

\subsection{Approximate homomorphisms}
The following is probably the most studied context for Ulam-stability. 
For $\varepsilon>0$, a function $f\colon G_1\to G_2$ between groups such that $G_2$ is equipped with a metric $d$ is said to be an \emph{$\varepsilon$-homomorphism} if 
\begin{equation}\label{Eq.approximate}
d(f(gh),f(g)f(h))<\varepsilon
\end{equation}
 for all $g$ and $h$ in $G_1$. The question is, under what conditions one can conclude that $f$ is near a true group homomorphism? 

The following convention clearly does not cover \cstar-algebras, to which we will return later. 

\begin{convention}\label{Conv.G}
Throughout this section $\cG$ denotes a category of finite algebraic structures in a fixed finite signature equipped with a metric, and of a uniformly bounded diameter. 
 \end{convention}

\begin{definition}\label{Def.UlamStable} Suppose that $\cG$ is as in Convention~\ref{Conv.G}. An \emph{$\varepsilon$-ho\-mo\-mor\-phism} between structures in~$\cG$ is defined by requiring the analog of \eqref{Eq.approximate} to hold for every function symbol. The category $\cG$ is said to be \emph{Ulam-stable} if for every $\varepsilon>0$ there exists $\delta>0$ such that every $\delta$-homomorphism can be uniformly $\varepsilon$-approximated by a homomorphism. 
\end{definition}

There are numerous variations of this definition---one can talk about a specific class of maps being stable, or specify different classes for the structures in the domain and those in the range (note that the former need not be metric structures).

\begin{example}\label{Example4.3}
 The category of all finite subgroups of the unit circle, with the induced metric, is not Ulam-stable. Define $f_n\colon\bbZ/(2n+1)\bbZ\to \bbZ/(2n)\bbZ$ by 
 \[
 f_n\left(e^{i 2\pi \frac{k}{2n+1}}\right)=e^{i\pi \frac{k}{n}}
 \]
 for $0\leq k<2n+1$. Since for all $k,k'<2n+1$ the difference (modulo $2n$) between $[k+k']_{2n+1}$ and $[k+k']_{2n}$ is at most $1$, we have that
 \[
\textstyle \left| f_n\left(e^{i 2\pi \frac{k}{2n+1}}\right) f_n\left(e^{i 2\pi \frac{k'}{2n+1}}\right)- f_n\left(e^{i 2\pi \frac{k+k'}{2n+1}}\right)\right| =|1-e^{i2\pi\frac{1}{2n+1}}|<\frac {2\pi}{2n+1}, 
 \]
 hence $f_n$ is a $2\pi/(2n+1)$-homomorphism. Since every homomorphism between these two groups is trivial, $f_n$ cannot be 1-approximated by a homomorphism. 
\end{example}

A clever way to remedy Example~\ref{Example4.3} is to allow passing to a larger group in the range. This approach is taken in \cite{gowers2017inverse}, where it was shown that an approximate representation of a finite group $\Gamma$ in $M_n(\bbC)$ (with respect to the Hilbert--Schmidt norm) is near a representation of $\Gamma$ in a larger matrix algebra in which $M_n(\bbC)$ embeds unitally; see also \cite{de2019operator} for a von Neumann-algebraic adaptation of the Gowers--Hatami result. Also, Kazhdan (\cite{Kazhdan.Repr}) showed that approximate representations of amenable groups (where the range is the unitary group on an infinite dimensional Hilbert space) are uniformly close to representations, while the same fails for the free group $\mathbb F_2$, or any group which contains it (\cite{burger2013ulam}).

\subsection{From Borel to asymptotically additive}
 Ulam-stability is directly connected with the existence of liftings for homomorphisms between reduced products, the definition of which we quickly recall. 
 
 \begin{definition}\label{Def.Reduced}
 Fix $\cG$ as in Convention~\ref{Conv.G}. If $(A_n,d_n)$, for $n\in \bbN$, is a sequence of metric structures in $\cG$ and $\cI$ is an ideal on $\bbN$, the reduced product $\prod_n A_n/\bigoplus_\cI A_n$ is the quotient of the product $\prod_n A_n$ modulo the congruence that identifies $(x_n)$ and $(y_n)$ if and only if 
 \[
\lim_{n \to \cI} d_n(x_n, y_n) = \inf_{X\in \cI} \sup_{n\notin X} d_n(x_n,y_n)=0.
 \]
 \end{definition}

The following convention will help to simplify the notation (see \eqref{Eq.AI} below). 
 
 \begin{convention}\label{Conv.0}
In addition to Convention~\ref{Conv.G}, we assume that the signature of $\cG$ contains a binary operation $+$ and a constant $0$, interpreted as the neutral element for $+$ in every element of $\cG$. 
 \end{convention}

In the case of groups, or rings, $+$ and $0$ have the obvious interpretation. In Boolean algebras, $+$ will stand for $\cup$. 

For a fixed sequence of structures $A_n$ the following notation will come handy. For $I\subseteq \bbN$, we write $A_I=\prod_{n\in I} A_n$ (thus $A_{\bbN}=\prod_n A_n$) and consider the natural projection map $\pi_I\colon A_{\bbN}\to A_I$. We identify $A_I$ with
 \begin{equation}\label{Eq.AI}
 \{a\in A_{\bbN}\mid a_n=0\text{ for all }n\notin I\}. 
 \end{equation}
This convention identifies a quotient $A_I$ of $A_\bbN$ with a subalgebra thereof. 

\begin{definition} \label{Def.AsympAdd} A function 
$
\Phi_*\colon \prod_n A_n\to \prod_n B_n
$
 is \emph{asymptotically additive} if there are partitions $\bbN=\bigsqcup_n I(n)$ and $\bbN=\bigsqcup_n J(n)$ into finite intervals and maps $f_n\colon A_{I(n)}\to B_{J(n)}$ such that (identifying $a\in \prod_n A_n$ with the sequence $(a_n)$ such that $a_n\in A_{I(n)}$ and similarly for the range) $\Phi_*((a_n))=(f_n(a_n))_n$. 
\end{definition}

In the following lemma $\cP(\bbN)$ is identified with $\prod_\bbN\{0,1\}$. 

\begin{lemma}[{\cite[Theorem 1.5.2]{Fa:AQ}}]\label{L.P-ideal.aa} If $\cI$ is an analytic P-ideal, then every topologically trivial homomorphism $\Phi\colon \cP(\bbN)\to \cP(\bbN)/\cI$ has an asymptotically additive lifting. 
\end{lemma}

In the case when $\cI=\Fin$, Lemma~\ref{L.P-ideal.aa} was extracted from Shelah's proof (\S\ref{S.Shelah}) in~\cite{Ve:Definable}. Its proof and the proof of the following (taken from \cite[\S 4]{Fa:Liftings}) use the method of stabilisers, isolated by Just in~\cite{Just:Modification}. 

\begin{lemma}\label{L.AsympAdd} 
Suppose that $\cG$ is a category of Boolean algebras or groups as in Convention~\ref{Conv.G} and $A_n$ and $B_n$ belong to $\cG$. Let $\cI$ and $\cJ$ be ideals in~$\cP(\bbN)$, both containing $\Fin$. If 
\[
\textstyle \Phi\colon \prod_n A_n/\bigoplus_{\cI} A_n\to \prod_n B_n/\bigoplus_{\cJ} B_n
\]
 is a topologically trivial homomorphism, then $\Phi$ has an asymptotically additive lifting. 
\end{lemma}

In the proof, the assumption that all $A_n$ are finite is needed in order to use the classical characterisation of comeager subsets of the product $\prod_n A_n$ (see e.g., \cite[Theorem~9.9.1]{Fa:STCstar}). 
 The analog of Lemma~\ref{L.AsympAdd} is true in some other categories (see~\S\ref{S.fields-etc}).

\begin{problem}\label{P.AsympAdd} Find new instances of categories $\cG$ and ideals $\cI$, $\cJ$ under which the analog of Lemma~\ref{L.AsympAdd} holds. 
\end{problem}

Every progress on Problem~\ref{P.AsympAdd} connects Ulam-stability to the question when topologically trivial homomorphisms are algebraically trivial for certain quotient structures. In the following,~$\cG$ is a class of structures as in Convention~\ref{Conv.G}. In addition, we assume that $\cG$ is closed under taking finite products.

\begin{proposition} \label{P.Ulam} For $\cG$ as in the previous paragraph, the following are equivalent.
\enumone
\begin{enumerate}
\item\label{P.Ulam.1} The class $\cG$ is Ulam-stable. 
\item\label{P.Ulam.2} If a homomorphism between reduced products of structures in~$\cG$ over ideals that include $\Fin$ has an asymptotically additive lifting then it is algebraically trivial. 
\end{enumerate}
\end{proposition}

\begin{proof}[A sketch of a proof]~\ref{P.Ulam.2} implies~\ref{P.Ulam.1}: Assume~\ref{P.Ulam.1} fails. Fix $\varepsilon>0$ such that for every $n\geq 1$ there is a $1/n$-homomorphism $f_n\colon A_n\to B_n$ that cannot be uniformly $\varepsilon$-approximated by a homomorphism. Then $(a_n)\mapsto (f_n(a_n))$ is a lifting of an algebraically nontrivial homomorphism from $\prod_n A_n/\bigoplus_{\Fin} A_n$ into $\prod_n B_n/\bigoplus_{\Fin} B_n$. 

\ref{P.Ulam.1} implies~\ref{P.Ulam.2}:
Fix an asymptotically additive lifting of a homomorphism $\Phi\colon \prod_n A_n/\bigoplus_\cI A_n\to \prod_n B_n/\bigoplus_\cJ B_n$. Since $\cG$ is closed under finite products, we can assume $I(n)=J(n)=\{n\}$ (see Definition~\ref{Def.AsympAdd}) for all $n$. Thus there are $f_n\colon A_n\to B_n$ such that $\Phi_*((a_n))=(f_n(a_n))$ lifts $\Phi$. Since the signature is finite and the structures are finite, a pigeonhole argument shows that for every $m\geq 1$ the set 
\[
X_m=\{n\mid f_n\text{ is not a $1/m$-homomorphism}\}
\]
 belongs to $\cI$. Choose a sequence $(\varepsilon_m)$ such that every $1/m$-homo\-mor\-phism can be $\varepsilon_m$-approximated by a true homomorphism and $\lim_m \varepsilon_m=0$. Let $X_0=\bbN$ and for $j\in X_n\setminus X_{n+1}$ replace $f_n$ with a true homomorphism $g_n$ that $\varepsilon_m$-approximates it (for $n$ that does not belong to $\bigcup_n X_n$, if any, $f_n$ is already a homomorphism). The asymptotically additive lift given by the sequence $(g_n)$ is a true homomorphism, as required. 
\end{proof}

\subsection{Boolean algebras}\label{Ulam-Boolean-subsection}
A submeasure (Definition~\ref{Def.Submeasure}) on a set $I$ is \emph{strictly positive} if $\varphi(X)>0$ for every nonempty $X\subseteq I$. In this case, the Boolean algebra $\cP(I)$ is equipped with the distance $d_\varphi(X,Y)=\varphi(X\Delta Y)$. Of interest to us is the following. 

\begin{proposition} \label{P.GenDensity.ReducedProduct}
If $\cI$ is a generalised density ideal (Definition~\ref{Def.GeneralisedDensity}), then $\cP(\bbN)/\cI$ is isomorphic to the reduced product $\prod_n (\cP(I_n),d_n)/\bigoplus_{\Fin}(\cP(I_n),d_n)$ for some finite sets $I_n$.
\end{proposition}

\begin{proof} 
Fix a partition $\bbN=\bigsqcup_n I_n$ into intervals and submeasures $\varphi_n$ that concentrate on each $n$ such that $\cI=\Exh(\sup_n \varphi_n)$. (This is possible thanks to Definition~\ref{Def.GeneralisedDensity}.) Identify $X\subseteq \bbN$ with the element $(X\cap I_n)_n$ of $\prod_n (\cP(I_n), d_{\varphi_n})$. Then $X\in \cI$ if and only if $\varphi_n(X\cap I_n)\to 0$ as $n\to \infty$. 
\end{proof}

Lemma~\ref{L.AsympAdd}, Proposition~\ref{P.Ulam}, and Proposition~\ref{P.GenDensity.ReducedProduct} together imply that, for homomorphisms between quotients over analytic P-ideals (and especially for generalised density ideals), the equivalence of algebraic triviality and topological triviality hinges on Ulam-stability of approximate homomorphisms between metric Boolean algebras of the form $(\cP(\bbN), d_\varphi)$ where $d_\varphi=\varphi(X\Delta Y)$ for a strictly positive submeasure $\varphi$ on $\bbN$. The failure of Ulam-stability for such class was proved in~\cite{Fa:Approximate}. 

It makes therefore sense to isolate the following property.
\begin{definition}An ideal $\cI$ on $\bbN$ has the \emph{Radon--Nikodym property} if every topologically trivial Boolean algebra homomorphism $\Phi\colon \cP(\bbN)\to \cP(\bbN)/\cI$ is algebraically trivial (see \cite[\S 1.4 and \S 1.9]{Fa:AQ}).\footnote{The kernel of $\Phi$ is possibly nontrivial.} 
\end{definition}

%\marginpar{SG: The first and most obvious question that has been lingering from the beginning of this chapter is what are the examples of $\cG$ that are Ulam stable. Now that we finally have something to say about it we dodge the question by switching to Ulam stable maps! Is this on purpose?} 
%\marginpar{AV: I don't think this comment makes much sense, so I cut it.}
A submeasure is said to be \emph{nonpathological} if it is the supremum of measures dominated by it.\footnote{This is not the standard definition, according to which no submeasure on a finite set is pathological. Pathological submeasures have been intensively studied in connection to Maharam's and von Neumann's problems on characterisation of measure algebras, in a work that culminated in~\cite{Tal:Maharam}.} As suggested by the adjective, pathological submeasures are not easy to come by, and for the naturally occurring $F_\sigma$ and analytic P-ideals, the submeasure $\varphi$ associated to them (see~Theorem~\ref{T.Solecki}) can usually be chosen to be nonpathological. Such ideals are called \emph{nonpathological}. In~\cite{Fa:Approximate} and \cite{Fa:ApproximateII} it was proved that the approximate Boolean algebra homomorphisms and group homomorphisms with respect to nonpathological submeasures are Ulam-stable. %(or, in the sense of Definition~\ref{Def.UlamStable}, that the classes of finite Boolean algebras / groups with metrics coming from nonpathological submeasures are Ulam-stable)
Thus Lemma~\ref{L.P-ideal.aa} and Proposition~\ref{P.Ulam} together imply that nonpathological analytic P-ideals have the Radon--Nikodym property. The other parts of the following partial answer to Question~\ref{Q.UlamAB} are proved using similar ideas. 

\begin{theorem} \label{T.idealiso} %supersedes \label{cor:idealiso}
The following ideals have the Radon--Nikodym property.
\enumone 
\begin{enumerate}
\item $($\cite[Theorem~1.9.1]{Fa:AQ}$)$ Nonpathological analytic P-ideals. 
\item $($\cite{KanRe:New}$)$ Nonpathological $F_\sigma$ ideals. 
\item $($\cite{KanRe:New}$)$ \label{3.T.idealiso} For every countable additively indecomposable ordinal $\alpha$, the ideal $\{X\subseteq \alpha\mid $ the order type of $X$ is less than $\alpha\}$. 
\item $($\cite{KanRe:New}$)$ \label{4.T.idealiso} For every countable multiplicatively indecomposable ordinal $\alpha$, the ideal $\{X\subseteq \alpha\mid X$ does not contain a closed (in the order topology) copy of $\alpha\}$. 
\item $($\cite[Theorem~12.1]{Fa:Luzin}$)$ The ideals $\NWD(\bbQ)$ and $\NULL(\bbQ)$, of subsets of $\bbQ$ whose closures are nowhere dense and Lebesgue null, respectively and the Weyl ideal $\cZ_W$. 
\end{enumerate}
Therefore every topologically trivial homomorphism between analytic quotients $\cP(\bbN)/\cI$ and $\cP(\bbN)/\cJ$ such that $\cJ$ belongs to one of these classes is algebraically trivial. 
\end{theorem}

The first four items in Theorem~\ref{T.idealiso} apply even when these structures are considered as just groups, with the appropriate notion of Radon--Nikodym property; see~\cite{KanRe:New} and~\cite{KanRe:Ulam} for additional information. Analytic, and even $F_\sigma$, ideals without the Radon--Nikodym property exist by \cite[Theorem~1.9.5]{Fa:AQ}, but the following is open.

\begin{question}\label{Q.P-ideal.iso} Is every topologically trivial Boolean algebra isomorphism between quotients of analytic P-ideals algebraically trivial? What about group isomorphisms? 
\end{question}

 A version of this question for analytic P-ideals appears in \cite[Question~1.14.3]{Fa:AQ}, and its reformulation in terms of Ulam-stability for approximate isomorphisms is \cite[Question~1.14.4]{Fa:AQ}. 

\subsection{\cstar-algebras}
We now describe the situation in the case of \cstar-algebras and how Ulam-stability phenomena affect the relation between topologically trivial and algebraically trivial automorphisms of corona \cstar-algebras (see Question~\ref{ques:ulamgeneral}).

For Banach space-based structures, such as \cstar-algebras, an \emph{$\varepsilon$-$^*$-homo\-morphism} is a map between their unit balls which satisfies the analog of \eqref{Eq.approximate} for the algebraic operations and for multiplication by scalars of modulus at most $1$.\footnote{This is not the only way to define an $\varepsilon$-homomorphism between Banach algebras; we will return to this point in the proof of Theorem~\ref{T.Semrl}.} By $A_1$ we denote the unit ball of a \cstar-algebra $A$.

\begin{definition} \label{def:ehomo}
A function $f\colon A_1\to B_1$ between unit balls of \cstar-algebras is an \emph{$\varepsilon$-$^*$-homomorphism} if for all $x$ and $y$ in $A_1$ and every $\lambda\in \bbC$ with $|\lambda|\leq 1$ each of the following has norm at most $\varepsilon$: 
\begin{align*}
f(a+b)-f(a)-f(b), \ f(ab)-f(a)f(b),\ f(\lambda a)-\lambda f(a),\ f(a^*)-f(a)^*. 
\end{align*}
We abuse notation, and say that a map between \cstar-algebras is an $\varepsilon$-$^*$-homomorphism if its restriction to the unit ball is such.
\end{definition}

In the following we deviate from the convention established in Definition~\ref{Def.UlamStable} and consider the stability of approximate morphisms between specified `source objects' and `target objects'. This is done partly in order to efficiently handle the Calkin algebra and partly out of despair for the current lamentable lack of sharper results. 

\begin{definition}
A pair $(\cB,\cC)$ of classes of \cstar-algebras is \emph{Ulam-stable} if for every $\varepsilon>0$ there is $\delta>0$ such that if $A\in \cB$ and $B\in \cC$, then every $\delta$-$^*$-homomorphism $\varphi\colon A\to B$ can be $\varepsilon$-approximated by a $^*$-homomorphism, uniformly on the unit ball of $A$. A class of \cstar-algebras $\cB$ is \emph{Ulam-stable} if the pair $(\cB,\cB)$ is.
\end{definition}

Apart from the following two results, the question of Ulam-stability for other classes of \cstar-algebras is wide open. 

\begin{theorem}[{\cite[Theorem A]{mckenney2018ulam}}]\label{FD-Ulam}
The pair (finite-dimensional \cstar-algebras, \cstar-algebras) is Ulam-stable.
\end{theorem}

\begin{theorem}[{\cite{Semrl.Pert}}]\label{T.Semrl}
The class of separable abelian \cstar-algebras is Ulam-stable.
\end{theorem}

\begin{proof}In \cite{Semrl.Pert} an $\varepsilon$-homomorphism between \cstar-algebras 
was defined by requiring 
\begin{align}
\begin{split}\label{Eq.epsilon.Semrl}
\|f(x+y)-f(x)-f(y)\|&\leq \varepsilon(\|x\|+\|y\|),\\
\|f(xy)-f(x)f(y)\|&\leq \varepsilon(\|x\|\|y\|),\\
\|f(x^*)-f(x)^*\|&\leq \varepsilon\|x\|
\end{split}
\end{align}
for all $x$ and $y$ in the domain, and proved that for such a function $f$ between abelian \cstar-algebras there is a $^*$-homomorphism $\Phi$ such that $\|\Phi-f\|\leq ((4+\frac 32 (1+\pi))\varepsilon +1-\sqrt {1-4\varepsilon})\|f\|$ (\cite[Theorem~5.1]{Semrl.Pert}). 

A simple computation (regrettably absent from \cite{vignati2018rigidity}) shows that if $\varphi\colon A_1\to B_1$ is an $\varepsilon$-$^*$-homomorphism in the sense of Definition
\ref{def:ehomo}, then $f\colon A\to B$ defined by $f(x)=\|x\|\varphi(x/\|x\|)$ for $x\neq 0$ and $f(0)=0$ is a $4\varepsilon$-homomorphism in the sense of \eqref{Eq.epsilon.Semrl} and that $\|f(x)-\varphi(x)\|\leq \varepsilon$ for all $x\in A_1$. The thesis follows by connecting the dots.
\end{proof} 

%%%% Semrl
%A starting point for expanding the current knowledge on Ulam-stability on \cstar-algebras is given by the following two questions. % (the interested reader is recommended to start from the first one!).

The proof of the following uses the Gelfand--Naimark duality, the topological definition of algebraic triviality (Definition~\ref{defin:trivialtop}), a weakening of Lemma~\ref{L.AsympAdd}, and Theorem~\ref{T.Semrl}.

\begin{theorem}[{\cite{vignati2018rigidity}}]\label{thm:UlamAB}
All topologically trivial automorphisms of coronas of abelian \cstar-algebras are algebraically trivial. 
\end{theorem}

\subsection{Reduced products of \cstar-algebras} In the noncommutative case, the route towards proving algebraic triviality from Ulam-stability is somewhat oblique (and in fact, not yet completely understood in the general case), unless the algebras are reduced products. 
A unital \cstar-algebra $A$ does not have central projections if it cannot be written as $A=B\oplus C$ where $B$ and $C$ are nonzero \cstar-algebras; this assumption is used only to mitigate notational issues. The reduced product $\prod_n A_n/\bigoplus_n A_n$ was defined in Example~\ref{ex:coronas}~\ref{ex.reduced}. 

\begin{theorem}[{\cite[Theorem 3.16]{mckenney2018forcing}}]\label{T.Ulam}
Let $\cB$ and $\mathcal C$ be classes of separable unital \cstar-algebras. Suppose that algebras in $\cB$ have no central projections, and that~$\cC$ is closed under finite products. The following are equivalent:
\begin{itemize}
\item The pair $(\cB,\cC)$ is Ulam-stable. 
\item If $A_n\in \cB$ and $B_n\in \cC$ for all $n$, and 
\[
\varphi\colon\prod_n A_n/\bigoplus_n A_n\to\prod_n B_n/\bigoplus_nB_n
\] 
is a topologically trivial $^*$-homomorphism, then it is algebraically trivial. 
\end{itemize}
\end{theorem}

The proof of Theorem~\ref{T.Ulam} roughly follows the line of the proofs of Lemma~\ref{L.AsympAdd}---Proposition~\ref{P.Ulam}, with some technical wrinkles. First, one applies Lemma~\ref{L.AsympAdd} to the projections in the center of $\prod_nA_n/\bigoplus_nA_n$, which is isomorphic to $\mathcal P(\bbN)/\Fin$, to get an asymptotically additive lift, and then one runs similar arguments as in Proposition~\ref{P.Ulam}.

% If all $A_n$ are finite-dimensional, then one replaces $A_n$ with a finite $1/n$-dense subset of its unit ball. The case when $A_n$ are infinite-dimensional uses stratification similar to the one described in the sketch of the proof of Theorem~\ref{thm:toptrivcalkin} below. 

\begin{definition} \label{Def.MX}
For an infinite $X\subseteq \bbN$ let 
\[
\textstyle \cM_X=\prod_{n\in X} M_n(\bbC)/\bigoplus_{n\in X} M_n(\bbC).
\] 
 \end{definition}

\begin{corollary} \label{cor:Ulamredprod}
Every topologically trivial $^*$-homomorphism between \cstar-algebras of the form~$\cM_X$ is algebraically trivial. 

Every topologically trivial $^*$-homomorphism from a \cstar-algebra of the form $\cM_X$ into the corona of a separable \cstar-algebra is algebraically trivial.
\end{corollary}

\begin{proof} The first part is a consequence of Theorem~\ref{FD-Ulam} and Theorem~\ref{T.Ulam}. Proof of the second part involves similar ideas and the observation that the product structure is used only in the domain. 
\end{proof}

Corollary~\ref{cor:Ulamredprod} applies only to reduced products, and $\cQ(H)$ isn't one; see \S\ref{S.Strat} for a remedy. 
%As we will see (\S\ref{6bii.PFAAVi}), one of the main application of Corollary~\ref{cor:Ulamredprod} regards automorphisms of the Calkin algebra $\mathcal Q(H)$, as Ulam-stability of matrix algebras is key in showing that under $\OCA$ all automorphisms of $\mathcal Q(A)$ are inner. 

We conclude with a very general question and two of its very specific instances. Each one of these questions corresponds to a particular instance of the lifting problem for $*$-homomorphisms between quotient \cstar-algebras.

\begin{question} Is the category of all separable \cstar-algebras Ulam-stable? 
\end{question}

\begin{question}
Is the category of \cstar-algebras of the form $C(X, M_n(\bbC))$, where $X$ is compact and metrisable, Ulam-stable?
\end{question}

\begin{question}
Is the pair (separable abelian \cstar-algebras, all \cstar-algebras) Ulam-stable?
\end{question}

\subsection{The Kadison--Kastler conjecture} Since our main interest is in isomorphisms, it is natural to consider the following notion (\cite[Definition 1.2]{mckenney2018ulam}). Recall that by $A_1$ we denote the unit ball of a \cstar-algebra $A$.

\begin{definition}
Let $A$ and $B$ be \cstar-algebras, $\varepsilon>0$. A map $\Phi\colon A\to B$ is an \emph{$\varepsilon$-$^*$-isomorphism} if the following holds. 
\begin{itemize}
\item It is an $\varepsilon$-$^*$-homomorphism;
\item It is $\varepsilon$-injective: if $a\in A_1$, then $|\norm{\Phi(a)}-\norm{a}|<\varepsilon$;
\item It is $\varepsilon$-surjective: $\Phi[A_1]$ is $\varepsilon$-dense in $B_1$.
\end{itemize}
 If there is an $\varepsilon$-$^*$-isomorphism $A\to B$, $A$ and $B$ are said to be $\varepsilon$-$^*$-isomorphic.

A \cstar-algebra $A$ is said to be \emph{stable under approximate isomorphisms} if there is $\varepsilon>0$ such that $A\cong B$ whenever $B$ is $\varepsilon$-$^*$-isomorphic to $A$. %A class of \cstar-algebras $\mathcal C$ is said to be stable under approximate isomorphisms if there is $\varepsilon>0$ such that $A\cong B$ whenever $A\in\mathcal C$ and $B$ is $\varepsilon$-$^*$-isomorphic to $A$
\end{definition}

 Ulam-stability of \cstar-algebras is related to the Kadison--Kastler perturbation theory. %and the theory of near inclusions of \cstar-algebras.
In~\cite{kadison1972perturbations}, Kadison and Kastler initiated the study of perturbations of algebras of operators. If $A$ and $B$ are two \cstar-algebras sitting in the same $\mathcal B(H)$, the \emph{Kadison--Kastler distance} is the Hausdorff distance between their unit balls, that is
\[
d_{\textrm{KK}}(A,B)=\max\{\sup_{a\in A_1}\inf_{b\in B_1}\norm{a-b},\sup_{b\in B_1}\inf_{a\in A_1}\norm{a-b}\}.
\]
Kadison and Kastler asked what properties are preserved at a small $d_{\textrm{KK}}$-distance. They implicitly conjectured that, given a \cstar-algebras $A$, there exists $\varepsilon>0$ such that if $d_{\textrm{KK}}(A,B)<\varepsilon$ then $A$ and $B$ are isomorphic. (A stronger version of this conjecture asks for the \cstar-algebras to be conjugated by a unitary close to the identity.)

We say that $A$ satisfies the Kadison--Kastler conjecture if such an $\varepsilon$ exists. A class $\cC$ of \cstar-algebras is said to be \emph{Kadison--Kastler stable} (or KK stable) if each algebra in $\cC$ satisfies the Kadison--Kastler conjecture and the same $\varepsilon$ witnesses stability uniformly on elements of $\mathcal C$. The class of nonseparable \cstar-algebras is not KK stable, and there is even a nonseparable \cstar-algebra that does not satisfy the Kadison--Kastler conjecture (\cite{choi1983completely}, see also \cite[\S 14.4]{Fa:STCstar}). The class of unital separable amenable\footnote{Amenable, or nuclear, \cstar-algebras form the most studied and arguably the most important class of \cstar-algebras.} \cstar-algebras is KK stable (\cite{CSSWW:Perturbations}), but it is not known whether the class of all separable unital \cstar-algebras is KK stable. 

Fix $\varepsilon>0$, and suppose that $A$ and $B$ are two separable \cstar-algebras such that $d_{\textrm{KK}}(A,B)<\varepsilon$. If to each element of $A_1$ we associate an element of $B_1$ at distance $<\varepsilon$, we obtain an $\varepsilon$-$^*$-isomorphism between $A$ and $B$. Therefore, if $A$ is stable under approximate isomorphisms, $A$ satisfies the Kadison--Kastler conjecture. Similarly, if a class of \cstar-algebras is Ulam-stable, it is KK stable. %This shows that our notion of stability is extremely strong; proving that a certain class of \cstar-algebras is Ulam-stable will then have consequences in perturbation theory. 
We do not know whether either of the two implications above can be reversed.

\section{Independence results, I. Nontrivial isomorphisms}\label{S.Independence}
\enumthree 
In this section we analyse constructions of isomorphisms of Borel quotient structures which are not topologically trivial, therefore giving partial answers to Question~\ref{Q.main}. To do this, we often rely on additional set theoretic assumptions such as the Continuum Hypothesis or its weakenings. 

The section is structured as follows: \S\ref{S.CHModelTheory} and \S\ref{S.6.2} are devoted to model theoretic methods, in particular the concept of countable (i.e. $\aleph_1$-)saturation, and how they are applied to prove the existence of topologically nontrivial automorphisms from $\CH$. In \S\ref{S.6aii} we investigate the failure of countable saturation for the Calkin algebra from a few angles, and mention a weaker version of countable saturation called the degree-1 saturation that is satisfied by the corona of every separable \cstar-algebra. In \S\ref{S.Strat} we show how the Calkin algebra can be stratified into a direct limit of Banach subspaces. 
\S\ref{S.Cohomology} uses \S\ref{S.Strat} to describe how $\CH$ implies the existence of topologically nontrivial automorphisms of a large class of corona \cstar-algebras to which model theoretic methods do not apply. In \S\ref{S.CHRemainders} we describe how to construct nontrivial homeomorphisms of \v{C}ech--Stone remainders of a large class of topological spaces, including all manifolds of nonzero dimension. Finally, in \S\ref{S.OtherModels}, we discuss some models of the negation of $\CH$ in which nontrivial automorphisms exist.

\subsection{Saturation}\label{S.CHModelTheory}
If there is no obvious isomorphism between two structures, then a standard method for constructing one uses `back-and-forth' methods. These techniques generally consist of an inductive process building an increasing sequence (transfinite, if need be) of partial isomorphisms\footnote{A partial isomorphism, as any function, is identified with its graph, thus `increasing sequence of partial isomorphisms’ has a precise and intended meaning.} between the objects involved, and whose domains become larger and larger as the induction goes on.
These constructions often require a considerable amount of bookkeeping to keep track of the partial maps defined at each step, especially if the structures are nonseparable. 
Nevertheless, a theorem of Keisler states that if two structures are saturated and have the same cardinality, then a back-and-forth system between them exists if and only if they are elementarily equivalent (Theorem~\ref{T.Keisler} below). This readily extends to metric structures of same density character. Since many quotients such as ultraproducts and some corona algebras are saturated under $\CH$, Keisler’s theorem often reduces the question of isomorphism to the question of elementary equivalence. Analogous back-and-forth methods can be used to prove, again under CH, the existence of topologically nontrivial isomorphisms.

In this subsection we review the effects of countable saturation on the isomorphism questions for coronas of \cstar-algebras under the assumption of $\CH$. The non-expert reader may refer to~\cite{BYBHU} for all basic definitions and for an introduction to the model theory of metric structures, and to~\cite{FaHaSh:Model2} for a reference focusing specifically on operator algebras.

\subsubsection{$\sigma$-complete back-and-forth systems} Each isomorphism between metric structures of density character $\aleph_1$ corresponds to a \emph{$\sigma$-complete back-and-forth system} (see~\cite[Proposition 16.6.1]{Fa:STCstar}). 
If $A$ and $B$ are metric structures, a $\sigma$-complete back-and-forth system between $A$ and $B$ is a poset $\bbF$ such that each element $p$ of $\bbF$ denotes a partial isomorphism $\Phi^p: A^p \to B^p$ for $A^p\subseteq A$ and $B^p\subseteq B$, with the following properties:
\begin{itemize}
 \item $p\leq q$ if $A^p \subseteq A^q$ and $B^p \subseteq B^q$, and $\Phi^q$ extends $\Phi^p$.
 \item For every $p\in \bbF$ and $a\in A$ and $b\in B$ there exists $q\geq p $ in $\bbF$ such that $a\in A^q$ and $b\in B^q$.
 \item $\bbF$ is $\sigma$-complete: For every increasing sequence $\{p_n\colon n\in \bbN\}$ in $\bbF$ the partial isomorphism $\overline{\bigcup_n\Phi^{p_n}}: \overline{\bigcup_n A^{p_n}} \to \overline{\bigcup_n B^{p_n}}$ belongs to $\bbF$.
\end{itemize}
We assume that the reader is familiar with the notion of a type and saturation, but the following is stated explicitly because we somewhat depart from the standard model-theoretic terminology. 
\begin{definition} A structure is \emph{countably saturated} if every consistent type over a countable set is realised. 
\end{definition}
Thus `countably saturated' is what is usually called `$\aleph_1$-saturated' and should not be confused with the notion of `countably saturated' as defined in e.g.,  \cite[\S 2.3]{ChaKe} (such models are countable by definition, and  countably saturated models in our sense are necessarily uncountable). The reason for this change of terminology is the same as the reason why countable additivity of Lebesgue measure is rarely referred to as $\aleph_1$-additivity. If the language is countable, then countable saturation is equivalent to countable compactness, i.e., the requirement that every countable  consistent type is  realised. 

If two structures are countably saturated and elementarily equivalent, then there is a $\sigma$-complete back-and-forth system between them (see~\cite[Theorem 16.6.4]{Fa:STCstar}). This implies a special case of the following theorem due to Keisler (e.g.,~\cite[Theorem 5.1.13]{ChaKe}). 

\begin{theorem}\label{T.Keisler}
 Countably saturated structures of density character $\aleph_1$ are isomorphic if and only if they are elementarily equivalent.
 \end{theorem}
 
Since infinite-dimensional countably saturated \cstar-algebras have density character at least $\mathfrak{c}$, if $\CH$ is assumed then two elementarily equivalent countably saturated \cstar-algebras of density character not greater than $\fc$ are isomorphic.

The countable saturation of metric structures of density character $\aleph_1$ also has fundamental effects on the automorphism groups, since it allows the construction of $2^{\aleph_1}$ distinct $\sigma$-complete back-and-forth systems, giving rise to $2^{\aleph_1}$ distinct automorphisms of the structure (see~\cite[Theorem 16.6.3]{Fa:STCstar}).

\begin{thm}\label{ctble-sat-auto}
Every countably saturated structure of density character $\aleph_1$ has $2^{\aleph_1}$ automorphisms.
\end{thm}

Unless $\CH$ holds, the theory of such structure has to be stable, for the following reason. If the theory of a countably saturated structure $M$ is unstable, then by the order property (which is equivalent to unstability) there are a countable subset $X$ of $M^n$ for some $n$ and a definable relation $\rho$ on $X$ such that $(X,\rho)$ is isomorphic to $(\mathbb Q,<)$. Since every gap in $(X,\rho)$ is described by a countable type, it is filled, and therefore the size of $M$ is at least $\fc$. On the other hand, such models exist. In general, if the theory of~$M$ is $\aleph_1$-categorical  then all uncountable models of its theory are saturated (this is a feature of the proof of Morley's theorem, \cite[Corollary~7.1.8]{ChaKe}). 

Given any reasonable definition of `trivial' automorphism, the conclusion of Theorem~\ref{ctble-sat-auto} is often paired with a cardinality argument to show that, under CH, these structures have many nontrivial automorphisms. First we review a few applications of this theorem to commutative coronas via the Stone and Gelfand dualities.

\subsubsection{Borel quotients of $\mathcal P(\bbN)$}\label{S.Borel_quotients_of_P(N)}

As already pointed out in \S\ref{S.Intro} and \S\ref{S.Abel}, in~\cite{Ru} Rudin used a back-and-forth argument to show that there are $2^{\mathfrak c}$ homeomorphisms of $\beta\bbN\setminus\bbN$ under $\CH$, and therefore $2^{\mathfrak c}$ automorphisms of the Boolean algebra $\mathcal P(\bbN)/\Fin$. This can indeed be considered as an instance of Theorem~\ref{ctble-sat-auto} since $\mathcal P(\bbN)/\Fin$ is countably saturated as a Boolean algebra (this is a consequence of~\cite{JonOl:Almost}). More generally, it was shown in~\cite{JustKr} that, for any $F_\sigma$ ideal $\mathcal I$ containing $\Fin$, the Boolean algebra $\mathcal P(\bbN)/\mathcal I$ is countably saturated. Since there are only $\mathfrak c$ many topologically trivial automorphisms, and all atomless Boolean algebras are elementarily equivalent, we have the following. 

\begin{corollary}\label{thm:Rudin}
Assume $\CH$. Then $\cP(\bbN)/\cI$ is isomorphic to $\cP(\bbN)/\Fin$ for every $F_\sigma$ ideal $\cI\supseteq \Fin$. Each of these quotients has $2^{\mathfrak c}$ topologically nontrivial automorphisms. 
\end{corollary}

 The situation with ideals which are not $F_\sigma$ is more interesting. 
As pointed out in \S\ref{ss.borel_quotients}, Erd\" os and Ulam asked whether the quotients over the density ideals $\cZ_0$ and $\cZ_{\log}$ (\eqref{eq.Z0} and \eqref{eq.Zlog}) are isomorphic. 
 These quotients, as well as the quotients over all analytic P-ideals, are equipped with a complete metric (Theorem~\ref{T.Solecki}). If an ideal is not $F_\sigma$, then one can take a strictly decreasing Cauchy sequence which witnesses that the quotient is not countably saturated as a classical (discrete) structure. This presents a serious obstacle to constructing an isomorphism between such quotients by a back-and-forth method. 

However, in~\cite{JustKr} CH was used to prove that the quotients over $\cZ_0$, $\cZ_{\log}$, and even all so-called EU-ideals, are isomorphic. The intricate back-and-forth construction produced an isometric isomorphism. It took a couple of decades before this heroic achievement was put into a proper context: 
 All quotients of $\mathcal P(\bbN)$ by generalised density ideals (Definition~\ref{Def.GeneralisedDensity}) are countably saturated metric structures by Proposition~\ref{P.GenDensity.ReducedProduct} and Theorem~\ref{ctble-sat-reduced}. This fact was used in \cite[Theorem~2]{Fa:CH} to prove that under CH there are only two non-isomorphic quotients over dense density ideals, and it implies the following. 

\begin{corollary}
Assume $\CH$ and let $\mathcal I\subseteq\mathcal P(\bbN)$ be a generalised density ideal. Then there are $2^{\mathfrak c}$ topologically nontrivial automorphisms of $\mathcal P(\bbN)/\mathcal I$.
\end{corollary} 

In general, we do not know whether $\CH$ implies the existence of topologically nontrivial automorphisms for every quotient of the form $\mathcal P(\bbN)/\mathcal I$, where~$\mathcal I$ is an analytic P-ideal containing $\Fin$. A vaguely related result, that there are $\fc$ many non-isomorphic quotients of $\cP(\bbN)$ over Borel ideals, appears in~\cite{Oliver_2004}.

\subsubsection{Zero-dimensional remainders} In~\cite{Pa:Universal} Parovi\v{c}enko characterised, under $\CH$, $\beta\bbN\setminus\bbN$ as the only space sharing certain topological properties (such properties define a Parovi\v{c}enko space, in today's terminology). These topological properties are indeed shared by all spaces of the form $\beta X\setminus X$, where $X$ is a second countable locally compact noncompact zero-dimensional space\footnote{It was proved in~\cite{van1978parovivcenko} that $\beta\bbN\setminus\bbN$ being unique among Parovi\v{c}enko spaces is equivalent to $\CH$.}. Through Stone duality, Parovi\v{c}enko's theorem essentially states that the Boolean algebras of clopen sets of those remainders are atomless and countably saturated. Since the theory of atomless Boolean algebras is complete, all of its models are elementarily equivalent. With the definition of algebraically trivial homeomorphisms (Definition~\ref{defin:trivialtop}) in mind and the fact that topologically trivial automorphisms of coronas of abelian \cstar-algebras are dual to algebraically trivial homeomorphisms (\cite[Proposition 2.7]{vignati2018rigidity}), this implies the following.
 
 \begin{corollary}\label{cor:Parov}
Assume $\CH$. Let $X$ be a second countable locally compact noncompact zero-dimensional space. Then $\beta X\setminus X$ is homeomorphic to $\beta\bbN\setminus \bbN$ and it has $2^{\mathfrak c}$ algebraically nontrivial autohomeomorphisms.
\end{corollary}

There is no obvious obstruction preventing the statement above from holding even outside the zero-dimensional framework, where Stone duality is lacking. 

\begin{conjecture} \label{conj:homeo}
Assume CH and let $X$ be a locally compact noncompact Polish space. Then $\beta X \setminus X$ has $2^\mathfrak{c}$
(algebraically nontrivial) autohomeomorphisms.
\end{conjecture}

Thanks to the Gelfand--Naimark duality, one can verify Conjecture \ref{conj:homeo} on a space $X$ by studying the countable saturation of $C(\beta X \setminus X)$.
In what was Question~5.7 in an earlier version of this paper it was asked whether the corona $C(\beta X\setminus X)$ is countably saturated for every locally compact, noncompact, Polish space $X$. In~\cite{farah2023obstructions} it was proven (among other things) that $C(\beta \bbR^n\setminus \bbR^n)$ is not countably saturated unless $n=1$ (compare footnote~\ref{footnote.hart}), giving a strong negative answer to this question. In~\cite{farah2023obstructions} there are two proofs of this result, one relatively elementary and one using `definable homotopy’ (it is not difficult to see that homotopy is in general not definable in continuous logic). 

On the positive side, if $X$ is equal to a strictly increasing union of compact subsets $K_n$ which satisfy ($\partial K_n$ is the topological boundary of $K_n$) 
$\sup_n | \partial K_n | < \infty$, then $C(\beta X\setminus X)$ is countably saturated (this is a special case of \cite[Theorem~2.5]{FaSh:Rigidity}, and includes the case $X=\mathbb R$).

\subsubsection{Countable saturation of the reduced products}

 Our next task is to review some of the applications of Theorem~\ref{ctble-sat-auto} in the noncommutative setting of reduced products. Ultraproducts associated with countably incomplete ultrafilters are countably saturated (see~\cite[Theorem 6.1.1]{ChaKe}). This, in particular, is the case for the norm ultraproducts of \cstar-algebras. As for coronas of separable \cstar-algebras, they are not necessarily countably saturated. Notably, the Calkin algebra is not countably saturated (\S\ref{S.6aii}). However, corona algebras that are reduced products over $\Fin$ are countably saturated by the following result (\cite{FaSh:Rigidity}, see also {\cite[Theorem 16.5.1]{Fa:STCstar}}).

\begin{thm}\label{ctble-sat-reduced}
If $(A_n)_{n\in \mathbb N}$ is a sequence of metric structures in the same language, then the reduced product $\prod_n A_n/\bigoplus_{\Fin} A_n$ is countably saturated.
\end{thm}

Recall that if $A_n$ are unital \cstar-algebras then $\prod_n A_n/\bigoplus_{\Fin} A_n$ is isomorphic to the corona of the direct-sum $\bigoplus_{\Fin} A_n$ (Example~\ref{ex:coronas}). This class includes all the \emph{asymptotic sequence algebras} $A_\infty=\ell_\infty(A)/c_0(A)$, associated to a \cstar-algebra $A$. 

 \begin{corollary}\label{cor:RP}
Assume $\CH$. If $(A_n)_{n\in \mathbb N}$ is a sequence of metric structures of density character $\leq \aleph_1$ in the same language, the reduced product $\prod_n A_n/\bigoplus_{\Fin} A_n$ has $2^{\mathfrak c}$ topologically nontrivial automorphisms. 
\end{corollary}

\subsubsection{Metric Feferman--Vaught theorem}\label{FV-subsection} Theorem~\ref{ctble-sat-reduced} is proved by combining the metric version of the classical \emph{Feferman--Vaught theorem} with the proof of the countable saturation of ultraproducts associated with countably incomplete ultrafilters (see e.g., \cite[\S 16.5]{Fa:STCstar}). %\marginpar{AVi: uhm..is this true?}. 
In classical model theory Feferman and Vaught (\cite{feferman1959first}) gave an effective way to determine the satisfaction of formulas in the reduced products of models of the same language associated with any ideal over $\bbN$. Their result was generalised in~\cite{ghasemi2016reduced} to metric structures. The metric Feferman--Vaught theorem essentially states that if $(A_n)_{n\in\bbN}$ are $\mathcal L$-structures and $\mathcal I$ is an ideal on $\bbN$ then the theory of the reduced product $\prod_n A_n/\bigoplus_{\mathcal I}A_n$ can be recursively computed from the theories of $A_n$ and the theory of $\mathcal P(\bbN)/\mathcal I$ as a Boolean algebra.

%that for any formula $\varphi(\bar x)$ in a metric language $\mathcal L$ and every $k \in \bbN$ one can recursively assign a finite set $\bbF(\varphi, k) $ of $\mathcal L$-formulas all of whose free variables are included in $\bar x$, and a finite set of formulas $\bbB(\varphi, k)$ in the language of Boolean algebras such that the value of $\varphi$ in every reduced product of $\mathcal L$-structures $(A_n)_{n \in \bbN}$ associated with an ideal $\mathcal I$ on $\bbN$ is determined up to $2/k$ by the values of the formulas of $\bbB(\varphi, k)$ interpreted in $\mathcal P(\bbN)/\mathcal I$ and the interpretations of formulas of $\bbF(\varphi, k) $ in each $A_n$. In particular, this implies that the theory of the reduced product of $(A_n)_{n \in \bbN}$ can be recursively computed from the theories of $A_n$. 
The following theorem can be deduced from the metric Feferman--Vaught theorem. It follows directly from \cite[Theorem 3.3]{ghasemi2016reduced}. It applies to both continuous and discrete languages, in which case, if a sentence $\theta$ is interpreted in a structure $A$, the value $\theta^A$ belongs to $\{\text{true, false}\}$.

\begin{thm}\label{FV-coro}
Let $\mathcal I$ be an ideal on $\bbN$, and suppose that $(A_n)_{n\in \bbN}$ and $(B_n)_{n\in\bbN}$ are structures in the same language. If for every sentence $\theta$ we have $\lim_{n\to \mathcal I} \theta^{A_n} = \lim_{n\to \mathcal I}\theta^{B_n}$ then $\prod_n A_n/\bigoplus_{\mathcal I} A_n\equiv \prod_n B_n/\bigoplus_{\mathcal I} B_n$. 
Therefore, if $A_n \equiv B_n$ for all $n$ then $\prod_n A_n/\bigoplus_{\mathcal I} A_n\equiv \prod_n B_n/\bigoplus_{\mathcal I} B_n$.

%\item[(b)] If some $\calL$-structure $A$ satisfies $\lim_{n\to \mathcal I} \theta^{A_n} = \theta^A$ for every $\calL$-sentence~$\theta$, then $\prod_n A_n/\bigoplus_{\mathcal I} A_n\equiv \prod_n A/\bigoplus_{\mathcal I} A$. 
%\end{itemize}
\end{thm}

In fact, the results of \cite{farah2019between} provide a canonical model for the theory of the reduced power $A_\infty:=\prod_n A/\bigoplus_{\mathrm{Fin}} A$ for every structure $A$. Let $C(K, A)$ denote the structure of the language of $A$ consisting of all continuous functions from $K$ into $A$, where $K$ is the Cantor space.\footnote{Note that $C(K,A)$ is naturally identified with the inductive limit $\lim_n A_n$ where $A_0 =A$ and $A_{n+1} = A_n \oplus A_n$ for all $n$, with the connecting maps $a \to (a, a)$. This definition also makes sense if $A$ is discrete.} 

\begin{thm}\label{T.C(K,A)}
If $A$ and $(A_n)_{n\in \bbN}$ are structures of the same language such that $\lim_{n\to \infty} \theta^{A_n} = \theta^A$ for every sentence $\theta$ then $\prod_n A_n/\bigoplus_{\mathrm{Fin}} A_n\equiv C(K,A)$.
\end{thm}

\begin{proof}
An elementary embedding from $C(K,A)$ into $A_\infty$ was exhibited in~\cite[Proposition 3.5]{farah2019between}, and the general case follows from Theorem~\ref{FV-coro}.
\end{proof}

\subsubsection{The Palyutin--Palmgren approach}\label{S.PP}
An elegant alternative approach to proving countable saturation of reduced products associated with countably generated ideals that does not use the Feferman—Vaught theorem is given in \cite{palmgren}. It uses analysis of reduced products due to Palyutin (\cite{palyutin1, palyutin2}) that deserves to be better known. Palyutin defined the set of $h$-formulas of a language $\calL$ (not to be confused with Horn formulas) to be the smallest set of $\calL$-formulas that includes all atomic formulas and such that if $\varphi$ and $\psi$ are $h$-formulas then so are $\varphi\wedge\psi$, $(\forall x)\varphi$, $(\exists x)\varphi$, and $(\exists x)\varphi\wedge(\forall x)(\varphi\rightarrow\psi)$. By induction on complexity, one proves the following analog of \L o\'s's theorem for arbitrary reduced products. 

\begin{thm} \label{T.Palyutin}
	If $M_i$, for $i\in I$, are structures of the same language,~$\cI$ is an ideal on $I$, $M=\prod_i M_i/\bigoplus_\cI M_i$, $\varphi(\bar x)$ is an $h$-formula, and $\bar a=[\bar a_i]_\cI$ in $M$ is of the appropriate sort then 
	$M\models \varphi(\bar a)$ if and only if 
	\[
	 \{i\in I \mid M_i\models \neg \varphi(\bar a_i)\}\in \cI. 
	\]
\end{thm}

The key to Palmgren's proof is the observation that in every reduced power associated with $\Fin$ (or any other ideal such that the quotient Boolean algebra is atomless) every formula is equivalent to a Boolean combination of $h$-formulas. Therefore every type over such reduced product is axiomatised by a set consisting of $h$-formulas and negations of $h$-formulas. This proof was extended to prove countable saturation of reduced powers over a larger class of the so-called layered ideals in \cite[Theorem~4.1]{debondt2023saturation}.

\subsubsection{Asymptotic sequence algebras} We write $A_\infty$ for $\prod_n A/\bigoplus_{\mathrm{Fin}} A$ for any structure $A$, metric or discrete (in case of \cstar-algebra, this is the asymptotic sequence algebra). The following is an immediate consequence of Theorem~\ref{ctble-sat-reduced}, Theorem~\ref{FV-coro}, Theorem~\ref{T.C(K,A)}, and Theorem~\ref{T.Keisler}.

\begin{corollary} \label{thm:eeiso}
Assume $\CH$.
If $(A_n)_{n \in \bbN}$ and $(B_n)_{n \in \bbN}$ are structures of the same language of density character $\leq \aleph_1$ such that $A_n\equiv B_n$ for all but finitely many $n$, then $\prod_n A_n/\bigoplus_{\Fin} A_n$ and $\prod_n B_n/\bigoplus_{\Fin} B_n$ are isomorphic. 

In particular, if $A$ and $B$ are elementarily equivalent and of density character $\leq \aleph_1$, then $A_\infty\cong B_\infty$, and both structures are isomorphic to the ultrapower of $C(K,A)$ associated with any nonprincipal ultrafilter on $\bbN$, $K$ being the Cantor space. 
\end{corollary}

For discrete structures `density character' is of course a synonym for cardinality. For more see \cite[Theorem~A--F]{farah2019between}.

This brings us to the question of elementary equivalence of \cstar-algebras. 
Finite-dimensional \cstar-algebras and separable UHF-algebras are elementarily equivalent if and and only if they are isomorphic (\cite[Theorem 3]{Mitacs2012}). However, even for (separable) AF-algebras the relation of elementary equivalence is `smooth' while the isomorphism of AF-algebras is a `non-smooth' relation (it is classifiable by countable structures~\cite{FaToTo:Descriptive}). Hence, there are non-isomorphic elementarily equivalent unital AF-algebras whose asymptotic sequence algebras, under $\CH$, as well as their ultrapowers are isomorphic (see~\cite{eagle2015saturation, games} for examples of such AF-algebras).

We finish this section by another application of the metric Feferman--Vaught theorem which was used in~\cite{ghasemi2016reduced} to provide examples of `genuinely' noncommutative \cstar-algebras $A$ and $B$ such that the existence of an isomorphisms between them is independent from $\ZFC$.

\begin{example}\label{Ex.FV}
Fix a sequence $(A_n)_{n \in \bbN}$ of pairwise non-isomorphic unital separable \cstar-algebras (for example, each $A_n$ could be $M_n(\bbC)$).
 Suppose $\mathcal L$ is the language of \cstar-algebras and fix a countable uniformly dense set $\{\psi_i\mid i\in \bbN\}$ of $\mathcal L$-sentences. That is, for every $\varepsilon>0$, for every $\mathcal L$-sentence $\theta$ and for every $\mathcal L$-structure $A$ there is $\psi_i$ such that 
\[
|\psi_i^A - \theta^A|<\varepsilon.
\]
The set $\{\psi_i\mid i\in \bbN\}$ exists by~\cite[Proposition 6.6]{BYBHU}. Since the range of each $\psi_i$ is a bounded subset of $\mathbb R$, by compactness arguments we can find a sequence of natural numbers $(k(n))_{n\in \bbN}$ such that every $i$ we have $\lim_n \psi_i^{A_{k(n)}} = r_i$ for some $r_i\in \bbR$. As $\psi_i$ are uniformly dense, for every $\mathcal L$-sentence $\theta$ there is a real number $r_\theta$ such that 
\[
\lim_n \theta^{A_{k(n)}}= r_\theta. \leqno{(*)}
\]
Suppose $X=\{g(n) \mid n\in \bbN\}$ and $Y=\{h(n) \mid n\in \bbN\}$ are disjoint infinite subsets of $\{k(n) \mid n\in \bbN\}$. Let
\[
A:=\prod_n A_{g(n)}/\bigoplus_{\Fin} A_{g(n)} \quad \text{and} \quad B:=\prod_n A_{h(n)}/\bigoplus_{\Fin} A_{h(n)}
\]

For every $\mathcal L$-sentence $\theta$ from $(*)$ we have $\lim_n \theta^{A_{g(n)}}= \lim_n \theta^{A_{h(n)}}$ and therefore Theorem~\ref{FV-coro} implies that $A\equiv B$.
Since $A$ and $B$ are countably saturated (Theorem~\ref{ctble-sat-reduced}) CH implies that $A \cong B$. However, $X\cap Y= \emptyset$, which means $A_{g(n)} \not \cong A_{h(m)},$ for all $n,m\in \bbN$. This shows that any isomorphism between $A$ and $B$, constructed under CH, cannot be algebraically trivial (see Corollary~\ref{C.trivial-matrix}). 
%In fact, if in this example $A_n\cong M_n(\bbC)$ then it is consistent with $\ZFC$ that $A =\cM_X$ and $B=\cM_Y$ (this notation was introduced in~\ref{Eq.MX}) are not isomorphic (Theorem~\ref{thm-FDD}). Therefore, the existence of an isomorphism between $\cM_X$ and $\cM_Y$ is independent from $\ZFC$.
\end{example}

\subsection{Full saturation: The dividing line}\label{S.6.2}
In Example~\ref{Ex.F2} we presented a counterexample to the second part of Conjecture~\ref{Meta.1}. The gist of this example is that the reduced product $(\prod_n \bbZ/2\bbZ)/(\bigoplus_\bbN\bbZ/2\bbZ)$ is fully saturated, regardless of whether CH holds or not. The fact that this does not apply to $\cP(\bbN)/\Fin$, the grandfather of all structures that satisfy the second part of Conjecture~\ref{Meta.1}, is a classical result of Hausdorff who in \cite{Hau:Summen} famously constructed an $(\omega_1,\omega_1)$-gap showing that $\mathcal P(\bbN)/\Fin$ is not an $\aleph_2$-saturated Boolean algebra.\footnote{Hausdorff’s motivation for constructing such gap was different: at the time it was considered as a partial result towards proving CH in ZFC.} By Theorem~\ref{ctble-sat-reduced} every reduced product associated with $\Fin$ is countably saturated. Are there other possibilities for saturation of reduced products associated with $\Fin$? 
 In Question~10.1 in the original draft of the present paper we asked a vague question along these lines, answered by \cite[Theorem~1 and Theorem~2]{debondt2023saturation} as follows (`stable' refers to the model-theoretic notion, see e.g., \cite{pillay2008introduction}). 

\begin{thm} \label{T.dividing-line} Suppose that $M$ is a reduced product over $\Fin$ of structures in the same language. Then
\begin{enumerate}
\item\label{Sat1} If its theory is stable, then $M$ is $\mathfrak c$-saturated. 
\item\label{Sat2} If its theory is not stable, then $M$ is not $\aleph_2$-saturated, and therefore, if $\CH$ fails, not saturated.
\end{enumerate}
\end{thm}

A result analogous to \ref{Sat1} for ultrapowers appears in \cite[Theorem 5.6]{FaHaSh:Model2}, and the original proof of \ref{Sat1} closely followed the proof of the latter. This proof was considerably simplified by using ideas of Palyutin and Palmgren (see \S\ref{S.PP}). 
Item \ref{Sat2} is proven by injecting $\cP(\bbN)/\Fin$ into the quotient in a gap-preserving manner. 
Its proof applies in a wider context, showing for example that if $M$ is a reduced product over the asymptotic density zero ideal $\cZ_0$ (see \S\ref{ss.borel_quotients}), or any other $\cI$ such that $\cP(\bbN)/\cI$ is not countably saturated, whose theory is not stable then $M$ is not even countably saturated (\cite[Corollary~2.3]{debondt2023saturation}).

\subsection{The failure of countable saturation}\label{S.6aii}

In hindsight, the two factors that make the construction of nontrivial automorphisms of $\cP(\bbN)/\Fin$ so easy (by today's standards) are its countable saturation and the elimination of quantifiers in atomless Boolean algebras. The latter bit can be easily dispensed with: In a structure whose theory does not admit elimination of quantifiers, one just needs to consider only partial isomorphisms that are elementary. On the other hand, saturation makes all the difference. In hindsight, it is not very surprising that, unlike `commutative' quotients, the Calkin algebra and many other coronas are not countably saturated. Known instances of the failure of countable saturation in coronas of $\sigma$-unital \cstar-algebras are associated to obstructions of K-theoretic nature. 

\subsubsection{The Fredholm index obstruction} The simplest example of the failure of countable saturation in the Calkin algebra is related to the original Brown--Douglas--Fillmore question, whether there is an automorphism of $\cQ(H)$ that sends $\dot s$ (where $s$ is the unilateral shift) to its adjoint (see Question~\ref{ques:bdf}). This is really a question about orbits of unitaries in the Calkin algebra under the action of the automorphism group by conjugation. Every unitary $ u$ in $\cQ(H)$ lifts to a partial isometry (see \ref{3.partial.isometry}) between closed subspaces of~$H$ of finite co-dimension. The difference between the co-dimensions of the domain and the range of the lift is its Fredholm index (\S\ref{BDF}). The function that sends $ u$ to the Fredholm index of its lift is a group homomorphism from $\cU(\cQ(H))$ into $(\bbZ,+)$. This implies that the type of a unitary that (for example) has $2^n$th root for all $n$, but no $3$rd root, is consistent but not realised in $\cQ(H)$. (The proof uses definability of the set of unitaries, in the sense of~\cite{BYBHU}, see also~\cite[\S 3]{Muenster} to show that in $\cQ(H)$---and in any \cstar-algebra---applying quantification over the unitary group to a definable predicate results in a definable predicate). By adding a variable for each $2^n$th root, this example provides a universal type which is consistent but not realised in $\cQ(H)$. 

\subsubsection{Homotopy} The Calkin algebra fails even to be quantifier-free saturated, as we can see in the following examples. If $A$ and $B$ are unital \cstar-algebras, unital $^*$-homomorphisms $\Phi_0$ and $\Phi_1$ from $A$ into $B$ are \emph{homotopic} if there is a unital $^*$-homomorphism $\Psi\colon A\to C([0,1],B)$ such that $\Psi(a)(0)=\Phi_0(a)$ and $\Psi(a)(1)=\Phi_1(a)$ for all $a\in A$. 

\begin{example} \label{Ex.Mn} For $n\geq 2$, all unital copies of $M_n(\bbC)$ in $\cQ(H)$ are unitarily equivalent, but there are $n$ homotopy classes of unital copies of $M_n(\bbC)$ only one of which lifts to a unital copy in $\cB(H)$. 
	When $n=2$, this can be seen by considering the matrix unit $e_{12}=\begin{bmatrix}
		0 & 1 \\ 0 & 0 
	\end{bmatrix}$. It can be lifted to a partial isometry (see \ref{3.partial.isometry}) $v$ in $\cB(H)$ such that $v^*v$ and $vv^*$ are orthogonal projections and $1-v^*v-vv^*$ has finite rank. The parity of this rank is a homotopy invariant. 
\end{example}

By iterating the construction of Example~\ref{Ex.Mn} one obtains the following, due to N.C. Phillips (\cite[Proposition~4.2]{FaHa:Countable}); recall the convention that $\dot p$ denotes a lift of $p$. 

\begin{example}\label{Ex.CAR}
There is a unital copy $A$ of the CAR algebra $\bigotimes_\bbN M_2(\bbC)$ in $\cQ(H)$ such that for every projection $p \in \cB(H)$ with $\dot p\in A$, none of the tensor factors $M_2(\bbC)$ lifts to a unital copy of $M_2(\bbC)$ in $p\cB(H) p$ (for more see~\cite{Scho:Pext}). Fix a unital copy $B$ of the CAR algebra that lifts to a unital copy of the CAR algebra in $\cB(H)$ and an isomorphism $\Phi\colon A\to B$. The type of a unitary that implements this isomorphism is, by Example~\ref{Ex.Mn}, consistent, but it is not realised in $\cQ(H)$. 
\end{example}

The Galois type of a tuple (i.e., its orbit under the action of the automorphism group) provides more information than its first-order type. Example~\ref{Ex.Mn} shows that in \cstar-algebras, the Galois type of a tuple needs to be supplemented with its homotopy type. There is an alternative way of defining the homotopy type of a tuple in a \cstar-algebra. For a definable set $X$ (such as the set of unitaries or the set of projections; see~\cite[\S 3]{Muenster}) one says that two elements of $X$ are homotopic if they can be connected by a continuous path inside $X$. This approach is used in defining K-theory groups of \cstar-algebras (see~\cite{RoLaLa:Introduction}), and we will take advantage of it in the next subsection.

\subsubsection{Countable homogeneity} In the construction of automorphisms, countable saturation can be replaced by a weaker condition. A structure $C$ is \emph{countably homogeneous} if for every isomorphism $\Phi$
%\marginpar{SG: Isn't it enought that $\Phi$ is an elementary embedding, or it's the same? AVi: an elementary embeding is an isomorphism on its image.}
 between separable elementary submodels $A$ and $B$ of $C$, a type over $A$ is realised in $C$ if and only if its $\Phi$-image is realised in $C$. Countable homogeneity implies that a structure in a countable language of density $\aleph_1$ has $2^{\aleph_1}$ automorphisms, by using $\sigma$-complete back-and-forth systems as in~\S\ref{S.CHModelTheory}. 

By~\cite[Theorem~1]{farah2017calkin}, $\cQ(H)$ is not countably homogeneous. Readers not used to continuous logic may find the fact that the offender is a quantifier-free 2-type consisting of one condition amusing. Surprisingly, in~\cite[Theorem~2]{farah2017calkin} it is shown that even the subgroup of $\cU(\cQ(H))$ consisting of the unitaries of Fredholm index 0 taken with the metric induced by the norm is not countably homogeneous. 

Another failure of countable saturation deserves mention. We will state it in the context of \cstar-algebras. If $A$ is a separable \cstar-algebra and~$\cU$ is a nonprincipal ultrafilter on $\bbN$, then both the (norm) ultrapower~$A_\cU$ and the asymptotic sequence algebra $A_\infty=\ell_\infty(A)/c_0(A)$ are countably saturated (\S\ref{S.CHModelTheory}). Let $\pi_\cU$ denote the quotient map from $A_\infty$ onto $A_\cU$. A question on the relation between $A_\infty$ and $A_\cU$, motivated by the recent progress on Elliott's classification programme of \cstar-algebras, brought the two-sorted structure $\fA=(A_\infty, A_\cU,\pi_\cU)$ into the focus (see~\cite{farah2019between}). At the hindsight, it is not too surprising that the countable saturation of~$\fA$ is equivalent to the ultrafilter $\cU$ being a P-point (\cite[Theorem~C]{farah2019between}). 

Despite all the failures of saturation considered so far, the following question is still open.
\begin{question} Is there a simple, separable \cstar-algebra $A$ such that $\cQ(A)$ is fully countably saturated, or even quantifier-free countably saturated? 
\end{question}

The remainder of this subsection is concerned with a weakening of saturation strong enough to explain the `massiveness' of coronas of $\sigma$-unital \cstar-algebras.

\subsubsection{Degree-1 saturation}\label{S.6aii.5}
Despite the failure of countable saturation, coronas of $\sigma$-unital \cstar-algebras satisfy a weaker version of countable saturation called `countable degree-1 saturation'. 
A \emph{degree-1 type} $t(\bar x)$ over a \cstar-algebra $C$ is a set of conditions of the form $\|P(\bar x)\|= r$ with noncommuting variables $\bar x=(x_0, \dots,x_{n-1})$, where $P(\bar x)$ is a $^*$-polynomial of the form 
\[
\textstyle \sum_{j<n} a_{0,j} x_j a_{1,j} + \sum_{j<n} a_{2,j} x^*_j a_{3,j}+a 
\]
with coefficients $a_{i,j}$ in $C$ and $r\in \mathbb R_+$. 
For a degree-1 type $t(\bar x)$ over $C$ the notions of \emph{satisfiablity} and \emph{realisation} in $C$ are defined similar to the general types (\S\ref{S.CHModelTheory}), replacing arbitrary conditions with degree-1 conditions. A \cstar-algebra $C$ is \emph{countably degree-1 saturated} if every satisfiable countable degree-1 type over $C$ is realised in $C$. 
\begin{thm}[\cite{FaHa:Countable}]\label{T.degree-1}
The coronas of $\sigma$-unital non-unital \cstar-algebras are countably degree-1 saturated.
\end{thm}
 The countable degree-1 saturation has been shown to unify the proofs of several properties of coronas of $\sigma$-unital \cstar-algebras (see~\cite{FaHa:Countable}).
Infinite-dimensional countably degree-1 saturated \cstar-algebras have density character at least $\mathfrak c$ and possess closure properties resembling those of von Neumann algebras. For instance, if $C$ is a countably degree-1 saturated \cstar-algebra then there are no $(\aleph_0, \aleph_0)$-gaps in $C_+$, or in operator algebraic terms $C$ has the \emph{countable Riesz separation property} (CRISP, for short). %That is, if $a_n$, $b_n$, for $n\in \bbN$, are in $C_+$ such that $a_n\leq a_{n+1} \leq b_n \leq b_{n+1}$ for all $n$, then there is $c$ in $C_+$ which satisfies $a_n \leq c \leq b_n$, for every $n$. 
Additionally, every countably degree-1 saturated \cstar-algebra is an SAW$^*$-algebra: For any two orthogonal elements $a$ and $b$ in $C_+$ there exists $c$ in $C_+$ such that $ac=a$ and $bc=0$. 
%\begin{enumerate}
 % \item $C$ has AA-CRISP (asymptotically abelian countable Riesz separation property): if $a_n$, $b_n$, for $n\in \bbN$, are in $C_+$ such that $a_n\leq a_{n+1} \leq b_n \leq b_{n+1}$ for all $n$, and $D$ is a separable subalgebra of $C$ such that $\lim_n \|[a_n, d]\|=0$ for all $d\in D$, then there is $c\in D'\cap C_+$ which satisfies $a_n \leq c \leq b_n$, for every $n$. 
%\end{enumerate}
The class of SAW$^*$-algebras was introduced by Pedersen in~\cite{pedersen1986saw} as noncommutative analog of sub-Stonean spaces, where it was shown that the corona algebra of any $\sigma$-unital non-unital \cstar-algebra is an SAW$^*$-algebra. SAW$^*$-algebras have CRISP (\cite{Pede:Corona}) and they are not isomorphic to a tensor product of two infinite-dimensional \cstar-algebras (\cite{Gha:SAW*}). The latter implies that, for example, the Calkin algebra is not isomorphic to the tensor product of the Calkin algebra with itself, which suggests the existence of (at least to some extent) dimension phenomena associated to the Calkin algebra (see \S\ref{dim.phenomena}). The class of SAW$^*$-algebras is much larger than the class of countably degree-1 saturated algebras since every von Neumann algebra is an SAW$^*$-algebra, while no infinite-dimensional von Neumann algebra is countably degree-1 saturated. For a more complete list of properties of countably degree-1 saturated \cstar-algebras, see~\cite[\S 15]{Fa:STCstar}.

\subsection{Stratifying $\cQ(H)$}\label{S.Strat} As hinted after Corollary~\ref{cor:Ulamredprod}, in this section we will show how $\cQ(H)$ can be presented as an inductive limit of subspaces that resemble reduced products. This stratification will be used in \S\ref{S.Cohomology} to construct an outer automorphism of $\cQ(H)$ and again in \S\ref{6bii.PFAAVi} to prove that all of its automorphisms are inner.

%\begin{theorem}[{\cite[Theorem 6.1]{Fa:All}}]\label{thm:toptrivcalkin}\label{cor:autocalk}
%Every topologically trivial automorphism $\Phi$ of $\mathcal Q(H)$ is algebraically trivial, and therefore inner.
%\end{theorem}

%%%%%%%%%

On $\bbN^\bbN $, the relation 
\begin{equation}\label{Eq.leq*}
f\leq^*g \text{ if and only if } (\forall^\infty n) f(n)\leq g(n)
\end{equation}
is a quasi-ordering (i.e., it is transitive and reflexive, but not antisymmetric), where $\forall^\infty$ stands for the quantifier `for all but finitely many'.
We will need a noncommutative analog of $(\bbN^\bbN,\leq^*)$. %The \emph{dominating number} $\fd$ is the least cardinality of a cofinal subset of $(\bbN^\bbN,\leq^*)$. A diagonalisation argument shows that it is uncountable, and $\fd=\aleph_1$ is a weakening of $\CH$ quite useful for our purposes. 

 %%%%%
\begin{definition} \label{Def.Part} 
By $\Part$ we denote the set of all sequences $ \bfE=\langle E_j\colon j\in \bbN\rangle$ where $E_j=[n(j), n(j+1))$ (an interval in $\bbN$) and $n(j)$, for $j\in \bbN$, is a strictly increasing sequence in~$\bbN$.
 \end{definition}

Order $\Part$ by the only order that makes \eqref{Eq.EleqF} below hold: 
 \begin{equation}\label{Eq.Leq*}
\bfE\leq^* \bfF \text{ if and only if } (\forall^\infty n)(\exists m) E_m\subseteq F_n. 
 \end{equation}
 An easy diagonalisation argument shows that $(\Part,\leq^*)$ is $\sigma$-directed (for the aficionados of Tukey ordering, $(\Part,\leq^*)$ is cofinally equivalent to $(\bbN^\bbN,\leq^*)$ by \cite[Theorem~9.7.8]{Fa:STCstar}). 
% The poset $(\Part,\leq^*)$ is $\sigma$-directed and even Tukey equivalent to the more common $(\bbN^\bbN,\leq^*)$ (see~\cite[\S 9.7]{Fa:STCstar}). 

 Let $e_n$ be an orthonormal basis for $H$. If $n\in\bbN$, $q_n$ is the rank 1 projection onto $\mathbb C e_n$. If $\bfE\in \Part$, let $q_n^{\bfE}=\sum_{i\in E_n}q_i$, and if $X\subseteq\bbN$, let $q_X^{\bfE}=\sum_{n\in X} q_n^{\bfE}$. Consider the algebra 
 \begin{equation}\label{Eq.D[E]}
\mathcal D[\bfE]=\{a\in\mathcal B(H)\mid a\text{ commutes with }q_X^{\bfE}\text{ for all }X\subseteq\bbN\}.
\end{equation}
This is a \cstar-subalgebra of $\cB(H)$ that is moreover closed in the \emph{strong operator topology}, SOT, \footnote{SOT-closed self-adjoint subalgebras of $\cB(H)$ are called \emph{von Neumann algebras}, and their relation to \cstar-algebras is analogous to the relation of descriptive set theory to combinatorial set theory.} which is the topology of pointwise convergence on the unit ball of $H$. $\mathcal D[\bfE]$ is isomorphic to $\prod M_{|E_n|}(\bbC)$, and it equals all $a$ such that $a=\sum_{n\in\bbN} q_n^{\bfE}aq_{n}^{\bfE}$.

If $\bfE=\langle E_j\colon j\in\bbN\rangle\in\Part$, let
 \begin{equation}\label{evenodd}
\bfE^{\even}=\langle E_{2j}\cup E_{2j+1}\colon j\in\bbN\rangle\text{ and }\bfE^{\odd}=\langle E_{2j+1}\cup E_{2j+2}\colon j\in\bbN\rangle.
\end{equation}
Define (see Fig.~\ref{Fig.D[E]})
\[
\cF[\bfE]=\mathcal D[\bfE^{\even}]+\mathcal D[\bfE^{\odd}].
\]
Alternatively, fix $\bfE\in\Part$, and consider the finite-to-one function $f_\bfE\in\bbN^\bbN$ that collapses $E_n$ to $\{n\}$. Then 
 \begin{multline}\label{Eq.F[E]}
 \cF[\bfE]=\{a\in \cB(H)\mid (\forall m,n) f_\bfE(m)\geq f_\bfE(n)+2\\ 
 \text{ implies } q_n a(1-q_m)=(1-q_m)aq_n=0\}. 
 \end{multline}
This is a Banach subspace of $\cB(H)$, but not a subalgebra. 

 \begin{figure}[h]
 \begin{tikzpicture}[scale=.5, font=\scriptsize]
 \draw (0,0) -- (9,0);
 \draw (0,0) -- (0,-9);
 \draw[fill=white,draw=none] (0,0) rectangle (1,-1);
 \draw (0,0) rectangle (1,-1);
 \draw (1,-1) rectangle (3,-3);
 \draw (3,-3) rectangle (6,-6);
 \draw (6,-6) -- (6,-9);
 \draw (6,-6) -- (9,-6);
 \draw[dashed] (0,0) rectangle (.3,-.3);
 \draw[dashed] (.3,-.3) rectangle (1.7,-1.7);
 \draw[dashed] (1.7,-1.7) rectangle (4.5,-4.5);
 \draw[dashed] (4.5,-4.5) rectangle (8,-8);
 \draw[dashed] (8,-8) -- (8,-9);
 \draw[dashed] (8,-8) -- (9,-8);
 \foreach \x / \y in {1/.3,2/1,3/1.7,4/3,5/4.5,6/6, 7/8}
 {
 \node[anchor=east] at (0,-\y) {$n(\x)$};
 \draw (-.1,-\y) -- (0,-\y);
 }
 \foreach \z in {.3,.5,.7}
 \node at (6+\z,-6-\z) {$\cdot$};
 \end{tikzpicture}
 \caption{\label{Fig.D[E]} With $E_j=[n(j),n(j+1))$, for every $a$ in $\cF[\bfE]$ its support (i.e., the set $\{\{m,n\}: q_naq_m\neq 0\}$) is included in the union of the solid line square and the dashed line square regions.}
 \end{figure}
It can be verified that 
\begin{equation}\label{Eq.EleqF}
\bfE\leq^* \bfF \text{ if and only if }\pi[\cF[\bfE]]\subseteq\pi[\cF[\bfF]].
\end{equation}

 % (mimicking, for example, the proof of \cite[Theorem 3.1]{Ell:Derivations}),
A diagonal argument provides the desired stratification.

\begin{lemma}[{\cite[Proposition 17.1.2]{Fa:STCstar}}]\label{lemma:strat}
For every $a\in \cB(H)$ there is $\bfE\in \Part$ such that $\pi(a)\in \pi[\cF[\bfE]]$. Therefore $\textstyle \mathcal Q(H)=\bigcup_{\bfE\in\Part}\pi[\cF[\bfE]]$.
\end{lemma}

\subsection{Cohomology and automorphisms}\label{S.Cohomology}

At the first glance, the outline of the original construction of an outer automorphism of the Calkin algebra (\cite{PhWe:Calkin}) using $\CH$ is unsurprising. The automorphism is constructed recursively in $\aleph_1$ steps, assuring that its restriction to every separable \cstar-subalgebra is implemented by a unitary. (In fact, the construction produces~$2^{\aleph_1}$ distinct outer automorphisms, but this is not surprising either.) At a closer glance, this is one of the subtlest $\CH$ constructions known to the authors. As could be expected, the subtlety of this approach shows at the limit stages. At the first limit stage, one has an increasing sequence $A_n$, for $n\in \bbN$, of separable \cstar-subalgebras of $\cQ(H)$ and unitaries $u_n$ such that for $m<n$ the actions of $u_m$ and $u_n$ on $A_m$ agree. In order to continue the construction, one needs to choose a single unitary $u$ that implements the action of $u_n$ on $A_n$ for all $n$. The construction of two non-unitarily equivalent unital copies of the CAR algebra in Example~\ref{Ex.Mn} and Example~\ref{Ex.CAR} shows that this is in general impossible. The latter example gives a sequence of unitaries $u_n$ whose actions on a unital copy of the CAR algebra in $\cQ(H)$ cohere, but the Fredholm indices of $u_n$ converge to $\infty$ and $u_n$ cannot be uniformised by a single unitary. In~\cite{PhWe:Calkin}, the unitaries implementing partial isomorphisms are assumed to be homotopic. The construction of a unitary extending the paths constructed up to that point uses difficult ideas developed in relation to Kasparov's KK theory. 

A simpler---to a logician---construction of an outer automorphism of $\cQ(H)$ was described in~\cite[\S 1]{Fa:All} and extended to coronas of some other $\sigma$-unital \cstar-algebras in~\cite{CoFa:Automorphisms} (see~\cite[\S 17.1]{Fa:STCstar}). 

\begin{theorem} \label{T.d=aleph1} If $\fd=\aleph_1$, then the Calkin algebra has $2^{\aleph_1}$ automorphisms. Therefore if in addition $2^{\aleph_0}<2^{\aleph_1}$, then it has outer automorphisms. 
\end{theorem}

The proof relies on \S\ref{S.Strat}. Let $\bbT$ denote the multiplicative group of complex numbers of modulus 1. Identify $\bbT^\bbN$ with the unitary group of~$\ell_\infty(\bbN)$, and therefore with the group of unitaries in $\cB(H)$ diagonalised by a distinguished basis of $H$.

The main idea of the proof of Theorem~\ref{T.d=aleph1} employs the stratification of $\cQ(H)$ from \S\ref{S.Strat} and an inverse limit of quotients of the compact metric group $\bbT^\bbN$, also indexed by the poset $\Part$ (see Definition \ref{Def.Part}).

 %%%%%%%%%%%
 \subsubsection{An inverse limit}
 We introduce the components of the inverse limit that will be injected into the automorphism group of $\cQ(H)$. Let 
 \[
 \sfF_\bfE=\{u\in \bbT^\bbN\mid (\forall a\in \cF[\bfE])ua-au\in \cK(H)\}. 
 \]
By parsing the definition one sees that $\sfF_\bfE$ is the set of all $u$ that satisfy (writing $S'\cap \cQ(H)$ for $\{a\in \cQ(H)\mid (\forall b\in S) ab=ba\}$): 
 \[
 \pi(u)\in \pi[\cF[\bfE]]'\cap \cQ(H). 
 \]
 In other words, $u\in \sfF_\bfE$ if and only if the restriction of $\Ad u$ to $\cF[\bfE]$ is equal to the identity modulo the compacts. 
 
 For any fixed $\bfE$, $\sfF_\bfE$ is clearly a subgroup of $\bbT^\bbN$. To see that it is nontrivial, fix an increasing sequence $r_n$ in $\bbR_+$ such that $r_{n+1}-r_n\to 0$ but $\sum_n r_n=\infty$. Then $u$ defined by $u(j)=e^{ i r_n}$ if $j\in E_n$ is a nontrivial (i.e., non-scalar) element of~$\sfF_\bfE$. By choosing different sequences $(r_n)_{n \in \bbN}$, one can prove that~$\sfF_\bfE$ contains a perfect set.

%In the following definition $\cF[\bfE]$ is as in \S\ref{S.Strat} and we identify $\bbT^\bbN$ with the unitary group of $\ell_\infty$, and hence with a subgroup of $\cU(\cB(H))$. 

\begin{definition}\label{Def.Coherent} A family $\cF$ is a \emph{coherent family of unitaries} if its elements are pairs $(\bfE,u) \in \Part \times \bbT^\bbN$ such that 
\enumone 
\begin{enumerate}
\item The set $\cF_0=\{\bfE\mid (\exists u) (\bfE, u)\in \cF\}$ is directed in $(\Part,\leq^*)$. 
\item For every $\bfE\in \cF_0$ there is a unique $u=u_\bfE$ such that $(\bfE,u)\in \cF$. 
\item If $\bfE\leq^*\bfF$ are in $\cF_0$, then $\Ad \pi( u_\bfE) \restriction \pi[\cF[\bfE]] = \Ad \pi(u_\bfF) \restriction \pi[\cF[\bfE]]$ (equivalently, $u_\bfE u_\bfF^*\in \sfF_\bfE$). 	
\end{enumerate}
A coherent family $\cF$ is \emph{total} if $\cF_0$ is $\leq^*$-cofinal in $\Part$.
\end{definition}

The third condition in Definition~\ref{Def.Coherent} is the coherence requirement. In particular a coherent family $\cF$ naturally induces a well-defined map on $\bigcup_{\bfE \in \cF_0} \pi[\cF[\bfE]]$ (see Lemma~\ref{L.coherent} below). A coherent family of unitaries $\cF$ is \emph{trivial} if there exists a unitary $v \in \cQ(H)$ such that $\Ad v \restriction \pi[\cF[\bfE]] = \Ad \pi(u_\bfE) \restriction \pi[\cF[\bfE]]$ for all $(\bfE, u) \in \cF$.
 
 In~\cite{Fa:STCstar}, all coherent families are by definition total. Our modification of terminology was made in order to make it possible to say that an easy diagonalisation argument shows that every countable coherent family of unitaries is trivial. 
 
 Let $\sfG_\bfE=\bbT^\bbN/\sfF_\bfE$. For all $\bfE$ and $\bfF$ in $\Part$ the following three conditions are equivalent. 
 \begin{enumerate}
 	\item $\bfE\leq^*\bfF$. 
 	\item $\sfF_\bfE$ is a subgroup of $\sfF_\bfF$. 
 	\item \label{Coh.5.9} The identity map on $\bbT^\bbN$ lifts the quotient map from $\sfG_\bfF$ onto~$\sfG_\bfE$. 
 \end{enumerate}
% \begin{lemma} \label{L.diagonalisation.coherent}
% 	Suppose that $\bfE(n)\in \Part$ and $u_n\in \bbT^\bbN$ are such that $\bfE(n)\leq^* \bfE(n+1)$ and $u_n u_{n+1}^*\in \sfF_n$ for all $n\in \bbN$. Then there are $v_0$ and $v_1$ in $\bbT^n$ such that $u_n v_j^*\in \sfF_n$ for all $n$ and $j<2$ and $v_0-v_1\notin \cK(H)$. \qed 
% \end{lemma}
The inverse limit of quotients of $\bbT^\bbN$, 
$
 %	\label{Eq.projlim}
\varprojlim_{\bfE\in \Part} \sfG_\bfE, 
$ with the connecting maps as in \ref{Coh.5.9}
is the key to our construction. Its elements naturally correspond to total coherent families of unitaries. If there is a cofinal $\leq^*$-increasing family of length $\aleph_1$ in $\Part$ (i.e., if $\fd=\aleph_1$), then by using the triviality of countable coherent families and the nontriviality of $\sfF_\bfE$ for all $\bfE$, one easily produces $2^{\aleph_1}$ distinct coherent families of unitaries. An application of the following lemma completes the proof of Theorem~\ref{T.d=aleph1}. 

\begin{lemma} [{\cite[Lemma 17.1.4]{Fa:STCstar}}] \label{L.coherent} Every total coherent family $\cF$ of unitaries defines an automorphism $\Phi_\cF$ of $\cQ(H)$. This automorphism is outer if and only if the coherent family is nontrivial. 
\end{lemma}

\begin{proof} %If $u\in \ell_\infty$ and $a\in \cF[\bfE]$ then $uau^*\in \cF[\bfE]$. 
Note that $\{(\bfE,u^*)\mid(\bfE,u)\in \cF\}$ is a coherent family of unitaries. By coherence, both $\Phi_\cF(a)=(\Ad \pi(u))(a)$ if $(\bfE,u)\in \cF$ and $a\in \pi[\cF[\bfE]]$ and $\Psi(a)=(\Ad \pi(u^*))(a)$ if $(\bfE,u)\in \cF$ and $a\in \pi[\cF[\bfE]]$ are well-defined endomorphisms of $\cQ(H)$. Clearly $\Psi$ is the inverse of $\Phi_\cF$, and $\Phi_\cF$ is as required. 
\end{proof}

\subsubsection{Automorphisms of other coronas}
The flexibility of the construction used to prove Theorem \ref{T.d=aleph1} was explored in~\cite{CoFa:Automorphisms}. If $A$ is a $\sigma$-unital, non-unital, \cstar-algebra that has an approximate unit $(p_n)_{n \in \bbN}$ consisting of projections, then one can consider the analog of $\cF[\bfE]$ in $\cM(A)$ instead of $\cB(H)$. 

 With this modification, the proof of Theorem~\ref{T.d=aleph1} adapts to provide a group homomorphism from $\varprojlim_{\bfE\in \Part}\sfG_\bfE$ into $\Aut(\cM(A)/A)$. Automorphisms obtained in this way are not necessarily interesting. For example, if all projections $p_n$ are central, then $\cF_A[\bfE]$ is equal to $\cM(A)$ for some $\bfE$ and all automorphisms obtained in this way are inner. However, if $A=B\otimes \cK(H)$ for some unital \cstar-algebra $B$ (such $A$ is said to be a \emph{stabilisation} of $B$), then $\cF_A[\bfE]$ is a proper subspace of $\cM(A)$ for every $\bfE$, and distinct coherent families of unitaries map to distinct automorphisms of $\cM(A)/A$. 
 
% If $A$ is abelian then so is $\cM(A)$, and one clearly cannot build anything from its inner automorphisms. Finding nontrivial automorphisms in this case required new ideas, discussed in the following section. 

\subsection{Higher-dimensional \v{C}ech--Stone remainders}\label{S.CHRemainders}
So far we discussed how to construct topologically nontrivial automorphisms for coronas which are countably saturated (\S\ref{S.CHModelTheory}), or coronas of stabilisations of unital separable \cstar-algebras (\S\ref{S.Cohomology}).

There are though plenty of algebras to which the arguments we exposed so far do not apply. A good example is $C_0(\bbR^2)$.\footnote{\label{footnote.hart}The corona $C_0(\mathbb R)$ is countably saturated (see~\cite{FaSh:Rigidity}) and therefore it has $2^{\mathfrak c}$ automorphisms that are not topologically trivial under $\CH$. These were also constructed, via topological methods, by Yu, see~\cite{Ha:Cech-Stone}.} Its corona is not saturated (\cite[Corollary 2]{farah2023obstructions}) and, being abelian, its only inner automorphism is the identity, making
the methods of \S\ref{S.Cohomology} completely ineffective.

The following result covers this and other abelian \cstar-algebras whose coronas are not countably saturated.

\begin{theorem}[\cite{vignati2017nontrivial}]\label{T.CHAbel}
Let $X$ be a locally compact noncompact second countable manifold which is not zero-dimensional. If $\mathfrak d=\aleph_1$, then $C(\beta X\setminus X)$ has $2^{\aleph_1}$ automorphisms.
\end{theorem}

The above gives the following immediate corollary.

\begin{corollary}
Assume $\mathfrak d=\aleph_1$ and $2^{\aleph_0}<2^{\aleph_1}$. Let $X$ be a locally compact noncompact second countable manifold. Then $C(\beta X\setminus X)$ has topologically nontrivial automorphisms.
\end{corollary}
%\begin{proof}
%Since $\CH$ is assumed then $2^{\aleph_0}<2^{\aleph_1}$. Since there can only be $2^{\aleph_0}$ automorphisms admitting a Borel lifting, Theorem~\ref{T.CHAbel} gives nontrivial automorphisms. 
%\end{proof}
\begin{proof}[A sketch of the proof of Theorem~\ref{T.CHAbel}]
Here we give a sketch of the argument in case $X=\mathbb R^2$. %By Gelfand duality, automorphisms of $C(\beta \mathbb R^2\setminus\mathbb R^2)$ correspond to autohomeomorphisms of $\beta \mathbb R^2\setminus\mathbb R^2$. 
With the definition of the quasi-order $\leq^*$ on $\bbN^\bbN$ in mind (see \eqref{Eq.leq*}), by using $\mathfrak d=\aleph_1$ fix a $\leq^*$-increasing and $\leq^*$-cofinal sequence of functions $f_\alpha$, for $\alpha<\aleph_1$, in $\bbN^\bbN$. 
A further technical requirement is that, for all $\alpha<\aleph_1$ and $n\in\bbN$, 
\begin{equation}\label{eqabelCH}
nf_\alpha(n)\leq f_{\alpha+1}(n).
\end{equation}
These functions are used to define homeomorphisms of $\beta\mathbb R^2\setminus\mathbb R^2$ in the following way: Let $X_n$ be the ball around $(0,3n)\in\mathbb R^2$ of radius $1$. For $m\in \bbN$, let $\varphi_{n,m}$ be a homeomorphism of $X_n$ fixing the boundary of $X_n$ and with the property that 
\[
\sup_{x\in X_n}d(x,\phi_{n,m}(x))=\frac{1}{m}
\]
where $d$ is the usual distance on $\mathbb R^2$. Such homeomorphisms can be obtained, for example, by rotating the circle $\{x\in X_n\mid d(x,(0,3n))=\frac{1}{2}\}$. As all the $X_n$'s are disjoint and each $\varphi_{n,m}$ does not move points of $\overline{\mathbb R^2\setminus X_n}$, for every sequence $(m_n)_{n\in \bbN}$, the map $\bigcup\varphi_{n,m_n}$ is an autohomeomorphism of $\mathbb R^2$ that canonically induces a $\Phi((m_n)_n)\in \Homeo(\beta \bbR^2\setminus \bbR^2)$. 
Let 
\[
\Phi_\alpha=\Phi((f_\alpha(n))_n),
\]
and let $\tilde\Phi_\alpha$ be the automorphism of $C(\beta\bbR^2\setminus\bbR^2)$ dual to $\Phi_\alpha$ (\S\ref{ss:pre}). 

For every branch $p$ of the complete binary tree of height~$\aleph_1$, we construct an automorphism $\tilde\Psi^p$ by composing exactly those $\tilde\Phi_\alpha$ for which $p(\alpha)=1$. This composition is well-defined because of the following strong form of coherence of $\tilde\Phi_\alpha$'s. 

\begin{claim}
There is an increasing (under the inclusion) sequence $D_\alpha$, for $\alpha<\aleph_1$, of unital \cstar-subalgebras of $C(\beta\bbR^2\setminus\bbR^2)$, such that 
\[
\textstyle C(\beta\bbR^2\setminus\bbR^2)=\bigcup_{\alpha<\aleph_1}D_\alpha
\]
and the restriction of $\tilde\Phi_{\alpha}$ to $D_\beta$ is equal to the identity if and only if $\beta \leq \alpha$.
\end{claim}

The $D_\alpha$ consists of functions which are more and more Lipschitz when restricted to $X_n$ (see for details~\cite[\S2]{vignati2017nontrivial}).

Fix $p\in2^{\aleph_1}$ and let $\tilde\Psi^p_0=\mathrm{Id}$. If $p(\alpha)=1$ and $\tilde\Psi^p_\alpha$ have been defined, let 
\[
\tilde\Psi^p_{\alpha+1}=\tilde\Phi_\alpha\circ \tilde\Psi^p_{\alpha}.
\]
Note that $\tilde\Psi^p_{\alpha+1}$ and $\tilde\Psi^p_{\alpha}$ agree on $D_\alpha$. At limit stages, a diagonalisation argument shows that indeed if $\beta$ is a limit ordinal, the map 
 \[
\textstyle\tilde\Psi^p_\beta= \bigcup_{\alpha<\beta}\tilde\Psi^p_\alpha
\]
 is a well-defined automorphism of $C(\beta\mathbb R^2\setminus\mathbb R^2)$. This is the coherence we need: for $p\in 2^{\aleph_1}$, set
 \[
\textstyle \tilde\Psi^p=\bigcup_{\alpha<\aleph_1} \tilde\Psi^p_\alpha.
 \]
It is routine to check that this is an automorphism of $C(\beta\mathbb R^2\setminus\mathbb R^2)$, and that if $p\neq q$, then $ \tilde\Psi^p$ and $ \tilde\Psi^q$ are different. This concludes the argument.
\end{proof}

A variant of the proof of Theorem~\ref{T.CHAbel} applies to a class of topological spaces called \emph{flexible}, whose technical definition asks for the existence of a variation of the sets $X_n$ and the homeomorphisms $\varphi_{n,m}$ given as above. All manifolds of nonzero topological dimension are flexible, but the converse is not true. Furthermore, by tensoring with a unital separable \cstar-algebra $A$, the above argument gives the existence of topologically nontrivial automorphisms of coronas of algebras of the form $C_0(X,A)$, for a flexible space $X$ and a unital \cstar-algebra $A$. 

\subsection{Other models for non-rigidity}\label{S.OtherModels}
%\input{6av.othermodelsIF}
%\section{Other models for non-rigidity}\label{S.OtherModels}

Due to their complexity, the questions of the existence of nontrivial automorphisms of $\cP(\bbN)/\Fin$, the Calkin algebra and other quotient structures, are open in `most' models of $\ZFC$.
%In this section we will explore the models in which these structures have nontrivial automorphisms. 
 
The success of CH in constructing nontrivial automorphisms and isomorphisms of Borel quotient structures has a satisfactory metamathematical explanation. As discussed in \S\ref{S.CHModelTheory}, part of the reason comes from countable saturation, nevertheless, even in this most general setting, Woodin's $\Sigma^2_1$-absoluteness Theorem provides a strong justification for the effectiveness of $\CH$ in solving these problems (we will return to this point in \S\ref{S.Absoluteness}). 

What interests us in this section is the dual problem: In what models of the negation of the Continuum Hypothesis do some quotient structures have nontrivial automorphisms?\footnote{The reason for avoiding the related question, 
asking in what models of the negation of the Continuum Hypothesis do nontrivial isomorphisms between some quotient structures exist, is that for \cstar-algebras we do not have any answers. No example of a pair of simple, separable, \cstar-algebras whose coronas are isomorphic for a nontrivial reason is known, even when the $\CH$ is assumed (see Example~\ref{Ex.FV} for the best known result in this direction).} 

This question has been resolved for $\cP(\bbN)/\Fin$ in Cohen's model for the negation of CH, obtained by adding $\aleph_2$ Cohen reals to a model of $\CH$ (\cite{ShSte:Non-trivial}). In this generic extension, $\cP(\bbN)/\Fin$ arises as the union of an increasing $\aleph_2$-chain of the algebras $\cP(\bbN)/\Fin$ in the intermediate forcing extensions, denoted $B_\alpha$, for $\alpha<\aleph_2$. The salient observation is that for every new $X\subseteq \bbN$, both sets $L_X=\{Y\in \cP(\bbN)^V\mid Y\subseteq^* X\}$ and $R_X=\{Y\in \cP(\bbN)^V\mid Y\subseteq^* \bbN\setminus X\}$ have countable cofinalities. Therefore, if $\Phi\in \Aut(\cP(\bbN)/\Fin)$, then $\Phi[L_X]$, $\Phi[R_X]$ can be split by a (new) Cohen real $C\subseteq \bbN$. If the real is sufficiently generic, then the restriction of $\Phi$ to $B_\alpha$ can be extended by sending $[X]$ to $[C]$, and a standard bookkeeping device produces $2^{\aleph_2}$ nontrivial automorphisms of $\cP(\bbN)/\Fin$. The full saturation of $B_\alpha$ is essential for this construction, and it is not known whether there are nontrivial automorphisms of $\cP(\bbN)/\Fin$ in a model obtained by adding more than $\aleph_2$ Cohen reals to a model of $\CH$.

The idea of representing $\cP(\bbN)/\Fin$ in this way was elucidated further in~\cite{steprans1992topological} and~\cite{dow2002applications}.
It is an open problem whether analogous results hold for random reals, or any other standard model of $\mathfrak c=\aleph_2$ obtained by the countable support iteration of a single Suslin forcing (see~\cite{Zap:Descriptive} for more on such models). 

Finally, by a result of Baumgartner reproduced in~\cite[p. 130]{Ve:Definable}, the existence of an $\aleph_1$-generated P-point ultrafilter $\cU$ implies the existence of a nontrivial automorphism of $\cP(\bbN)/\Fin$ that fixes $\cU$.
On the other hand, an automorphism of $\cP(\bbN)/\Fin$ is \emph{nowhere trivial} if its restriction to $\cP(A)/\Fin$ is not trivial for any infinite $A\subseteq \bbN$ (the restriction of Baumgartner's automorphism to any set in $\cU$ is trivial).
It is possible to construct such automorphisms from CH, and in~\cite{shelah2015non} (see also \cite{shelah2025revised}) they were constructed from an assumption reminiscent of~$\fd=\aleph_1$.

\section{Independence results, II. Rigidity}
\label{S.Independence2}
\enumthree 
In this section we study the `other side of the coin' and survey the existing rigidity results for Borel quotients. There are two kinds of such results. Some of them are obtained in forcing extensions invented specifically for the occasion (\S\ref{S.Shelah}, \S\ref{S.6bi}), while the others are consequences of forcing axioms (\S\ref{S.FABoole}, \S\ref{6bii.PFAAVi}). The question of whether a given topologically trivial isomorphism is algebraically trivial, studied in \S\ref{S.Ulam}, is in reasonable categories absolute between models of ZFC (see \S\ref{S.Absoluteness}). Therefore in this section we concentrate on the question of whether all isomorphisms are topologically trivial. 

\subsection{Shelah's proof} \label{S.Shelah} This proof (first appearing in \cite[\S IV]{Sh:Proper}, then in \cite[\S IV]{Sh:PIF}, and in an improved form, allowing $\fc$ to be arbitrarily large, in 
 \cite{Just:Modification}) started the entire subject that this survey is about. Since several features of Shelah's proof still, over forty years later, resonate throughout the subject, we take the opportunity to sketch it. 

Some standard terminology: Subsets of $\bbN$ are said to be almost disjoint if they are disjoint modulo $\Fin$, and we write $a=^*b$ if the symmetric difference $a\Delta b$ is finite. 

Suppose that CH holds and $\bbP_{\aleph_2}=(\bbP_\alpha,\dot\bbQ_\alpha)_{\alpha<\aleph_2}$ is a finite-support iteration of ccc posets, each of cardinality $\aleph_1$. Suppose that $G\subseteq \bbP_{\aleph_2}$ is a generic filter. A standard reflection argument shows that if $\Phi$ is an automorphism of $\cP(\bbN)/\Fin$ in $V[G]$ then the set of $\alpha<\aleph_2$ such that the restriction $\Phi_\alpha$ of $\Phi$ to $(\cP(\bbN)/\Fin)^{V[G\cap \bbP_\alpha]} $ is an automorphism of this structure includes a \emph{club} (i.e. closed unbounded set), and that if $\Phi$ is nontrivial then the set of $\alpha$ such that $\Phi_\alpha$ is nontrivial in $V[G\cap \bbP_\alpha]$ also includes a club. 

Thus, in order to force that all automorphisms are inner, it suffices to construct an iteration so that, for every potential offender $\Phi$, its reflection $\Phi_\alpha$ is intercepted and nipped in the bud stationarily often. The interception device is $\diamondsuit_{\aleph_2}$ on cofinality $\aleph_1$ ordinals. The `nipping in the bud' technique is more exciting. 

\begin{lemma} \label{L.SaharonMainLemma} 
If $\CH$ holds and $\Phi_*$ is a lifting of a nontrivial automorphism $\Phi$ of $\cP(\bbN)/\Fin$, then
\begin{enumerate}
\item there are $A_\alpha \subseteq B_\alpha\subseteq \bbN$ for $\alpha<\aleph_1$ such that $B_\alpha$ are almost disjoint,
\item there are a forcing $\bbT=\bbT_{\bar B, \bar A}$, and a $\bbT$-name $\dot X$ for a subset of $\bbN$ such that $\bbT$ forces $\dot X\cap B_\alpha\in \{A_\alpha, B_\alpha\setminus A_\alpha\}$ for all $\alpha$,
\item \emph{for every further forcing $\dot \bbQ$ and every $\bbT*\dot\bbQ$-name $\dot Y$} it is forced that, for some $\alpha$, $\Phi_*(\dot X\cap B_\alpha)\neq^* \dot Y\cap \Phi_*(B_\alpha)$. 
\end{enumerate}
\end{lemma}

The phrase `every further forcing' requires an explanation. For this proof, Shelah introduced the notion of \emph{oracle-c.c.}---a class of particularly mild forcings with the property that every real added by an oracle-c.c. forcing exists in an intermediate forcing extension obtained by adding a Cohen real over the universe. The iteration $\bbP_{\aleph_2}$ is a finite-support iteration of oracle-c.c. forcings and thus it `only' remains to prove Lemma~\ref{L.SaharonMainLemma} with the simplification that $\dot\bbQ$ is the poset for adding a single Cohen real, but with the requirement that $\bbT_{\bar B, \bar A}$ is oracle-c.c.\footnote{`Whatever that means', an impatient reader may grumble.}

\begin{proof}[Sketch of the proof of Lemma~\ref{L.SaharonMainLemma}] 
Given $\alpha\leq \aleph_1$, almost disjoint sets $B_\alpha$, and sets $A_\beta \subseteq B_\beta\subseteq \bbN$ for $\beta<\alpha$, write $\bar B=(B_\beta: \beta<\alpha)$, and for $\gamma<\alpha$ write $\bar B\rs \gamma=(B_\beta: \beta<\gamma)$. We use a similar notation for $\bar A$ and $\bar A\rs \gamma$. Consider the forcing notion $\bbT_{\bar B, \bar A}$ defined as follows. Each condition $p$ is a partial function from a subset of $\bbN$ into $\{0,1\}$, the domain of $p$ is included modulo finite in $\bigcup_{\beta\in F(p)}B_\beta$, for a finite $F(p)\subset \alpha$, and $p^{-1}(\{1\})\cap B_\beta$ is equal modulo finite to $A_\beta$ or to $B_\beta\setminus A_\beta$ for all $\beta\in F(p)$. If $\alpha$ is countable, then so is $\bbT_{\bar B, \bar A}$. It is therefore forcing-equivalent to the forcing for adding a single Cohen real. A $\Delta$-system argument shows that even when $\alpha=\aleph_1$, this poset has ccc. It generically adds a function from $\bbN$ into $\{0,1\}$. With carefully chosen parameters, this function is the characteristic function of $\dot X$, as required.
 
Write $\bbT$ for a forcing of the form $\bbT_{\bar B, \bar A}$ if $\bar B$ and $\bar A$ are clear from the context, $\alpha(\bbT)$ for $\alpha$ which is the length of the sequence $\bar B$, and $\bbT(\gamma)$ for $\bbT_{\bar B\rs \gamma, \bar A\rs\gamma}$. The poset $\bbT(\gamma)$ is in general not a regular subordering of $\bbT$ (much of the oracle-c.c. machinery is devoted to mitigating this inconvenience; see footnote~\ref{Foot.oracle-cc}). Nevertheless, by ccc-ness, every $\bbT$-name for a real is countable (even when $\alpha$ isn't), and it is therefore a $\bbT(\gamma)$-name. 

Therefore by CH we can enumerate all forcings of the form $\bbT$ with $\alpha(\bbT)<\aleph_1$ and all names for a real in forcings $\bbT\times\bbQ$, where $\bbQ$ is the forcing for adding a single Cohen real, as $\dot Y_\alpha$, $\alpha<\aleph_1$. Assure that each name occurs cofinally often. 

One recursively finds $B_\alpha$ and $A_\alpha$, for $\alpha<\aleph_1$, assuring that, if $\dot Y_\alpha$ is a $\bbT(\alpha)\times \bbQ$-name, then $\bbT(\alpha+1)\times \bbQ$ forces that neither $\Phi_*(A_\alpha)$ nor $\Phi_*(B_\alpha\setminus A_\alpha)$ is equal to $\Phi_*(B_\alpha)\cap \dot Y_\alpha$ modulo finite. The choice of $B_\alpha$'s depends on the following :
\[
\cJ_\Phi=\{B\subseteq\bbN\mid \Phi \restriction \cP(B)/\Fin \text{ is algebraically trivial}\}. 
\]
This is an ideal, and we now know that it is equal to the ideal of all $B$ such that the restriction of $\Phi$ to $\cP(B)/\Fin$ is topologically trivial. (This fact, implicit in Shelah's proof, first appears explicitly in \cite{Ve:Definable}.) One considers the following two cases.
\begin{enumerate}
\item \label{1.J} $\cJ_\Phi$ is \emph{dense}, i.e., $\cJ_\Phi^\perp=\{C\subseteq \bbN\mid C\cap B\text{ is finite for all }B\in \cJ_\Phi\}$, is equal to $\Fin$. 
\item\label{2.J} There is an infinite set $B$ in $\cJ_\Phi^\perp$. 
\end{enumerate}
If \ref{1.J} applies, then one chooses $B_\alpha\in \cJ_\Phi$. If \ref{2.J} applies, then one carefully chooses $B_\alpha\in \cJ_\Phi^\perp$. The sets $B_\alpha$ have to be taken with an extra care, to assure that the poset $\bbT$ has the oracle-c.c..\footnote{\label{Foot.J}The construction is more complicated; for every $\alpha$ one chooses a partition $B_\alpha=C_1\cup C_2$ and $A_\alpha$ is included in one of these sets, but hiding this under the rug does not make much difference in this sketch, except at one subtle point later on.}

If this recursive construction can be completed, then $\Phi$ is `nipped in the bud' by $\bbT$, as desired. We proceed to analyse the situation in which the construction halts at stage $\alpha$. Assume for a moment that $B_\alpha$ has been fixed. For $A\subseteq B_\alpha$ write $\bbT[A]$ for $\bbT_{(\bar B^\frown B_\alpha, \bar A^\frown A)}$. 
If \ref{2.J} applies, then the set 
\[
Z=\{(A,C)\mid A\subseteq B, \forces_{\bbT[A]\times \bbQ} C=^*\dot Y_\alpha \cap \Phi_*(B_\alpha)\} 
\]
is Borel.\footnote{For the readers who enjoy reading complicated formulas: $(A,C)\in Z$ if and only if $(\forall p)(\exists q\leq p)(\exists n)(\forall r\leq q)(\forall m\geq n) r\forces m\in \dot Y\Leftrightarrow m\in C$, where $p,q$, and $r$ range over the countable poset $\bbT[A]\times \bbQ$.} Therefore $\Phi$ is topologically trivial, and algebraically trivial, on $\cP(B_\alpha)/\Fin$; contradiction.

We may therefore assume \ref{1.J}, hence $B_\alpha\in \cJ_\Phi$ as witnessed by an injection $g\colon B_\alpha\to \bbN$ such that $A\mapsto g[A]$ lifts $\Phi$ on $\cP(B)/\Fin$. Consider $\dot Y_\alpha$ as a $\bbT'=\bbT(\alpha)\times \bbQ$-name. For $p$ in $\bbT'$, the set $\dom(p)\cap B_\alpha$ is finite and (with a slight abuse of notation, harmlessly ignoring the $\bbQ$ component) for $m\in B_\alpha\setminus \dom(p)$ both $p\cup \{(m,0)\}$ and $p\cup \{(m,1)\}$ are conditions extending $p$. The key claim is that 
$(\forall p)(\exists q\leq p) (\forall^\infty m\in B_\alpha) $ 
\begin{equation}\label{Eq.pq}
q\cup \{(m,0)\}\forces g(m)\notin \dot Y_\alpha \text { and } q\cup \{(m,1)\}\forces g(m)\in \dot Y_\alpha. 
\end{equation}
Otherwise, one recursively finds a decreasing sequence of conditions in~$\bbT'$ that define $A\subseteq B_\alpha$ and whose union is a condition in $\bbT[A]$ that forces $\dot X\cap B_\alpha = A$ and $\dot Y_\alpha\cap \Phi_*(B_\alpha)=^*\Phi_*(A)$; contradiction. 

From $q$ as in \eqref{Eq.pq} one can extract a partial function $g_q$ whose restriction to $B_\alpha$ is equal to $g$ modulo finite (let's say that $g_q$ `works' on $B_\alpha$). This is true for every eligible candidate for $B_\alpha$. Since \eqref{Eq.pq} does not refer to $B_\alpha$, there is a countable set of functions, $g_q$, for $q\in \bbT'$, that between them serve as a lift of $\Phi$ on $\cP(B)$, for every $B\in \{B_\beta\mid \beta<\alpha\}^\perp\cap \cJ_\Phi$. 
The ideal $\cI=\{B_\beta\mid \beta<\alpha\}^\perp$ is included in $\cJ_\Phi$.\footnote{This is the white lie promised in footnote~\ref{Foot.J}.} As an orthogonal of a countable set, it is countably directed under $\subseteq^*$, and therefore a single function $g_q$ `works' on a cofinal set of $B\in \cI$. An easy diagonalisation argument shows that there is $k\in \bbN$ such that $g_q$ `works' on $\bbN\setminus \bigcup_{j<k} B_j$, hence this set belongs to $\cJ_\Phi$. 
Thus $\bbN$, as the union of $k+1$ sets in $\cJ_\Phi$, belongs to $\cJ_\Phi$ and $\Phi$ is trivial; contradiction. 

One can therefore recursively find $B_\alpha$ and $A_\alpha$, for $\alpha<\aleph_1$, such that the poset $\bbT$ adds $\dot X\subseteq \bbN$ but no $\dot Y\subseteq \bbN$ which satisfies \[
\Phi_*((\dot X\cap B_\alpha)=^*\Phi_*(\dot Y)\cap \Phi_*(B_\alpha)
\]
 for all $\alpha$, and this property is preserved by the remaining part of the iteration.\footnote{\label{Foot.oracle-cc} We are using the feature of oracle-c.c., that $\bbT(\alpha)\times\bbQ$ is a `sufficiently regular' subordering of $\bbT\times \bbQ$ to preserve the property of~$\dot Y_\alpha$.} 
\end{proof}

Lemma~\ref{L.SaharonMainLemma} implies that the only automorphisms that cannot be `nipped in the bud' are the trivial ones, hence we can construct a forcing extension in which all automorphisms of $\cP(\bbN)/\Fin$ are trivial. 

By using a modification of the oracle-c.c., Just proved that the assertion that all automorphisms of $\cP(\bbN)/\Fin$ are trivial is consistent with the continuum being arbitrarily large (\cite[Theorem~A]{Just:Modification}). More importantly for our story, he also proved that in his model $\cP(\bbN)/\cI$ cannot be embedded into $\cP(\bbN)/\Fin$ for any Borel ideal that is not generated by a single set over $\Fin$ (\cite[Theorem~0.6]{Just:Modification}).\footnote{Just's proof used a measurable cardinal, but this assumption has later been removed.} Just proved that in the same model every isomorphism between quotients over Borel ideals is topologically trivial on a large set (see \cite{Just:Repercussions} for a nice application). 

Having these statements hold in a model of ZFC constructed specifically for the purpose is not completely satisfactory,\footnote{This is clearly a personal opinion, but it is a very strong personal opinion.} and the question arose whether forcing axioms imply these conclusions. Which brings us to the next subsection. 

\subsection{Quotients of $\mathcal P(\bbN)$ and zero-dimensional remainders}\label{S.FABoole}
Shelah's proof outlined in \S\ref{S.Shelah} relies on creating a gap in $\cP(\bbN)/\Fin$ and then meticulously arranging for this gap not to be filled by the remaining part of the iteration. In~\cite{Ku:Gaps} Kunen introduced the method of `freezing gaps': arranging, by a ccc forcing, that a given gap cannot be filled without collapsing~$\aleph_1$. An infusion of Kunen's idea into Shelah's construction resulted in a proof that the Proper Forcing Axiom (PFA) implies all automorphisms of $\mathcal P(\bbN)/ \Fin$ are trivial (\cite{ShSte:PFA}). Around the same time, Veli\v{c}kovi\'c deduced the same conclusion from the conjunction of $\OCA$ (Definition~\ref{Def.OCA} below) and Martin's Axiom, $\MA_{\aleph_1}$ (\cite{Ve:OCA}). We take a moment to present Veli\v ckovi\'c ingenious open colouring, a variant of which appeared in every single rigidity result for quotient Boolean algebras and \cstar-algebras since. 

%\subsubsection{Veli\v ckovi\'c's Partition}
For a set $\cX$, let $[\cX]^2$ denote the set of unordered pairs of elements of $\cX$. If $\cX$ is a topological space, we say that $L\subseteq [\cX]^2$ is \emph{open} if $\{(x,y)\in \cX^2\mid\{x,y\}\in L\}$ is open. The following was isolated and proved to follow from $\PFA$ in \cite[\S 8]{To:Partition}. 

\begin{definition} \label{Def.OCA} $\OCA$ asserts that for every separably metrisable space $\cX$ and an open $L\subseteq [\cX]^2$ either there exists an uncountable $\cY\subseteq \cX$ such that $[\cY]^2\subseteq L$ or there is a partition $\cX=\bigcup_n \cX_n$ such that each $[\cX_n]^2$ is disjoint from $L$. 
\end{definition}

An easy, but far-reaching, observation is that if $[\cX_n]^2$ is disjoint from an open set $L$ then so is the topological closure of $\cX_n$. This is a part of the magic of $\OCA$ and its ability to turn arbitrary objects (such as liftings) into topologically trivial ones. 

Two more definitions before we can state Veli\v ckovi\'c's result. A lifting of a homomorphism $\Phi$ between analytic quotients is \emph{$\sigma$-Borel} if there are Borel functions $\theta_n\colon \cP(\bbN)\to \cP(\bbN)$, for $n\in \bbN$, such that for every $a\in \cP(\bbN)$ there is $n$ for which $\Phi([A])=[\theta_n(A)]$. A family $\cA$ of subsets of $\bbN$ is \emph{tree-like} if there is an injection $\chi\colon \bbN\to \twolo$ such that each $\chi[A]$, for $A\in \cA$, is included in a different branch, $\bfB(A)$, of $\twoo$. 

The following is an initial segment of \cite[Lemma~2.2]{Ve:OCA}. 

\begin{lemma} \label{L.Vel.2.2}
Assume $\OCA$. If $\Phi$ is an automorphism of $\cP(\bbN)/\Fin$ and $\cA$ is a tree-like family, then for all but countably many $A\in \cA$ the restriction of $\Phi$ to $\cP(A)$ has a $\sigma$-Borel lifting. 
\end{lemma}

\begin{proof}[Sketch of a proof]
Consider the subset of the Cantor space $(\{0,1\}^\bbN)^2$ defined by $\cX=\{(B,C)\mid C\subseteq B\subseteq A\text{ for some $A\in \cA$ and $B$ is infinite$\}$}$.
As $\cA$ is almost disjoint, for every $(B,C)\in \cX$ we have $B\subseteq A$ for a unique $A\in \cA$, and we set $\bfB(B)=A$. 
Fix a lifting $\Phi_*$ of $\Phi$ and identify $(B,C)\in \cX$ with $(B,C,\Phi_*(B), \Phi_*(C),\bfB(B))$ in the Cantor space $(\twoo)^5$. Consider~$\cX$ with the subspace topology and define $L\subseteq [\cX]^2$ by $\{(B,C), (B',C')\}\in L$ if the following conditions hold.
\begin{enumerate}
\item [(L1)] $\bfB(B)\neq \bfB(B')$. 
\item [(L2)] $B\cap C'=B'\cap C$. 
\item [(L3)] $\Phi_*(B)\cap \Phi_*(C')\neq \Phi_*(B')\cap \Phi_*(C)$.	
\end{enumerate}
This set is open (the set of pairs that satisfy (L2) is not open, but the set of pairs that satisfy both (L1) and (L2) is).

Assume that $\cY\subseteq \cX$ is uncountable and $[\cY]^2\subseteq L$. By (L1) and (L2), $Z=\bigcup\{ B\mid (B,C)\in \cY\}$ satisfies $Z\cap B=C$ for all $(B,C)\in \cY$. Therefore, just like in Lemma~\ref{L.SaharonMainLemma}, $\Phi_*(Z)\cap \Phi_*(B)=^*\Phi_*(C)$ for all $(B,C)\in \cX$. By the pigeonhole principle, fix $m$ such that  $(\Phi_*(Z)\cap \Phi_*(B))\Delta \Phi_*(C)\subseteq m$ for uncountably many $(B,C)\in \cY$. By another use of the pigeonhole principle, find subsets $s$ and $t$ of $m$ and a further uncountable $\cY_1\subseteq \cY$ such that $\Phi_*(B)\cap m=s$ and $\Phi_*(C)=t$ for all $(B,C)\in \cY_1$. Then $[\cY_1]^2$ is disjoint from $L$; contradiction. 
	
By $\OCA$, $\cX$ can be covered by \emph{closed} sets $\cX_n$ whose combinatorial squares are disjoint from $L$. Fix a countable dense $D_n\subseteq \cX_n$ for each $n$. An intricate argument shows that every $B\in \cA\setminus \{\bfB(A)\mid A\in \cA\}$ can be partitioned as $B=B_0\sqcup B_1$ such that the restriction of $\Phi$ to $\cP(B_j)$, for $j=0,1$ has a $\sigma$-Borel lifting. Since $\Phi$ is a Boolean algebra homomorphism, this implies that the restriction of $\Phi$ to $\cP(B)$ has a $\sigma$-Borel lifting. 
\end{proof}

In the following step of the proof that $\Phi$ is topologically trivial one shows that if $\Phi$ has a $\sigma$-Borel lifting on $\cP(A)$ and $A=\bigcup_n A_n$ then $\Phi$ has a Borel lifting (i.e., is topologically trivial) on some $\cP(A_n)$. An elegant proof of this using the Kuratowski--Ulam Theorem is due to Fremlin, given in a more general context in \cite[Lemma~3.12.3]{Fa:AQ}. This, together with Lemma~\ref{L.Vel.2.2} and $\MA$, shows that the ideal $\cJ_\Phi$ of $A\subseteq \bbN$ such that $\Phi$ is topologically trivial on $\cP(A)$ is large (e.g., it is nonmeager, but the details are not relevant for our purposes). 

%In~\cite{Fa:Luzin} the freezing gaps technique was extended to $\mathcal P(\bbN)/\mathcal I$ for certain $F_{\sigma\delta}$ ideals $\mathcal I$. 

Once again, it was Just who broadened the scope of the subject by considering ideals other than $\Fin$. In \cite{Just:WAT} he introduced the notions of a closed approximation and a $\cK$-approximation to an ideal. 

\begin{definition} \label{Def.Approximation}
Some $\cK\subseteq \cP(\bbN)$ is \emph{hereditary} if $A\in \cK$ implies $B\in \cK$ for all $B\subseteq A$. For subsets $\cA$ and $\cB$ of $\cP(\bbN)$ let 
\[
\cA\ucup \cB=\{A\cup B\mid A\in \cA, B\in \cB\}
\]
 A hereditary $\cK\subseteq \cP(\bbN)$ is an \emph{approximation} to an ideal $\cI$ if 
$\cK\ucup \Fin\supseteq \cI$. 
For $d\geq 1$ write $\cK^d=\cK\ucup \dots\ucup \cK$ ($d$ times).
\end{definition}

\begin{definition} \label{Def.K-approx}
 If $\Phi\colon \cP(\bbN)\to \cP(\bbN)/\cJ$ is a Boolean algebra homomorphism\footnote{Since no conditions are imposed on $\ker(\Phi)$, this covers the case when $\Phi$ is an automorphism of $\cP(\bbN)/\cJ$.} and $\cK\subseteq \cP(\bbN)$, then some $\theta\colon \cP(\bbN)\to \cP(\bbN)$ is a \emph{$\cK$-approximation to} $\Phi$ if $\theta(a)\Delta \Phi_*(a)\in \cK\ucup \Fin$ for all $a\in \cP(\bbN)$. If $\theta$ is Baire-measurable, then $\Phi$ is called \emph{$\cK$-sharp} and if $\theta$ is $\sigma$-Borel then $\Phi$ is called \emph{$\cK$-precise}. 
 \end{definition}

If $\mu$ is a lower semicountinuous submeasure (Definition \ref{Def.Submeasure}) on $\bbN$ %(see Definition~\ref{Def.Submeasure}, \ref{Eq.Fin(phi)}, and \ref{Eq.Exh(phi)}), 
and $0\leq r<\infty$, let
\[
\cK(\mu,r)= \{A\subseteq \bbN\mid \mu(A)\leq r\}.
\]
These give closed approximations to both $\Fin(\mu)$ and $\Exh(\mu)$. The utility of the following easy lemma will become clear in a few paragraphs. 

\begin{lemma} \label{L.countablydetermined} 
If $\mu$ is a lower semicontinuous submeasure on $\bbN$ and $d\geq 1$. Then
$\Fin(\mu)=\bigcup_{m\in\bbN}\cK(\mu,m)$ and $\Exh(\mu)=\bigcap_m (\cK(\mu,{1}/{m})^d\ucup \Fin)$.
% IF: The displayed formula looked quite ugly. 
%\[
%\textstyle \Fin(\mu)=\bigcup_{m\in\bbN}\cK(\mu,m),\text{ and }\Exh(\mu)=\bigcap_m (\cK(\mu,\frac{1}{m})^d\ucup \Fin).
%\] 
\end{lemma}

 The following is a special case of \cite[Theorem~11]{Just:WAT} (Just considered $\Phi$ which is a lattice embedding and obtained a slightly weaker conclusion). 

\begin{lemma} \label{L.Just.WAT}
		Assume $\OCA$. If $\cK$ is a closed approximation to an ideal $\cJ$, $\Phi\colon \cP(\bbN)\to \cP(\bbN)/\cJ$ is a Boolean algebra homomorphism, and $\cA$ is a tree-like family, then for all but countably many $A\in \cA$ the restriction of $\Phi$ to $\cP(A)$ is $\cK\ucup\cK$-precise. 
\end{lemma}
\begin{proof}[Sketch of a proof] 
 With $\cX$ as in the proof of Lemma~\ref{L.Vel.2.2}, define $L_\cK\subseteq [\cX]^2$ by letting $\{(B,C), (B',C')\}\in L_\cK$ if the following conditions hold
\begin{enumerate}
\item [($L_\cK1$)] $\bfB(B)\neq \bfB(B')$. 
\item [($L_\cK2$)] $B\cap C'=B'\cap C$. 
\item [($L_\cK3$)] $(\Phi_*(B)\cap \Phi_*(C'))\Delta (\Phi_*(B')\cap \Phi_*(C))\notin \cK$.\end{enumerate}
Then $L_\cK$ is open, and an analog of the proof of Lemma~\ref{L.Vel.2.2} yields the desired $\sigma$-Borel $\cK\ucup\cK$-approximation on $\cP(A)$ for all but countably many $A\in \cA$. 
\end{proof}

 Solecki's characterisation of analytic P-ideals (Theorem \ref{T.Solecki}) associates a sequence of closed approximations to every analytic P-ideal
 by Lemma~\ref{L.countablydetermined}. Together with a meticulous uniformisation of locally defined liftings (\cite[\S 3.13]{Fa:AQ}) one obtains the following (see~\cite[Theorem 3.3.5]{Fa:AQ}). 

\begin{theorem}[OCA Lifting Theorem]\label{thm:ilijasOCAideals}
Assume $\OCA$ and $\MA$  and let $\Phi\colon \cP(\bbN)\to \cP(\bbN)/\cJ$ be a homomorphism such that $\cJ\supseteq \Fin$ is countably generated or an analytic P-ideal. 
 Then $\Phi$ is the direct sum of a topologically trivial homomorphisms and one with a nonmeager kernel. In particular, every isomorphism between a quotient over an analytic ideal and $\cP(\bbN)/\cJ$ is topologically trivial.\footnote{It is well-known that every proper analytic ideal that includes $\Fin$ is meager.} 
\end{theorem}

As a consequence, thanks to Ulam-stability and the Radon--Nykodim property (see Theorem~\ref{T.idealiso}), we obtain the following:

\begin{corollary}\label{C.OCAMA}
Assume $\OCA$ and $\MA$. Let $\mathcal I$ and $\mathcal J$ be Borel ideals in $\mathcal P(\bbN)$ such that $\mathcal J$ be either a nonpathological $P$-ideal or countably generated. Then
\begin{itemize}
\item all automorphisms of $\mathcal P(\bbN)/\mathcal J$ are algebraically trivial, 
\item $\mathcal P(\bbN)/\mathcal I$ and $\mathcal P(\bbN)/\mathcal J$ are isomorphic if and only if $\mathcal I$ and $\mathcal J$ are Rudin--Keisler isomorphic, 
\end{itemize}
\end{corollary}

As pointed out in \S\ref{S.Abel}, all analytic P-ideals are $F_{\sigma\delta}$, but there are Borel ideals of arbitrarily high complexity (e.g., the ideals from Theorem~\ref{T.idealiso} \ref{3.T.idealiso} and \ref{4.T.idealiso} are Borel, with arbitrarily high Borel complexity). 

This is a good moment to state a revised version of \cite[Question~15.2]{Fa:Luzin} (the original question asked for $d=3204$ in place of $d=80$;  compare \cite[Theorem~1.4]{farah2024biba} and \cite[Theorem~10.4]{Fa:Luzin}). 
\begin{question} \label{Q.countablydetermined} Suppose that $\cI$ is an $F_{\sigma\delta}$ ideal on $\bbN$. Are there closed approximations $\cK_m$, for $m\in \bbN$, such that $\cI=\bigcap_m (\cK_m^d\ucup \Fin)$ for every $d\geq 1$? What about the cases $d=1$ and $d=80$? 
\end{question}

An ideal for which the answer to Question~\ref{Q.countablydetermined} is positive is said to be $80$-\emph{determined by closed approximations}. By \cite[Theorem~1.4]{farah2024biba} (see \cite[Theorem~10.4]{Fa:Luzin}), $\OCA$ and $\MAsigmalinked$  together imply that for such $\cI$ and any analytic ideal $\cJ$, every isomorphism between $\cP(\bbN)/\cI$ and $\cP(\bbN)/\cJ$ is topologically trivial. This class includes the ideals $\NWD(\bbQ)$, $\NULL(\bbQ)$, and the Weyl ideal $\cZ_W$ (see Theorem~\ref{T.idealiso}). 

\begin{question} \label{Q.OCAlt} Does $\OCA+\MA$ imply the analog of OCA Lifting Theorem for all analytic ideals, or at least that every isomorphism between their quotients is topologically trivial? 
\end{question}

We don't even know whether Question~\ref{Q.OCAlt} has a positive answer for  all~$F_{\sigma\delta}$ ideals (see however Theorem~\ref{thm:ConsRigIdeals}). After two decades with virtually no progress, in \cite{farah2024biba} the conclusions of Theorem~\ref{thm:ilijasOCAideals} and Corollary~\ref{C.OCAMA} have been strengthened to include all ideals  that satisfy the property from Question~\ref{Q.countablydetermined} with $d=80$ and their assumptions have been weakened to $\OCA$ and $\MAsigmalinked$ (see~\S\ref{S.OCAsharp}).

\subsubsection{\v{C}ech--Stone remainders of zero-dimensional spaces}
The following was proved by combining the ideas from \cite[\S 4]{Fa:AQ} and a stratification of the clopen algebra of $\beta X$ with copies of $\cP(\bbN)$. 

\begin{theorem}[\cite{FaMcK:Homeomorphisms}] \label{T.all}\label{cor:FAzerodim}
Assume $\OCA$ and $\MA_{\aleph_1}$. Let $X$ and $Y$ be second countable locally compact zero-dimensional topological spaces. Then all automorphisms of $C(\beta X\setminus X)$ are topologically trivial and $\beta X\setminus X$ and $\beta Y\setminus Y$ are homeomorphic if and only if $X$ and $Y$ have co-compact homeomorphic subsets. 
\end{theorem}

Related results along the lines of the OCA lifting theorem will be discussed in the context of general endomorphisms in \S\ref{S.endo}. 
Theorem~\ref{T.all} is as far as one could get by sticking to Boolean algebras in the rigidity results for \v Cech--Stone remainders. We move on to \cstar-algebras. 

\subsection{Coronas and Forcing Axioms}
\label{6bii.PFAAVi}
This section is dedicated to a sketch of the proof that Forcing Axioms imply all isomorphisms of coronas of separable \cstar-algebras are topologically trivial. We start by studying the first noncommutative example: the Calkin algebra.
\subsubsection{The Calkin algebra}
\begin{theorem}[{\cite[Theorem~1]{Fa:All}}]\label{T.Calkin}
Assume $\OCA$. Then all automorphisms of $\mathcal Q(H)$ are inner.
\end{theorem}

We will sketch the proof from \cite[\S 17]{Fa:STCstar}, based on \cite{Fa:All} (see also \cite{Fa:Errata}). Fix an automorphism $\Phi$ of $\cQ(H)$. By using the stratification of $\cQ(H)$ from~\S\ref{S.Strat}, an initial segment of the proof is concerned with the restrictions of $\Phi$ to reduced products of the form $\calD[\bfE]$. We quickly recap the key ideas from~\S\ref{S.Strat} while slipping in a few additional definitions. Recall that $\Part$ is the poset of partitions $ \bfE=\langle E_j\colon j\in \bbN\rangle$ of~$\bbN$ into finite intervals, ordered by $\bfE\leq^*\bfF$ if $(\forall^\infty n)(\exists m) E_m\subseteq F_n$. Fixing a basis of $H$ once and for all, to $\bfE\in \Part$ one associates a decomposition of $H$ into a direct sum of finite-dimensional subspaces $H=\bigoplus_n H_n$. Denote the projection to $H_n$ by $q^{\bfE}_n$ and let, for $X\subseteq\bbN$, $q_X^{\bfE}=\sum_{n\in X}q_n^{\bfE}$. %the projections $q^{\bfE}_n$ are pairwise orthogonal and the partial sums of $\sum_n q^{\bfE}_n$ converge to the identity on $H$ in the strong operator topology, or SOT. 
Let
 \begin{equation*}
\mathcal D[\bfE]=\{a\in\mathcal B(H)\mid a\text{ commutes with }q_X^{\bfE}\text{ for all }X\subseteq\bbN\}.
\end{equation*}
%This is a \cstar-subalgebra of $\cB(H)$. 
Let $\bfE^{\even}$, $\bfE^{\odd}$, and $\cF[\bfE]$ be as in \S\ref{S.Strat}. 
The following are the main steps of the proof, obtained for all $\bfE$ simultaneously. 
\begin{enumerate}
 \item \label{Auto.1} $\Phi$ is topologically trivial on $\calD_X[\bfE]$ for some infinite~$X$. 
 \item \label{Auto.1.5} $\Phi$ is algebraically trivial on $\calD_X[\bfE]$ for some infinite~$X$. 
 \item \label{Auto.2} $\Phi$ is algebraically trivial on $\calD[\bfE]$. 
 \item \label{Auto.2.5} $\Phi$ is algebraically trivial on $\cF[\bfE]$. 
 \item \label{Auto.3} A certain coherent family of unitaries (Definition~\ref{Def.Coherent}) is trivial. 
\end{enumerate}
We will first sketch the implications from \ref{Auto.1} to \ref{Auto.3}, and then how $\OCA$ implies \ref{Auto.1}.

The implication from \ref{Auto.1} to \ref{Auto.1.5} uses the methods from \S \ref{S.Ulam}. In order to use the analog of Lemma~\ref{L.AsympAdd} for finite-dimensional \cstar-algebras, one needs to replace the matrix algebra components of $\calD[\bfE]$ with finite sets. Identify $\calD[\bfE]$ with $\prod_n M_{k(n)}(\bbC)$ (with $k(n)=|E_n|$) and choose a finite $2^{-n}$-dense $\sfD(n)\subseteq M_{k(n)}(\bbC)$ (it is convenient to impose additional properties on $\sfD_n$, see \cite[Definition~17.4.2]{Fa:STCstar}). Then $\sfD[\bfE]=\prod_n \sfD(n)$ is a \emph{discretisation of $\calD[\bfE]$}. Since for $a\in (\calD[\bfE])_1$ there is $a'\in \sfD[\bfE]$ such that $a-a'$ is compact, topological triviality on $\sfD[\bfE]$ implies topological triviality on $\sfD[\bfE]$. By the Ulam-stability (Lemma~\ref{FD-Ulam} and the analog of Proposition~\ref{P.Ulam}), \ref{Auto.1.5} follows. (See \cite[\S 17.4]{Fa:STCstar} for the details.)

We move into the fun part---the middle of the proof, \ref{Auto.1.5} $\Rightarrow$ \ref{Auto.2}---leaving the naughty bits for later. One of the main causes of the difference in behaviour between the commutative and the noncommutative settings is that for every nonzero projection $p\in \cQ(H)$ some isometry $v$ satisfies $vv^*=p$ (this is an instance of the \emph{Murray--von Neumann equivalence}, the noncommutative analog of equinumerosity). Suppose $\bfE$ and $X$ satisfy $\sup_{n\in X} |E_n|=\infty$. For every $\bfF$ there is an injection $g\colon \bbN\to \bbN$ such that $|F_n|\leq |E_{g(n)}|$ for all $n$. Thus there is a partial isometry (see \ref{3.partial.isometry}) $v_n$ such that $v_n^*v_n=q^{\bfF}_n$ and $v_n v_n^*\leq q^{\bfE}_{g(n)}$. Then $\sum_n v_n$ SOT-converges to an isometry $v$ such that $\Ad v$ is an injective $^*$-homomorphism from $\calD[\bfF]$ into $\calD_X[\bfE]$. This, and Lemma~\ref{lemma:partialiso}, are the reasons why the present proof, unlike those in \S\ref{S.FABoole}, does not require $\MA$.

\begin{lemma}[The Isometry Trick] \label{lemma:partialiso} Suppose that $A$ and $B$ are \cstar-subalgebras of $\cQ(H)$ and $v$ is an isometry such that $(\Ad v)[A]\subseteq B$. If $\Phi\in \Aut(\cQ(H))$ is implemented by a unitary $u$ on $B$, then it is implemented by any lifting of $\Phi (\pi(v^*))\pi(uv)$ on $A$. 
\end{lemma}

\begin{proof} Fix $a\in A$. Since $v^*v=1$, we have $a=v^*vav^*v$, hence 
\[
\Phi(\pi(a))=\Phi(\pi(v^*))\Phi(\pi(vav^*))\Phi(\pi(v))
=\Phi(\pi(v^*)) \pi(u vav^* u^*)\Phi(\pi(v))
\]
which is equal to $(\Ad \Phi(\pi(v^*))\pi(uv))\pi(a)$, as required. 
\end{proof}

Thus in order to prove \ref{Auto.1.5} $\Rightarrow$ \ref{Auto.2} it suffices to start from $\bfE$ such that $\lim_{n\to \infty} |E_n|=\infty$. Moreover, the proof of Lemma~\ref{lemma:partialiso} is flexible enough to transfer other properties of $\Phi$, such as topological triviality, and the existence of an $\varepsilon$-approximation (see below) from $\calD[\bfE]$ to $\calD[\bfF]$ (see \cite[Lemma~17.5.3]{Fa:STCstar}). 

\ref{Auto.2} $\Rightarrow$ \ref{Auto.2.5} is easy (\cite[Proposition~17.5.8]{Fa:Errata}).

Still postponing the technicalities, we handle \ref{Auto.2.5} $\Rightarrow$ \ref{Auto.3}. Let $\bfF$ be such that each $F_n$ is a singleton. Then $A=\calD[\bfF]$ is a maximal abelian \cstar-subalgebra of $\cB(H)$ (this is the \emph{atomic masa} of $\cB(H)$; see \cite[\S 12.3]{Fa:STCstar}) isomorphic to $\ell_\infty$. Let $u$ be the unitary that implements $\Phi$ on $\calD[\bfF]$. Then $\Phi_1=\Ad \pi(u^*) \circ \Phi$ is an inner automorphism if and only~$\Phi$ is, and its restriction to $A$ is equal to the identity. Also, for every $\bfE$, if $v$ is a unitary that implements $\Phi$ on $\cF[\bfE]$, then $(\Ad u^* v)(a)-a$ is compact for all $a\in A$. By the Johnson--Parrot Theorem (\cite[Theorem~12.3.2]{Fa:STCstar}), there is $v_0\in A$ such that $v_0- u^*v$ is compact. A standard perturbation argument gives a unitary $v_\bfE\in A$ such that $v_\bfE-v_0$ is compact. 

This shows that, as in Lemma~\ref{L.coherent}, $\Phi_1$ is induced by a coherent family of unitaries $\cF_\Phi=\{(\bfE, u_\bfE)\}$. The poset $\Part$ is $\sigma$-directed and every subset of cardinality $\aleph_1$ is bounded (this is a consequence of \cite[Theorem~9.7.8]{Fa:STCstar} and the analogous facts for $(\bbN^\bbN,\leq^*)$), and uniformisation of coherent families with such an index set is something that $\OCA$ is really good at. A somewhat elaborate, yet fun, argument shows that $\cF_\Phi$ is trivial (\cite[Theorem~17.8.2]{Fa:STCstar}). The trivializing unitary implements $\Phi_1$ on $\cQ(H)$, and therefore~$\Phi$ is inner. 

We procrastinate for another moment to point out that, together with~\S\ref{S.Cohomology}, this shows that the assertion `every automorphism of $\cQ(H)$ induced by a coherent family of unitaries is inner' is independent from ZFC.

%%%%%%%%%%%

It `only' remains to prove that $\OCA$ implies \ref{Auto.1}. Fix $\bfE$ such that $\lim_n |E_n|=\infty$ and a discretisation $\sfD[\bfE]=\prod_n \sfD(n)$ of $\calD[\bfE]$. Fix $\varepsilon>0$. A function $\theta\colon \bfD_X[\bfE]\to \cB(H)$ is an \emph{$\varepsilon$-approximation} to $\Phi$ if 
\[
\|\pi(\Phi_*(a)-\theta(a))\|\leq \varepsilon
\]
for all $a$ in $\bfD_X[\bfE]$. 
Approximations can be topologically trivial, Borel, or $\sigma$-Borel, with the natural definitions. 
 We will need a self-strengthening of $\OCA$. For distinct $x$ and $y$ in $\twoo$, $\Delta(x,y)=\min\{n\mid x(n)\neq y(n)\}$. 

\begin{definition} $\OCAi$ asserts that for every separably metrisable space~$\cX$ and a decreasing sequence of open sets $L^n\subseteq [\cX]^2$, for $n\in \bbN$ one of the following alternatives applies. 
\begin{enumerate}
\item [(1)] There exist an uncountable $T\subseteq \twoo$ and a continuous $f\colon T\to \cX$ such that
$
\{f(x), f(y)\}\in L^{\Delta(x,y)}
$
for distinct $x$ and $y$ in $T$. 

\item [(2)] There is a partition $\cX=\bigcup_n \cX_n$ such that $[\cX_n]^2$ is disjoint from $L^n$ for all $n$. 
\end{enumerate}
\end{definition}

Moore proved that $\OCA$ implies $\OCAi$ (\cite{moore2021some}, see also \cite[Theorem~8.6.6]{Fa:STCstar}). 
Here is another descendant of Lemma~\ref{L.Vel.2.2}, via Lemma~\ref{L.Just.WAT} and \cite[Proposition~3.12.1]{Fa:AQ}.

\begin{lemma} [{\cite[Lemma~17.6.3]{Fa:Errata}}] For every $\bfE$ and every $\varepsilon>0$, there exists an infinite $X$ such that $\Phi$ has a $\sigma$-Borel $\varepsilon$-approximation on $\calD_X[\bfE]$. 
\end{lemma}

\begin{proof} We first fix some notation. Fix a discretisation $\sfD[\bfE]$. Since $\bfE$ is fixed, let $q_Y=q_Y^\mathbf E$, for $Y\subseteq\bbN$. We also fix a second basis for $H$, and let $r_{n}$ be the canonical projection onto the $n$-th element of the such basis, and $r_X$, for $X\subseteq\bbN$, be equal to $\sum_{n\in X}r_n$. Further, fix an injection $\chi\colon \bbN\to \twolo$ and a lifting $\Phi_*$ of $\Phi$ such that $p_Y:=\Phi_*(q_Y)$ is a projection for every $Y\subseteq \bbN$ and that $\|\Phi_*(a)\|\leq 1$ whenever $\|a\|\leq 1$. 

The space 
\[
\cX=\{(Y, a)\mid Y\subseteq \bbN \text{ is infinite, } a\in (\calD_Y[\bfE])_1\}
\]
is equipped with a separable metric topology, obtained by identifying $(Y,a)$ with $(Y,a,q_Y, \Phi_*(a))$. As before, $\bfB(Y)$ denotes the unique branch in $\twoo$ containing $\chi[Y]$. For $n\geq 1$ define a subset of $ [\cX]^2$ by $\{(Y,a), (Z,b)\}\in M^{\varepsilon,n}$ if the following conditions hold. 
\begin{enumerate}
\item [(M1)] $\bfB(Y)\not =\bfB(Z)$. 
\item [(M2)] $q_Z a=q_Y b$. 
\item [(M$^n$3)] $\max(\|r_{[n,\infty)} (\Phi_*(a) p_Z -p_Y \Phi_*(b))\|,\|r_{[n,\infty)} (p_Z \Phi_*(a)- \Phi_*(b)p_Y)\|)>\varepsilon/2$. 
\end{enumerate}
(Note that $\|\pi(a)\|=\lim_{n\to \infty} \|r_{[n,\infty)} a\|$.)
Then $M^{\varepsilon, n}$, for $n\in \bbN$, is a decreasing sequence of open subsets of $[\cX]^2$.

The first alternative given by $\OCAi$ is that there is an uncountable $T\subseteq \twoo$ and a continuous $f\colon T\to \cX$ such that $\{f(x), f(y)\}\in M^{\varepsilon, \Delta(x,y)}$ for all distinct $x$ and $y$ in $T$. Define $c=(c_j)\in \calD[\bfE]$ by $c_j=q_{\{j\}} a$,\footnote{We would normally write $p_j$, but in this situation there is a danger of confusion.} if $(Y,a)$ is in the range of $f$ and $j\in Y$. Then $c$ is well-defined, and by (M1) and (M2) both $p_Y \Phi_*(c)-\Phi_*(a)$ and $\Phi_*(c) p_Y-\Phi_*(a)$ are compact for all $(Y,a)$ in the range of $f$. Fix $n$ such that for an uncountable $T_1\subseteq T$ we have $\|r_{[n,\infty)} (p_Y \Phi_*(c)-\Phi_*(a))\|<\varepsilon/2$ and $\|r_{[n,\infty)} (\Phi_*(c) p_Y-\Phi_*(a))\|<\varepsilon/4$ for all $(Y,a)$ in $f[T_1]$. We may assume $\Delta(x,y)>n$ for all distinct $x$ and $y$ in $T_1$ (and this is the step in the proof that requires $\OCAi$ instead of $\OCA$). Then distinct $(Y,a)$ and $(Z,b)$ in the range of $T_1$ satisfy
\[
r_{[n,\infty)} \Phi_*(a) p_Z \approx^{\varepsilon/4} r_{[n,\infty)} p_Y \Phi_*(c) p_Z\approx^{\varepsilon/4} r_{[n,\infty)} p_Y \Phi_*(b)
\]
and therefore the first half of the condition (M$^n$3) fails. By swapping $(Y,a)$ and $(Z,b)$, we see that the second half fails as well, and therefore $\{(Y,a), (Z,b)\}\notin M^n$ although their $f$-preimages $x$ and $y$ satisfy $\Delta(x,y)>n$; contradiction.

The second alternative of $\OCAi$ gives a sequence of closed subsets of $\cX$ which can be used to find a $\sigma$-Borel $\varepsilon$-approximation to $\Phi$ on $\calD_X[\bfE]$ for an infinite $X\subseteq \bbN$. 
\end{proof}

A Kuratowski--Ulam argument and a $\sigma$-Borel $\varepsilon$-approximation on $\calD_X[\bfE]$ are used to find a C-measurable $3\varepsilon$-approximation on $\calD_Y[\bfE]$ for an infinite $Y\subseteq X$ (\cite[Lemma~17.7.2]{Fa:STCstar}). A Baire category argument produces an infinite $Z\subseteq Y$ and a continuous $3\varepsilon$-approximation on $\calD_Z[\bfE]$. By the analog of the Isometry Trick (Lemma~\ref{lemma:partialiso}) $\Phi$ has a continuous $3\varepsilon$-approximation on $\calD[\bfE]$ for all $\varepsilon>0$, and the Jankov--von Neumann Uniformisation Theorem can be used to prove that $\Phi$ is topologically trivial on $\calD[\bfE]$ (\cite[Lemma~17.4.5]{Fa:STCstar}), completing the proof of \ref{Auto.1.5}.

\begin{question} Is it relatively consistent with ZFC that there exists an automorphism $\Phi$ of $\cQ(H)$ whose restriction to some $\calD[\bfE]$ is not implemented by a unitary? 
\end{question}

If $\Phi$ sends the atomic masa to a masa unitarily equivalent to it, then $\Phi$ cannot send the unilateral shift to its adjoint (see footnote \ref{footnote:Truss}), thus a negative answer implies a negative answer to Question~\ref{ques:bdf}. However, an adaptation of the construction from \cite{PhWe:Calkin} may provide a positive answer.

%%%%%%%
\subsubsection{More coronas}
Theorem~\ref{T.Calkin} inspired the second part of Conjecture~\ref{conj:CFgeneral}, asserting that under forcing axioms every isomorphism between coronas of separable \cstar-algebras is topologically trivial. Building on \cite{McK:Reduced} and \cite{mckenney2018forcing}, and a noncommutative version of the OCA Lifting Theorem from the latter reference in particular, this conjecture was confirmed in \cite{vignati2018rigidity}. 

\begin{theorem}[{\cite[Theorem B]{vignati2018rigidity}}]\label{T.mainPFA}
Assume $\OCA+\MA_{\aleph_1}$. Let $A$ and $B$ be separable \cstar-algebras. Then all isomorphisms between $\mathcal Q(A)$ and $\mathcal Q(B)$ are topologically trivial.
\end{theorem}

The outline of the proof of Theorem~\ref{T.mainPFA} follows the proof of Theorem~\ref{T.Calkin}, but is much more difficult. First, even though one can stratify a general corona $\mathcal Q(A)$ (similarly to \S\ref{S.Strat}), the building blocks consist not of products of matrix algebras, but of products of finite-dimensional Banach spaces instead. The appropriate version of the `$\OCA$ lifting theorem' for reduced products of finite-dimensional spaces appears in~\cite[Lemma~5.10]{mckenney2018forcing}. Secondly, the assumptions of the \emph{isometry trick} do not apply in most coronas (Lemma \ref{lemma:partialiso}). The latter obstruction is the reason why we need to assume Martin's Axiom in Theorem~\ref{T.mainPFA}.

The reader interested in the proof of Theorem~\ref{T.mainPFA} can find the details in ~\cite{vignati2018rigidity} and \cite{mckenney2018forcing}.

As a consequence of Theorem~\ref{T.mainPFA} and Ulam-stability (\S\ref{S.Ulam}, in particular Theorem~\ref{thm:UlamAB}), we have the following extension of Corollary~\ref{cor:FAzerodim} to the non zero-dimensional setting.

\begin{corollary}
Assume $\OCA+\MA_{\aleph_1}$, and let $X$ and $Y$ be second countable locally compact topological spaces. Then:
\begin{itemize}
\item all automorphisms of $C(\beta X\setminus X)$ are algebraically trivial;
\item if $\beta X\setminus X$ and $\beta Y\setminus Y$ are homeomorphic then $X$ and $Y$ are homeomorphic modulo open sets with compact closure.
\end{itemize}
\end{corollary}

As a corollary of this and Corollary~\ref{cor:Parov}, there are abelian separable \cstar-algebras $A$ and $B$ such that the assertion $\cQ(A)\cong \cQ(B)$ is independent from ZFC. Since $M_n(\cQ(A))\cong \cQ(M_n(A))$ for all $n$, by tensoring $A$ and $B$ by $M_n(\bbC)$ for some $n$ one obtains a pair of noncommutative \cstar-algebras with this property. 

The following is what remains of Question~\ref{ques:sakai} after \S\ref{S.Independence} and \S\ref{S.Independence2}.
\begin{question} Are there simple, separable \cstar-algebras $A$ and $B$ such that the assertion $\cQ(A)\cong \cQ(B)$ is independent from ZFC? 
\end{question}

A \cstar-algebra is \emph{subhomogeneous} if it is isomorphic to a subalgebra of $M_n(\ell_\infty(\kappa))$ for some $n\in \bbN$ and cardinal $\kappa$.\footnote{This is equivalent to the standard definition.} It is not difficult to see that if $A$ is subhomogeneous, then so is $\cQ(A)$, but not vice versa.

\begin{corollary}\label{C.trivial-matrix} There are separable \cstar-algebras $A$ and $B$ whose coronas are not subhomogeneous and such that the assertion `the coronas $\cQ(A)$ and $\cQ(B)$ are isomorphic' is independent from ZFC. 
\end{corollary}

\begin{proof} As in Example~\ref{Ex.FV}, fix an infinite $X\subseteq \bbN$ such that $\Th(M_{n}(\bbC))$, for $n\in X$, converge (in the logic, i.e., weak*, topology). Let $Y$ and $Z$ be disjoint and infinite subsets of $X$. By Example~\ref{Ex.FV}, CH implies that the coronas $\cM_Y$ and $\cM_Z$ (see Definition~\ref{Def.MX}), are isomorphic. However, Theorem~\ref{T.mainPFA} implies that every isomorphism between such algebras is topologically trivial, and therefore algebraically trivial (see~\cite{Gha:FDD}).
\end{proof}

\subsection{Rigidity in other categories} \label{S.fields-etc} 
In Theorem~\ref{T.dividing-line} we saw that a dividing line for the first part of Conjecture \ref{Meta.1} is the stability of the theory of the reduced product. We do not have a clean dividing line for the second part of this conjecture, but it has been confirmed in several categories. So far we have seen that the second part of this conjecture holds for quotient Boolean algebras $\cP(\bbN)/\cI$ for certain analytic ideals $\cI$ as well as for coronas of separable \cstar-algebras. The case of semilattices $\cP(\bbN)/\cI$ was considered in \cite{Just:WAT}, \cite{Just:Nowhere}, and \cite{Fa:Liftings}. 
Rigidity results for some other categories of discrete structures  appear in~\cite{de2023trivial}. The following example, suggested to us by Zo\' e Chatzidakis,  is suggestive. 
If $F=\prod_n F_n/\cI$ is a reduced product of fields, then the idempotents in $F$ form a copy of $\cP(\bbN)/\cI$, with the multiplication and addition (on disjoint sets) interpreted as the Boolean algebraic intersection and union. Thus we have a definable copy of the Boolean algebra $\cP(\bbN)/\cI$, and to an isomorphism $\Phi$ between such reduced products one can associate an isomorphism $\alpha_\Phi$ between quotients of the form $\cP(\bbN)/\cI$. Triviality of the latter is the first step in a proof (using $\OCA$ and $\MA$) that~$\Phi$ is (with a natural definition) trivial. 
This observation was formalised as `theory recognizing coordinates' in  \cite[Definition~2.6]{de2023trivial}, giving  rigidity in reduced products of structures in some other categories (linear orderings, trees, sufficiently random graphs) in which there is no obvious `global' copy of $\cP(\bbN)/\cI$. 
 These rigidity results came as a surprise (even to the authors) and they also led to more optimal proofs that all automorphisms of $\cP(\bbN)/\Fin$ are trivial and of the OCA Lifting Theorem (see \S\ref{S.OCAsharp}). 
In \cite[Theorem~1]{farah2025coordinate} it was proved that a category (not necessarily axiomatisable)  of countable structures recognises coordinates (and therefore $\OCA$ and $\MA$ imply that all automorphisms of reduced products over $\Fin$ of its elements are trivial) if and only if the formula $x=x'\rightarrow y=y'$ is equivalent to an $h$-formula (see \S\ref{S.PP}). One consequence is that a category recognises coordinates if and only if its theory does. This paper, true to its title, contains other surprises on the theory of recognizing coordinates.  

The continuous version of the main rigidity result of \cite{de2023trivial} appears in \cite{de2024metric} and it was applied to universal sofic groups and Higson coronas, respectively, in \cite{de2024automorphism} and \cite{V.Higson}.

\subsection{Other models for rigidity} 
\label{S.6bi}
\label{6bi}
We continue the narrative of \S\ref{S.Shelah} by reviewing some of the consistency results regarding the rigidity questions for the Borel quotients of the Boolean algebra $\mathcal P(\bbN)$ and their generalisations to noncommutative \cstar-algebras. We will briefly discuss some iterated forcing models and, more specifically, the forcing extension used in~\cite{FaSh:Trivial}, in which all the isomorphisms between the Borel quotients of $\mathcal P(\bbN)$ are topologically trivial. Then we review the main theorem of~\cite{Gha:FDD} which uses a slightly different forcing to generalise the result of~\cite{FaSh:Trivial} to the Borel quotients of direct products of sequences of matrix algebras. 
As we shall see, the results of~\cite{FaSh:Trivial} and~\cite{Gha:FDD} are more peculiar when noticing that in these generic extensions the Calkin algebra has outer automorphisms.

\subsubsection{$\mathcal P(\bbN)/ \Fin$ vs. the Calkin algebra}\label{S.FinvsCalkin} Another major progress in the rigidity of the isomorphisms of Borel quotients of $\mathcal P(\bbN)$ was made in~\cite{FaSh:Trivial}, where the following theorem was proved.

\begin{thm}[{\cite{FaSh:Trivial}}]\label{thm:ConsRigIdeals}
Assume there is a measurable cardinal. There is a forcing extension in which every isomorphism between quotient Boolean algebras $\mathcal P(\bbN)/ \mathcal I$ and
$\mathcal P(\bbN)/ \mathcal J$ over Borel ideals $\mathcal I$ and $\mathcal J$ has a continuous lifting, and yet the Calkin algebra has outer automorphisms.
\end{thm}

The forcing extension satisfies $\mathfrak d = \aleph_1$ and $2^{\aleph_0}<2^{\aleph_1}$, therefore outer automorphisms of the Calkin algebra exist by Theorem~\ref{T.d=aleph1}. For the same reason, in this model, \v{C}ech--Stone remainders of non zero-dimensional noncompact manifolds have nontrivial autohomeomorphisms.

The general idea is as in \S\ref{S.Shelah}, with a few twists. The forcing extension used in~\cite{FaSh:Trivial} is a countable support iteration of length $\mathfrak c^+$ of creature forcings (see~\cite{roslanowski1999norms}) and random reals (see~\cite[\S 3.2]{BarJu:Book}). The creature forcings assure that the ideal of sets on which $\Phi$ is topologically trivial is dense while the random reals facilitate the uniformisation, namely that $\bbN$ is in this ideal. 
 Several definability and absoluteness results from descriptive set theory have been used along the way in the proof. In particular, it is crucial that these forcings are proper, \emph{Suslin}\footnote{A forcing notion is Suslin if its underlying set is an analytic set of reals and both $\leq$ and $\bot$ are analytic relations.} and \emph{$\bbN^\bbN$-bounding}. The reflection argument used here is analogous to the one described at the beginning of \S\ref{S.Shelah}.

The results of~\cite{FaSh:Trivial} were generalised in~\cite{Gha:FDD} to the noncommutative setting of the reduced products of the matrix algebras. Moreover, in the forcing extension used in~\cite{Gha:FDD} the creature forcings are replaced with a `groupwise Silver forcing'.

\subsubsection{Groupwise Silver forcings} They are defined similar to the \emph{infinitely equal} forcing notion $\mathbb E$ (see~\cite[\S 7.4.C]{BarJu:Book}).
Let ${\bf I} = (I_n)$ be a partition of $\bbN$ into nonempty finite intervals and $G_n$ be a finite set (of finite rank operators, in the case of~\cite{Gha:FDD}) for each $n\in\bbN$. Endow the set $\prod_{n}G_n$ with the product topology. For each $n$, let $F_n=\prod_{i\in I_{n}}G_{i}$ and $F=\prod_{n\in\bbN}F_n$.
Define the \emph{groupwise Silver forcing} $\mathbb S_F$ associated with $F$ to be the following forcing notion: A condition $p\in \mathbb S_F$ is a function from $M\subseteq \bbN$ into $\bigcup_{n=0}^\infty F_n$, such that $\bbN \setminus M$ is infinite and $p(n)\in F_{n}$. A condition $p$ is stronger than $q$ if $p$ extends $q$. Each condition $p$ can be identified with $[p]$, the set of all its extensions to $\bbN$, as a compact subset of $F$. % For a generic filter $G\subseteq \mathbb S_F$, $f = \bigcup \{p: p\in G\}$ is the generic real.
Groupwise Silver forcings, as well as the random forcing, fall in the scope of `forcing with ideals'. These forcing notions and their countable support iterations were investigated by Zapletal in~\cite{Zap:Descriptive} and~\cite{Zap:Forcing}.
Some of the properties of these forcings were applied to a countable support iteration of Silver and random forcings in~\cite{Gha:FDD} in order to prove that consistently every isomorphism between the Borel quotients of products of matrix algebras (and in particular Borel quotients of $\mathcal P(\bbN)$) is topologically trivial.

\subsubsection{Reduced products of matrix algebras}\label{S.Reduced} For an ideal $\mathcal I$ on $\bbN$ consider the reduced product (Example~\ref{ex:coronas})
\[
\textstyle M_\mathcal I := \prod_{n}M_{n}(\bbC)/\bigoplus_{\mathcal I} M_{n}(\bbC).
\]
This is a generalisation of the algebras $\cM_X$ introduced in Definition~\ref{Def.MX}. 

\begin{thm}[\cite{Gha:FDD}]\label{thm-FDD} Assume there exists a measurable cardinal. There is a forcing extension in which the following assertions hold:
\begin{enumerate}
\item[(1)] for all Borel ideals $\mathcal I$ and $\mathcal J$ on $\bbN$ every isomorphism between $ M_\mathcal I$ and $ M_\mathcal J$ has a continuous lifting,
 \item[(2)] for all analytic P-ideals $\mathcal I$ and $\mathcal J$ on $\mathbb{N}$, $ M_\mathcal I$ and $ M_\mathcal J$ are isomorphic if and only if $\cI=\cJ$, and every isomorphism is trivial and equal to the identity on the center, 
 \item[(3)] the Calkin algebra has outer automorphisms.
\end{enumerate}
In particular, all automorphisms of coronas of the form $\cM_X$ are inner while the Calkin algebra has outer automorphisms. 
\end{thm}

%Without the assumption of existence of measurable cardinals, in the same forcing extension every *-homomorphisms between the algebras of the form $M^{\{k(n)\}}_\mathcal I$ is locally topologically trivial. 
The main reason why (2) of Theorem~\ref{thm-FDD} is true for \emph{all} analytic P-ideals is that the class of all finite-dimensional \cstar-algebras is Ulam-stable (Theorem~\ref{FD-Ulam}). It is not known whether all approximate homomorphisms between finite Boolean algebras with respect to arbitrary submeasures are stable (see \S\ref{S.Ulam}).

Theorem~\ref{thm-FDD} gives an alternative proof of Corollary~\ref{C.trivial-matrix}, modulo the existence of a measurable cardinal. 
We do not know whether the conclusion of Theorem~\ref{thm-FDD} (1) and (2) follows from forcing axioms. 
%A closer inspection of the trivial isomorphisms between these algebras implies the following proposition.

% \begin{proposition}\label{prop. trivial-matrix}
%Assume\marginpar{IF: This nice theorem looks horrible because of the notation} there is a measurable cardinal. There is a forcing extension in which the following holds. For analytic P-ideals $\mathcal{I}$, $\mathcal{J}$ on $\mathbb{N}$, and sequences $\{k(n)\}$ and $\{l(n)\}$ of natural numbers, $M^{\{k(n)\}}_\mathcal I$ and $M^{\{l(n)\}}_\mathcal J$ are isomorphic if and only if there are sets $B\in \mathcal{I}$ and $C \in \mathcal{J}$ and a bijection $\sigma: \mathbb{N}\setminus B\mapsto \mathbb{N}\setminus C$ such that $X\in \mathcal{I}$ if and only if $\sigma[X]\in \mathcal{J}$ (that is, $\mathcal{I}$ and $\mathcal{J}$ are Rudin-Keisler isomorphic), and$|k(n)|=|l(\sigma(n))|$ for every $n\in \bbN\setminus B$. \end{proposition}

\subsubsection{Large continuum}
All models of $\PFA$ and  all known models of $\OCA$ satisfy $\fc=\aleph_2$, and so do all models described in earlier sections. 
Models of $\ZFC$ with arbitrarily large continuum in which all automorphisms of $\cP(\bbN)/\Fin$ are trivial were obtained in \cite[p. 327]{Just:Modification} (using a modification of oracle-cc) and then again in \cite{dow2025automorphisms}, this time together with the $\MA$.

\section{Endomorphisms} \label{S.endo}
\enumtwo 
%Endomorphisms
In the previous parts of the survey we considered notions of topological and
algebraic triviality mainly focusing on isomorphisms or automorphisms of
quotient structures. In the current section we extend the scope of our analysis to more general maps, in various particular cases.

An example where the rigidity phenomena verified under Forcing Axioms for homeomorphisms extend to other classes of functions, arises in the context of continuous surjections between \v{C}ech--Stone remainders of locally compact, second countable, spaces. In~\cite{DoHa:Images}, Dow and Hart proved that, if $\beta X \setminus X$ is a continuous image of $\beta \bbN \setminus \bbN$, then $X$ is the union of $\bbN$ with a compact set, when $\OCA$ is assumed. Their result was later extended in~\cite[Theorem E]{vignati2018rigidity} as follows.

\begin{theorem}\label{T.7.1}
Assume $\OCA$ and $\MA_{\aleph_1}$, and let $X$ and $Y$ be locally compact, second countable spaces. If $\beta X \setminus X$ is continuous image of $\beta Y \setminus Y$, then there are $Z \subseteq Y$ and $K \subseteq X$ such that $\beta Z \setminus Z$ is clopen in $\beta Y \setminus Y$, $K$ is compact, and $X \setminus K$ is a continuous image of $Z$.
\end{theorem}

A stronger conclusion applies to a more restricted class of spaces.\footnote{An analogous result holds for powers of \v Cech--Stone remainders, but this is not a good moment to digress; see the subsection `Dimension phenomena' in \S\ref{S.other}.} 

\begin{theorem}{\cite[Theorem~4.9.1]{Fa:AQ}}\label{T.7.2}
Assume $\OCA$ and $\MA_{\aleph_1}$, and let $X$ and $Y$ be locally compact, second countable, zero-dimensional spaces. If in addition $X$ is countable, then for every continuous $f\colon \beta X \setminus X\to \beta Y\setminus Y$ there are a clopen partition $U\sqcup V =\beta X$ and a continuous $g\colon V\to \beta Y$ such that $f[U]$ is nowhere dense in $\beta Y\setminus Y$ and $f=g\rs (\beta X \setminus X)$ 
\end{theorem}

This theorem was stated and proved for arbitrary finite powers of $\beta X\setminus X$ and $\beta Y\setminus Y$. Soon afterwards it was discovered that  the above implies the full version of the theorem because continuous maps between \v Cech--Stone remainders of locally compact second countable spaces respect coordinates  (see \S\ref{dim.phenomena}). By these results, Definition~\ref{def:wep} and Theorem~\ref{thm:wep} below are equivalent to their analogues stated for arbitrary finite powers.

The conclusion of Theorem~\ref{T.7.2} was called \emph{weak Extension Principle} in~\cite{Fa:AQ}. This principle was extended beyond the zero-dimensional setting in \cite{VY:wep}. Its (slightly) more intricate definition is due to the fact that clopen sets in $\beta X\setminus X$ do not necessarily come from clopen sets in $\beta X$ when $X$ is not zero-dimensional.

\begin{definition}\label{def:wep}
Let $X$ and $Y$ be locally compact noncompact second countable topological spaces. We say that $X$ and $Y$ satisfy the \emph{weak Extension Principle} if the following happens:

For every continuous map $F\colon \beta X\setminus X \to \beta Y\setminus Y$ there exists a partition into clopen sets $\beta X\setminus X =U\cup V$, an open set with compact closure $V_X\subseteq X$ such that $F[U]$ is nowhere dense in $\beta Y\setminus Y$, and a continuous function $G\colon \beta(X\setminus V_X)\to \beta Y$ which restricts to $F$ on $V$.

The weak Extension Principle is the statement ``All pairs of  locally compact noncompact second countable topological spaces satisfy the weak Extension Principle".
\end{definition}

The following is the main result of \cite{VY:wep}.

\begin{theorem}\label{thm:wep}
Assume $\OCA$ and $\MA_{\aleph_1}$. Then the weak Extension Principle holds.
\end{theorem}

The conclusions of Theorem~\ref{T.7.2} and Definition~\ref{def:wep} look much better with $U=\emptyset$, and in~\cite[\S 4.11]{Fa:AQ} it was all but conjectured that this ``Strong Extension Principle'' follows from forcing axioms. Alas, the following example shows that we cannot take $U=\emptyset$. 

%On the other hand, the behavior of embeddings of \v{C}ech--Stone remainders departs from what one would expect by looking at homeomorphisms.

Focusing on the case $X = Y = \bbN$, we say that a subset $A \subset \beta \bbN \setminus \bbN$ homeomorphic to $\beta \bbN \setminus \bbN$ is a \emph{nontrivial copy of $\beta \bbN \setminus \bbN$} if it is not equal to $\overline B \setminus B$ for any countable $B \subset \beta \bbN$. By the Gelfand--Naimark Duality, this corresponds to a surjective unital $^*$-homomorphism
$\Phi: \ell^\infty/c_0 \to \ell^\infty /c_0$ that it is not induced by a unital $^*$-homomorphism from $\ell^\infty $ into itself, that is a \emph{algebraically nontrivial} map. While Forcing Axioms forbid the existence of algebraically nontrivial homeomorphisms of $\beta \bbN \setminus \bbN$ (this was thoroughly discussed in \S\ref{S.Abel} and \S\ref{S.6bi}), the existence of nontrivial copies of $\beta \bbN \setminus \bbN$ on the other hand follows from $\ZFC$ alone.

\begin{theorem}[{\cite{dow2014non}}]
There exists a nontrivial copy of $\beta \bbN \setminus \bbN$ inside $\beta \bbN \setminus \bbN$.
\end{theorem}

Turning to the noncommutative setting, in~\cite{vaccaro2019trivial} rigidity phenomena are investigated for $^*$-homomorphisms from the Calkin algebra to itself. Unlike the abelian setting, in this case the rigidity forced on the structure of $\cQ(H)$ by $\OCA$, along with some peculiar features of this algebra, allow to extend the triviality results obtained for automorphisms in Theorem~\ref{T.Calkin} to all unital $^*$-homomorphisms.

As $\cQ(H)$ is simple (as a \cstar algebra), all $^*$-homomorphisms from the Calkin algebra to itself are either zero or injective, hence we are simply going to refer to them as endomorphisms. Although the results in~\cite{vaccaro2019trivial} also apply to non-unital maps, in the present discussion we restrict to the unital case, since the core ideas are the same and this allows us to avoid unnecessary technical details.

%In~\cite{vaccaro2019trivial} the following notion of triviality for endomorphisms of $\cQ(H)$ is introduced.

\begin{definition} \label{def:trivialendo}
A unital endomorphism $\Phi \colon \cQ(H) \to \cQ(H)$ is \emph{algebraically trivial} if there is a unitary $v \in \cQ(H)$ and a strictly continuous unital endomorphism $\Phi_*\colon \cB(H) \to \cB(H)$ such that the following diagram commutes. 
\[
\begin{tikzpicture}
 \matrix[row sep=1cm,column sep=1.5cm] 
 {
& & \node (M1) {$\cB(H)$}; & \node (M2) {$\cB(H)$};&
\\
& & \node (Q1) {$\cQ(H)$}; & \node (Q2) {$\cQ(H)$} ;
\\
};
\draw (M1) edge [->] node [above] {$\Phi_*$} (M2);
\draw (Q1) edge [->] node [above] {$\text{Ad}(v) \circ \Phi$} (Q2);
\draw (M1) edge [->] node [left] {$\pi$} (Q1);
\draw (M2) edge [->] node [left] {$\pi$} (Q2);
\end{tikzpicture}
\]
\end{definition}

Note that an automorphism of $\cQ(H)$ is algebraically trivial as in Definition~\ref{def:trivialendo} if and only if it is inner. Because of Theorem~\ref{thm:autocalk}, Definition~\ref{def:trivialendo} agrees with Definition~\ref{def:trivialmap} on automorphisms of $\cQ(H)$. Hence the following is an extension of Theorem~\ref{T.all} to endomorphisms of $\cQ(H)$.

\begin{theorem}[{\cite[Theorem 1.3]{vaccaro2019trivial}}] \label{thm:trivial}
Assume $\OCA$. All endomorphisms of $\cQ(H)$ are algebraically trivial.
\end{theorem}

The techniques required to prove Theorem~\ref{thm:trivial} are refinements of the tools developed to prove Theorem~\ref{T.all}. Indeed a significant portion of the most technical results in~\cite{Fa:All} already applies to endomorphisms (the same arguments are also exposed in full detail in~\cite[\S 17]{Fa:STCstar}). By the appropriate analog of Lemma~\ref{lemma:partialiso}, it suffices to prove that the restriction of the endomorphism under consideration to some corner $p\cQ(H)p$ is algebraically trivial.

Along with some elementary observations on endomorphisms of $\cB(H)$ (see~\cite[Remark 2.3]{vaccaro2019trivial}), Theorem~\ref{thm:trivial} permits to completely classify unital endomorphisms of $\cQ(H)$ up to unitary equivalence via the Fredholm index of the image of the unilateral shift $s$. What can be obtained
under $\OCA$ is thus a result in the same spirit of the early studies on essential unitary equivalence discussed in \S\ref{3a.BDF}.

\begin{theorem}[{\cite[Theorem 1.1]{vaccaro2019trivial}}] \label{thm:classification}
Assume $\OCA$. Two unital endomorphisms $ \Phi_1, \Phi_2: \cQ(H) \to \cQ(H)$ are unitarily equivalent if and only if the Fredholm indices $\ind(\Phi_1(\dot s))$ and $\ind(\Phi_2(\dot s))$ are equal.
\end{theorem}
The latter theorem enables to define a bijection between the set of unital endomorphisms $\End(\cQ(H))$ modulo unitary equivalence $\sim_u$, and the set of positive integers. The index map is concretely defined as
\begin{align*}
\Theta\colon \End(\cQ(H)) / \sim_u & \to \bbN^+ \\
[\Phi] & \mapsto -\ind(\Phi(\dot s)).
\end{align*}
Both Theorem~\ref{thm:trivial} and Theorem~\ref{thm:classification} fail badly under $\CH$. In~\cite{PhWe:Calkin} the authors build a collection of size
$2^{\fc}$ of pairwise inequivalent, outer (hence nontrivial) automorphisms. Composing these automorphisms with an (algebraically trivial) endomorphism of index $m > 0$ produces an uncountable family $\{ \Phi_\alpha \}_{\alpha < 2^{\aleph_1}} \subseteq \End(\cQ(H))$ such that $[\Phi_\alpha] \not = [\Phi_\beta]$, but $\Theta([\Phi_\alpha]) = \Theta([\Phi_\beta]) = m$, for all $\beta \not = \alpha$.

A further difference in the behavior of the index map $\Theta$, when assuming CH, is its range. Under $\OCA$ the range of $\Theta$ coincides with the set of positive integers. On the other hand, $\CH$ entails the existence of a unital endomorphism $\Phi \in \End(\cQ(H))$ that maps every unitary of $\cQ(H)$ into a unitary with index zero, thus giving $\Theta([\Phi]) = 0$ (this follows from the main result of~\cite{farah2017calkin}; see also~\cite[Example 5.4]{vaccaro2019trivial}). It is not known whether the map $\Theta$ can assume negative values consistently with $\ZFC$. This problem relates to Question~\ref{ques:bdf}, since a $\Phi \in \End(\cQ(H))$ such that $\Theta([\Phi]) = -1$ is an endomorphism which, up to unitary equivalence, sends the unilateral shift to its adjoint. Finding such an endomorphism does not seem to be a much easier task than answering Question~\ref{ques:bdf} itself.

We conclude with an amusing corollary to Theorem~\ref{thm:trivial}. In \cite{farah2017calkin} it was proved that every \cstar-algebra of density character $\leq \aleph_1$ embeds into $\cQ(H)$, and therefore $\CH$ implies that $A$ embeds into $\cQ(H)$ if and only if its density character is not greater than $\aleph_1$. In particular, $\CH$ implies that the class of \cstar-algebras which embed into $\cQ(H)$ is, not surprisingly, closed under taking limits of countable inductive systems and under tensor products. What is surprising is that this is not a theorem of ZFC. 

\begin{corollary} $\OCA$ implies that the class of \cstar-algebras that embed into $\cQ(H)$ is not closed under tensor products and that it is not closed under countable inductive limits. 
\end{corollary} 

\begin{proof} The conclusion of Theorem~\ref{thm:trivial} implies that, under $\OCA$, $\cQ(H)\otimes A$\footnote{For the readers who would care, this is the spatial tensor product.} embeds into $\cQ(H)$ if and only if $A$ is finite-dimensional. The first part follows immediately. For the second part, choose $A$ that is an inductive limit of a sequence of finite-dimensional \cstar-algebras (e.g., any infinite-dimensional AF algebra). 
\end{proof}

\section{Larger Calkin algebras} \label{S.large}
\enumtwo 
This section is devoted to the strain of research initiated in the last two pages of the pivotal paper~\cite{Ve:OCA}. There Veli\v ckovi\'c proved that $\OCA$ and $\MA_{\aleph_1}$ imply that all automorphisms of $\cP(\aleph_1)/\Fin$\footnote{$\Fin$ is the ideal of finite subsets of $\kappa$, denoted $[\kappa]^{<\aleph_0}$ by some authors. The cardinal $\kappa$ will be clear from the context and this notation should not lead to any confusion.} are trivial and that $\PFA$ implies that, for every infinite cardinal $\kappa$, all automorphisms of $\cP(\kappa)/\Fin$ are trivial. In the topological reformulation (see \S\ref{ss.remainders}), the conclusion of the latter result asserts that every autohomeomorphism of the \v Cech--Stone remainder $\beta\kappa\setminus \kappa$ taken with the discrete topology can be extended to a continuous map from $\beta\kappa$ into itself. A nontrivial automorphism of $\cP(\bbN)/\Fin$ trivially (no pun intended!) extends to a nontrivial automorphism of $\cP(\kappa)/\Fin$, and therefore by Rudin's result $\CH$ implies that $\cP(\kappa)/\Fin$ has a nontrivial automorphism for every infinite cardinal $\kappa$. The situation with the Calkin algebra associated to a nonseparable Hilbert space is, not surprisingly, considerably more complicated. 

\subsubsection{Large Calkin algebras}Fix an uncountable cardinal $\kappa$ and (in this paragraph only) let $H=\ell_2(\kappa)$. Let $\cK(H)$ denote the ideal of compact operators on $\cB(H)$. For an infinite cardinal $\lambda\leq \kappa$ let 
\[
\cK_\lambda(H)=\overline{\{a\in \cB(H)\mid \text{the density character of $a[H]$ is smaller than $\lambda$}\}}, 
\]
in particular $\cK(H)=\cK_{\aleph_0}(H)$. 
It is not difficult to see that every norm-closed, self-adjoint ideal of $\cB(H)$ is of this form for some $\lambda$ (\cite[Proposition~12.3.4]{Fa:STCstar}). In particular, for any $\aleph_\alpha$ the proper ideals of $\cQ(\ell_2(\aleph_\alpha))$ are linearly ordered by inclusion in type $\alpha$. Suppose that $\Phi$ is an automorphism of the Calkin algebra $\cQ(H)=\cB(H)/\cK(H)$. Then $\Phi$ sends $\cK_\lambda(H)/\cK(H)$ to itself for every uncountable $\lambda\leq \kappa$. A moment of thought reveals that it is not at all clear whether the existence of an outer automorphism of $\cQ(\ell_2(\aleph_0))$ implies the existence of an outer automorphism of $\cQ(\ell_2(\kappa))$ for any uncountable value of $\kappa$. Indeed, it is an open problem whether the existence of an outer automorphism of $\cQ(\ell_2(\kappa))$ is relatively consistent with $\ZFC$ for a single uncountable~$\kappa$. Constructions of awkward objects using $\CH$ are typically easier (and historically obtained much earlier) than the proofs from forcing axioms that those objects do not exist. Nevertheless, a modification of Veli\v ckovi\'c's argument can be adapted, by introducing curious objects called `Polish Aronszajn trees', to show that $\PFA$ implies all automorphisms of $\cQ(\ell_2(\kappa))$ are inner for all infinite $\kappa$ (\cite{Fa:AllAll}), in spite of an observation that the noncommutative analog of the poset for adding uncountably many Cohen reals fails to be ccc.

%\subsubsection{More Calkin algebras} 
Our understanding of the triviality of automorphisms of `other Calkin algebras', those of the form $\cQ(\ell_2(\kappa))/\cK_\lambda(\ell_2(\kappa))$ for $\aleph_1\leq\lambda\leq \kappa$ brings yet another reversal. 

\begin{question}
Let $\lambda \le \kappa$ be uncountable cardinals. Are there outer automorphisms of $\cQ(\ell_2(\kappa))/\cK_\lambda(
\ell_2(\kappa))$?
\end{question}

The answer is easily positive if $\kappa=\lambda$ and $2^\kappa=\kappa^+$ (\cite[Theorem~1]{FaMcKSch}), but nothing else is known.
\subsubsection{The noncommutative Katowice nonproblem} The ideas from the previous paragraph can be used to undermine another heuristic principle, that questions about noncommutative objects are more difficult than their `commutative' analogs. The Katowice problem asks whether the Boolean algebras $\cP(\bbN)/\Fin$ and $\cP(\aleph_1)/\Fin$ can be isomorphic in some model of $\ZFC$; it is wide open (see e.g.,~\cite{chodounsky2016katowice}). On the other hand, the fact that the ideals of $\cQ(\ell_2(\aleph_\alpha))$ are linearly ordered in type $\alpha$ immediately implies that $\cQ(\ell_2(\kappa))$ and $\cQ(\ell_2(\lambda))$ are non-isomorphic for any pair of distinct infinite cardinals $\kappa$ and $\lambda$. (It should be noted that, by~\cite{balcar1978distinguish}, for uncountable cardinals $\kappa$ and $\lambda$, $\ZFC$ implies that $\cP(\kappa)/\Fin$ and $\cP(\lambda)/\Fin$ are isomorphic if and only if $\kappa=\lambda$.)

\subsubsection{$\cP(\kappa)/\Fin$ again} The study of automorphisms of $\cP(\kappa)/\Fin$ for uncountable~$\kappa$ is only loosely related to our story, but it hides surprises worth mentioning. If $\kappa$ is uncountable, then $\MA_{\aleph_1}$ and $\OCA$ already imply that all automorphisms of $\cP(\kappa)/\Fin$ are trivial, for all $\kappa\leq \mathfrak c$ (\cite[Theorem~3.4]{larson2016automorphisms}). 
%(This contradicts the remark made in the last paragraph of~\cite{Ve:OCA}, stating that using $\square_{\aleph_1}$, $\diamondsuit_{\aleph_2}^*$, $\MA_{\aleph_1}$, and $\fc=\aleph_2$ one can construct a nontrivial automorphism of $\cP(\aleph_2)/\Fin$.) 
In~\cite[Corollary~1.2]{shelah2016automorphisms} this conclusion was extended to all $\kappa$ smaller than the first inaccessible cardinal (or all $\kappa$, if there are no inaccessible cardinals).

Rigidity for the quotient of $\cP(\kappa)$ modulo the ideal of subsets of cardinality smaller than $\kappa$ was studied in \cite{larson2016automorphisms, kellner2024automorphisms, kellner2024nowhere}. 

%\cite{shelah2018trivial}

\section{Uniform Roe coronas}\label{S.Roe}
\enumtwo 
In this section we study automorphisms of certain quotient structures constructed from a class of \cstar-algebras linked in a natural way with the coarse geometry of metric spaces.

Coarse geometry is the study of metric spaces when one forgets about the small scale structure and focuses only on large scale. This philosophy underlies much of geometric group theory. Since we can forget about what happens at a small scale, we focus on metric spaces which are discrete, or, even more, uniformly locally finite. A metric space $(X,d)$ is uniformly locally finite (u.l.f.\ from now on) if for every $r>0$ we have $\sup_{x\in X}|B_r(x)|<\infty$, $B_r(x)$ being the ball of radius $r$ around $x$. Examples of such spaces are finitely generated groups with the word metric and discretisations of Riemannian manifolds. 

Associated to a u.l.f.\ metric space $(X,d)$ is a \cstar-algebra called the \emph{uniform Roe algebra of $X$}. Prototypical versions of this \cstar-algebra were introduced by Roe~\cite{Roe:1988qy} for index-theoretic purposes. The theory was consolidated in the 1990s, and uniform Roe algebras have since found applications in index theory (e.g.,~\cite{Spakula:2009tg,Engel:2018vm}), \cstar-algebra theory (\cite{Rordam:2010kx}), single operator theory (\cite{Rabinovich:2004xe}), topological dynamics (\cite{Kellerhals:2013aa}), and mathematical physics (\cite{Ewert:2019tr}). 

The formal definition of $\cstu(X)$ is the following. For a metric space $(X,d)$, the \emph{propagation} of an $X$-by-$X$ matrix\footnote{We view an $X$-by-$X$ matrix as a function $a\colon X\times X\to \mathbb C$.} $a=[a_{xy}]$ is
\[
\propg(a) \coloneqq \sup\{d(x,y)\mid a_{xy}\neq 0\}\in [0,\infty].
\]
If $a=[a_{xy}]$ has finite propagation and uniformly bounded entries, then $a$ canonically induces a bounded operator on the Hilbert space $\ell_2(X)$ as long as $(X,d)$ is u.l.f.. For any such $(X,d)$, the operators with finite propagation form a $^*$-algebra and its norm closure, a \cstar-algebra, is the uniform Roe algebra of $(X,d)$, denoted by $\cstu(X)$. The algebra $\cstu(X)$ contains the compacts $\mathcal K(\ell_2(X))$ as a minimal ideal. We define the \emph{uniform Roe corona of $X$}, denoted by $\roeq(X)$, as the quotient $\cstu(X)/\mathcal K(\ell_2(X))$. 

In the setting of coarse geometry, homeomorphisms are replaced by maps remembering the large scale structure.
\begin{definition}
If $(X,d)$ and $(Y,\partial)$ are metric spaces, a map $f\colon X\to Y$ is said to be \emph{coarse} if for every $r>0$ there is $s>0$ such that if $x,x'\in X$ are such that $d(x,x')<r$ then $\partial(f(x),f(x'))<s$. If $f\colon X\to Y$ and $g\colon Y\to X$ are coarse maps such that 
\[
\sup_{x\in X}d(x,g\circ f(x))<\infty\text{ and } \sup_{y\in Y}\partial(y,f\circ g(y))<\infty,
\]
$f$ and $g$ are called \emph{mutual coarse inverses}. In this case, each of $f$ and $g$ is a \emph{coarse equivalence}, and the spaces $X$ and $Y$ are said coarsely equivalent.
\end{definition}

The `rigidity problem for uniform Roe algebras'\footnote{The word rigidity is so flexible that it is obviously overused in mathematics.} asks whether two u.l.f.\ metric spaces $X$ and $Y$ whose uniform Roe algebras are isomorphic must be coarsely equivalent. This was recently solved in~\cite{BBFKVW.Roe}\footnote{Although uniform Roe algebras are nonseparable, they are separably representable and an isomorphism between them is implemented by a unitary. The rigidity problem for uniform Roe algebras is therefore absolute, i.e., its solution cannot depend on the set theoretic ambient.} The `quotient' version of this problem was formulated and studied in~\cite{braga2018uniform}. It is known as the rigidity problem of uniform Roe coronas.

\begin{problem}\label{Problem:RigRoeCoronas}%\emph{\textbf{(Rigidity of Uniform Roe Coronas)}}
Let $X$ and $Y$ be u.l.f.\ metric spaces such that $\roeq(X)$ and $\roeq(Y)$ are isomorphic. Does it follow that $X$ and $Y$ are coarsely equivalent? 
\end{problem}

In~\cite{braga2018uniform}, Problem~\ref{Problem:RigRoeCoronas} was solved in the presence of geometrical assumptions on the spaces involved, asking for the non-existence of certain `ghost' operators. Relying on the recent results obtained in~\cite{BBFKVW.Roe}, it is possible to remove these hypotheses and fully solve the rigidity problem for uniform Roe coronas.

\begin{theorem}[{\cite[Theorem 1.5]{BBFKVW.Roe}}]\label{T.RoeCoronas} Assume $\OCA$ and $\mathrm{MA}_{\aleph_1}$. 
Let $X$ and $Y$ be u.l.f.\ metric spaces. If $\roeq(X)\cong \roeq(Y)$, then $X$ and $Y$ are coarsely equivalent.
\end{theorem} 

We streamline a proof of this result. Fix two u.l.f.\ metric spaces $(X,d)$ and $(Y,\partial)$, and let $\Phi\colon\roeq(X)\to\roeq(Y)$ be an isomorphism. We focus on the algebra of propagation $0$ operators, corresponding to diagonal matrices. As characteristic functions on subsets of $X$ give diagonal matrices which in turn correspond to propagation $0$ operators, we have that $\ell_\infty(X)\subseteq\cstu(X)$. Since $\ell_\infty(X)\cap\mathcal K(\ell_2(X))=c_0(X)$, we have a canonical copy of $\ell_\infty/c_0$ sitting inside $\roeq(X)$, and therefore we have an injective $^*$-homomorphism
\[
\Phi\restriction \ell_\infty(X)/c_0(X)\colon\ell_\infty/c_0\to\roeq(Y).
\]
The methods of~\cite[\S6]{braga2018uniform} (assuming geometric assumptions on the spaces $X$ and $Y$), recently generalised to all u.l.f.\ spaces in~\cite[\S3]{BBFKVW.Roe}, show that, to get coarse equivalence, it is enough that $\Phi$ lifts on the canonical copy of $\ell_\infty(X)$.

\begin{proposition}\label{prp:liftdiagonal}
Let $(X,d)$ and $(Y,\partial)$ be u.l.f.\ metric spaces. Suppose that $\Phi\colon \roeq(X)\to\roeq(Y)$ is an isomorphism and that $\Phi\restriction \ell_\infty(X)/c_0(X)$ has a lifting $\Phi_*\colon \ell_\infty(X)\to\cstu(Y)$ which is a $^*$-homomorphism. Then $X$ and $Y$ are coarsely equivalent.
\end{proposition}
\begin{proof}[Sketch of a proof]
Using the geometric property on $Y$ (e.g.,~\cite[\S6]{braga2018uniform}), or the novel recent results contained in~\cite[\S3]{BBFKVW.Roe}, one can prove that 
\[
(\exists\delta>0)(\forall x\in X)(\exists y\in Y) \norm{\Phi_*(\chi_x)\chi_y}>\delta.
\]
This induces a function $f\colon X\to Y$, and such a function can be proved to be coarse. Similarly, by using that $\Phi$ is an isomorphism, one applies the above reasoning to get a coarse $g\colon Y\to X$. Using that $f$ is
constructed with an eye on $\Phi$, while $g$ gets constructed from $\Phi^{-1}$, one then shows that $f$ and $g$ are mutual coarse inverses, and therefore $X$ and $Y$ are coarsely equivalent.
\end{proof}

To get a $\Phi_*$ as above, we need Forcing Axioms. Using techniques similar to the ones of \S\ref{6bii.PFAAVi}, we show that under $\OCA$ and $\MA_{\aleph_1}$, $\Phi\restriction\ell_\infty(X)/c_0(X)$ must admit a Borel lifting $\ell_\infty(X)\to\cstu(Y)$. Consequently, using Ulam-stability type of results (in particular~\cite[Proposition 7.6 and 7.8]{mckenney2018forcing}), we can show that the $^*$-homomorphism $\Phi\restriction\ell_\infty(X)/c_0(X)\to\roeq(Y)$ is algebraically trivial (in the terminology of \S\ref{S.endo}), and therefore there exists a $^*$-homomorphism $\ell_\infty(X)\to\cstu(Y)$ that lifts $\Phi$.

Theorem~\ref{T.RoeCoronas} begs the question: 
Does the solution to the rigidity problem of uniform Roe coronas depend on set theory? Here is a soft partial result. Let $X=\{n^2\mid n\in\omega\}$ with the usual metric. In this case, all operators in $\mathcal B(\ell_2(X))$ with strictly positive propagation must be compact, hence $\cstu(X)=\ell_\infty(X)+\mathcal K(\ell_2(X))$, and $\roeq(X)=\ell_\infty(X)/c_0(X)$. If $\ell_\infty(\bbN)/c_0(\bbN)$ has nontrivial automorphisms, this gives an automorphism $\Phi$ of $\roeq(X)$ such that $\Phi\restriction \ell_\infty(X)/c_0(X)$ does not have a lifting $\Phi_*\colon \ell_\infty(X)\to\cstu(X)$ which is a $^*$-homomorphism, and therefore Proposition~\ref{prp:liftdiagonal} does not apply. Therefore, consistently, not all isomorphisms of uniform Roe coronas can be lifted $^*$-homomorphically on the canonical diagonal copy of $\ell_\infty$. In this case, though, we started by setting $X=Y$. More interestingly, 
by combining the idea from Theorem~\ref{FV-coro} with an analytic argument, in \cite[Theorem~6.5]{brian2024conjugating} it was proven that there are $2^{\aleph_0}$ coarsely non-equivalent uniformly locally finite metric spaces whose uniform Roe coronas are isomorphic under $\CH$. Together with Theorem~\ref{T.RoeCoronas} this gives an example of uniformly locally finite metric spaces  such that the assertion $\roeq(X)\cong \roeq(Y)$ is independent from $\ZFC$. 

\subsection{The Higson corona}
Another interesting quotient \cstar-algebra associated to a u.l.f.\ metric space $(X,d)$ is the \emph{Higson corona}. This algebra was introduced by Higson (yet it is sometimes referred to as the Higson--Roe corona) in connection with a $K$-theoretic analysis of the Roe index theorem for non-compact Riemannian manifolds. It is studied for its relevance in the connections between topology and coarse geometry (\cite{Dranishnikov:1997wd}) and index theory (e.g. \cite[\S6 and 7]{Roe1993} and \cite{Willett:2009jx}). In general, the Higson corona gives a unique access to study the ends (or directions) of a u.l.f.\ metric space, and specifically of a finitely generated group.

A bounded function $f\colon X\to\mathbb C$ is said to be \emph{slowly oscillating} (also known as \emph{a Higson function}) if 
\[
(\forall\epsilon,r>0)(\forall^\infty x,y\in X) (d(x,y)<r\Rightarrow |f(x)-f(y)|<\epsilon).
\]
In layman's terms, slowly oscillating functions become more and more constant as we go to infinity. Slowly oscillating functions belong to $\ell_\infty(X)$, and they form a \cstar-algebra denoted by $C_h(X)$. The \emph{Higson corona} of $X$, denoted by $C_\nu(X)$, is the quotient 
\[
C_\nu(X):=C_h(X)/C_0(X)\subseteq\roeq(X).
\]
The \emph{Higson remainder} $\nu X$ is the spectrum of the Higson corona, that is, the compact topological space such that $C_\nu(X)\cong C(\nu X)$. It is a good exercise (see \cite[Proposition 3.6]{baudier2023embeddings}) to show that $C_\nu(X)$ is exactly the center of the uniform Roe corona $\roeq(X)$. 

The following is (a restricted version of) Proposition 2.41 in \cite{roe2003lectures}.
\begin{proposition}
Let $X$ and $Y$ be infinite u.l.f.\ metric spaces. Every proper\footnote{In the setting of u.l.f.\ spaces properness reduces to `finite-to-one'.} coarse map $\varphi\colon X\to Y$ extends to a continuous map $\nu\varphi\colon\nu X\to \nu Y$. Also,  $\varphi$ is a coarse equivalence if and only if $\nu\varphi$ is a homeomorphism.
\end{proposition}

In \cite{Protasov.HC}, Protasov noticed that if the u.l.f.\ metric space $X$ has \emph{asymptotic dimension zero}\footnote{Asymptotic dimension is the large-scale analog of Lebesgue covering dimension. It was introduced in \cite{Gromov:1993tr} by Gromov in the context of geometric group theory.}, then the spectrum of $C_\nu(X)$ is a Parovi\v{c}enko space (see~\S\ref{S.CHModelTheory}).
This implies that the Higson corona of $X$ is a countably saturated \cstar-algebra elementarily equivalent to $\ell_\infty/c_0$. Parovi\v{c}enko's Theorem (see Corollary~\ref{cor:Parov}), as well as Keisler's Theorem (Theorem~\ref{T.Keisler}), then gives the following:

\begin{theorem}
Assume $\CH$. Let $X$ be a u.l.f.\ metric space of asymptotic dimension zero. Then $C_\nu(X)$ is isomorphic to $\ell_\infty/c_0$.
\end{theorem}
Since there are many coarsely inequivalent u.l.f.\ metric spaces of asymptotic dimension zero (e.g., if $X=\{n^2\mid n\in\omega\}$, then $X$ and $X^2$ are not coarsely equivalent), the result above shows that isomorphism of Higson coronas cannot detect coarsely equivalence, at least under $\CH$. 

What about the other side of the coin? The problem of establishing \emph{rigidity} of Higson coronas under suitable Forcing Axioms was studied by the fourth author, who in \cite{V.Higson} proved the following:
\begin{theorem}\label{thm:Higsoncoronas}
Assume $\OCA$ and $\mathrm{MA}_{\aleph_1}$.  Let $X$ and $Y$ be u.l.f.\ metric spaces. If $X$ and $Y$ have isomorphic Higson coronas then they are coarsely equivalent.
 \end{theorem}
In addition to Theorem~\ref{thm:Higsoncoronas}, \cite{V.Higson} characterises in a strong way all $^*$-homomorphisms between Higson coronas, stating the coarse version of the weak Extension Principle (see \S\ref{S.endo}) and proving it in presence of Forcing Axioms.

\section{This paper was too short for\dots}\label{S.other}
Given the scope of the present paper, it was inevitable that some of the topics closely related to its subject matter had to be omitted. We therefore feel obliged to indicate some of these topics, in spite of the fact that every topic mentioned brings us closer to additional omitted topics, conflicting with the finiteness requirement on this paper.

 \subsubsection{Dimension phenomena}\label{dim.phenomena} The reader may have noticed some asymmetry in our treatment of Boolean algebras and \cstar-algebras. In the latter case, the quotients were of the form $\cM(A)/A$, hence we considered quotients of a variety of \cstar-algebras $\cM(A)$. In the former case we (so far) considered only the quotients of $\cP(\bbN)$. There is however a rich rigidity theory of quotients $\cM/\cI$ where $\cM$ is a closed Boolean subalgebra of $\cP(\bbN)/\Fin$ and $\cI$ is an ideal of $\cM$. For example, the rigidity of homeomorphisms between the \v Cech--Stone remainders of countable locally compact spaces (i.e., the countable ordinals) discussed in \S\ref{S.Abel} are the Stone duals of quotients $\cM/\cI$ for a Borel subalgebra $\cM$ of $\cP(\bbN)$ and a Borel ideal $\cI$ on $\cM$. Another related thread of research arose from the consideration of dimension phenomena, initiated by Eric van Douwen in~\cite{vD:Prime}. 
 
 For $n\geq 2$, the tensor product $\bigotimes_n \cP(\bbN)$ can be identified with a Borel subalgebra of $\cP(\bbN^n)$. The Stone dual of this tensor product is $(\beta\bbN)^n$. Via this observation, the isomorphisms between finite powers of $\beta\bbN\setminus \bbN$ dualise to isomorphisms between quotients $\cM/\cI$ where $\cM$ is a Borel subalgebra of $\cP(\bbN)$ and $\cI$ is a Borel ideal, and hence   belong to our rigidity program. In~\cite{vD:Prime}, van Douwen proved that $(\beta\bbN\setminus \bbN)^m$ and $(\beta\bbN\setminus \bbN)^n$ are isomorphic if and only if $m=n$. This was extended in~\cite{Fa:Dimension}. The most appealing formulation of this result is topological: for every continuous function $f$ from a product of compact Hausdorff spaces $\prod_{\xi<\kappa} X_\xi$ into $\beta\bbN$\footnote{The theorem is more general: it holds when $\beta\bbN$ is replaced by any $\beta\bbN$-space, as defined in~\cite{vD:Prime}.} there is a partition of the domain into clopen sets such that the restriction of $f$ to each one of them depends on at most one coordinate (a similar result was proved from $\OCA$ in~\cite{Just:Omega^n}). This is one of the few nontrivial facts about $\beta\bbN\setminus \bbN$ that can be proved in $\ZFC$. The only known result along these lines for \cstar-algebras is that no corona of a separable \cstar-algebras can be presented as a tensor product of infinite-dimensional \cstar-algebras (\cite{Gha:SAW*}).

\subsubsection{The role of MA, $\OCAsharp$, and Biba's trick}\label{S.OCAsharp}
In~\cite{Ve:OCA}, it was proven that MA alone does not imply that all automorphisms of $\cP(\bbN)/\Fin$ are trivial, and in~\cite{shelah1994somewhere} it was shown that in Veli\v ckovi\'c's model every automorphism is somewhere trivial (see the last paragraph of \S\ref{S.OtherModels}). This was complemented in~\cite{shelah2002martin}, where a model of MA in which nowhere trivial automorphisms exist was constructed. These results explain, to some, why Shelah's eradication of nontrivial automorphisms proceeds in two stages, by considering somewhere trivial and nowhere trivial automorphisms separately. 

 MA is not needed in order to prove that all automorphisms of the Calkin algebra are inner, because of the \emph{isometry trick} (see Lemma \ref{lemma:partialiso}). In \cite[Theorem~1]{de2023trivial} it was finally proved that $\OCA$ alone imples all automorphisms of $\cP(\bbN)/\Fin$ are trivial. The proof filters through a variant of $\OCA$ called $\OCAsharp$ that follows from $\OCA$ by a modification of Moore's proof that $\OCA$ implies $\OCAi$ (\cite{moore2021some}, \cite[Theorem~3.3]{de2023trivial}). This axiom does not have the elegance of $\OCA$, but  its complexity is generously compensated by its utility in analysing automorphisms of quotient structures, via the so-called Biba's trick. It was used in \cite{farah2024biba} to prove that $\OCA$ and $\MAsigmalinked$ together imply the conclusion of the OCA Lifting Theorem (Theorem~\ref{thm:ilijasOCAideals}) and to show that all isomorphisms between quotients over countably $80$-determined ideals on $\bbN$ are topologically trivial, partially answering Question~\ref{Q.OCAlt} from an earlier version of this paper.

\subsubsection{The cardinality of the group of automorphisms of $\cP(\bbN)/\Fin$} In~\cite{steprans2003autohomeomorphism} it was proven that, in the presence of a topologically nontrivial automorphism, it can be any regular cardinal between $\mathfrak c$ and $2^{\mathfrak c}$.
% it can be anything reasonable, for any reasonable prescribed valued for $\mathfrak c$. In this paper it was also proved that $\ZFC$ is relatively consistent with the assertion that $\cP(\bbN)/\Fin$ has nontrivial automorphisms, but the cardinality of its automorphism group is only $\mathfrak c$. 

\subsubsection{A `commutative' version of the BDF question} An analog of Question~\ref{ques:bdf} has attracted some attention. Consider the action of $\bbZ$ on $\cP(\bbN)/\Fin$ where 
\[
n.[A]=[\{j+n\mid j\in A\}]
\]
 for $n\in \bbN$ and $A\in \cP(\bbN)$. Is there an automorphism $\Phi$ of $\cP(\bbN)/\Fin$ such that $1.\Phi([A])=\Phi(-1.[A])$ for all $A$?\footnote{It should be noted that a positive answer to this question would not be a progress towards solving Question~\ref{ques:bdf}. This is because Truss's theorem (see Example~\ref{Ex.AlCoMac}) implies that if an automorphism of the Calkin algebra sends the atomic masa $\ell_\infty(\bbN)/c_0$ to itself then its restriction to the normaliser of the atomic masa is implemented by a unitary; see~\cite[Notes to \S 17.9]{Fa:STCstar}.\label{footnote:Truss}} Such $\Phi$ cannot be trivial. Numerous partial positive results on this problem have been obtained in~\cite{geschke2010shift} and~\cite{silveira2016quotients}. This and related problems were tackled from the dynamic side in~\cite{brian2020isomorphism} and~\cite{brian2019universal}. A positive answer to this question was announced in \cite{brian2024does}.

 \subsubsection{Near actions}
 The concept of near action studied in~\cite{cornulier2019near} is closely related and highly relevant to questions on the rigidity of uniform Roe coronas associated with countable groups (see \S\ref{S.Roe} and~\cite[Example~9.4]{braga2018uniform}). 
 
\subsubsection{Rigidity for ultrapowers} \label{S.Ultrapowers} The abstract and general setup outlined in~\S\ref{S.Intro} is an attempt to capture the notion of a definable (Borel) quotient structure whose properties (automorphism group in particular) are sensitive to the choice of additional set-theoretic axioms. Needless to say, some quotient structures of general interest are not covered by our setup---for example, ultraproducts (see also \S\ref{S.Hardy}). 
 Throughout this survey, we considered only structures $\cM/E$ for a Borel $\cM$ and a Borel $E$. An important instance of the rigidity question is obtained by relaxing the Borelness requirement on~$E$, is the rigidity question for ultrapowers. For a given Borel structure $\cM$ (belonging to an axiomatisable category, so that the ultrapowers are defined) consider the Borel structure $\cM^\bbN$. If $\cU$ is an ultrafilter on $\bbN$, then the ultrapower $\cM_\cU$ of $\cM$ associated with $\cU$ is of the form $\cM^\bbN/E_\cU$ for an appropriately defined~$E_\cU$. The standard saturation arguments\footnote{We assume that the ultrafilters $\cU$ and $\cV$ are nonprincipal.} (\S\ref{S.CHModelTheory}) implies that, under $\CH$, $\cM_\cU$ has~$2^{\aleph_1}$ automorphisms and that $\cM_\cU$ and $\cN_\cV$ are isomorphic if and only if $\cM$ and $\cN$ are elementarily equivalent. Whether this conclusion is true without the $\CH$ depends on whether the theory of~$\cM$ is stable or not (see~\cite{FaHaSh:Model2} and~\cite{FaSh:Dichotomy}). 
 
 Much deeper rigidity results appear in Shelah's series of papers (\cite{shelah1992vive, shelah1994vive, shelah2008vive}). In the third paper it was proved that, relatively consistently with ZFC, there exists a nonprincipal ultrafilter $\cU$ on $\bbN$ such that if $\cM$ and $\cN$ are models of the canonical theory of the independence property then every isomorphism between $\cM_\cU$ and $\cN_\cU$ is topologically trivial (and even a product isomorphism). To say that the ramifications of this remarkable result have not been explored would be a gross understatement. 
 
 \subsubsection{Maximal Hardy fields} \label{S.Hardy} Another class of examples not covered by our setup, well-known to model-theorists but largely overlooked by contemporary set-theorists, is given by maximal Hardy fields (\cite{aschenbrenner2023filling}). Being defined as \emph{maximal} subfields of the ring of germs at $+\infty$ of real functions differentiable on an end-segment of $\bbR$ and closed under differentiation, their construction involves the Axiom of Choice. The exact analogy with the second part of \META{}.\ref{1.item.1} holds: CH implies that all maximal Hardy fields are isomorphic (\cite[Corollary~B]{aschenbrenner2023filling}). Since the proof involves a lemma resembling countable saturation (\cite[Lemma~10.1]{aschenbrenner2023filling}), there is a possibility that the analogy extends to \META{}.\ref{1.item.2} and that forcing axioms (or some other additional set-theoretic axioms) imply the existence of non-isomorphic maximal Hardy fields. It should be noted that CH implies that all maximal Hardy fields are isomorphic to a canonical object, the ordered field $\mathbf{No}(\omega_1)$ of surreal numbers of countable length. 
 Analogous remarks apply to maximal analytic Hardy fields (\cite[Theorem~A]{aschenbrenner2023maximal}), and \cite[\S 8]{aschenbrenner2023maximal} contains several questions about the structure of maximal analytic Hardy fields with strong set-theoretic flavour.

 \subsubsection{Borel liftings of the measure algebra}
A well-known question on the boundary of our framework is: Do forcing axioms imply that the measure algebra has no Borel lifting (see~\cite{Sh:PIF} for an oracle-cc consistency proof)? Any attempt at sealing gaps in the algebra Borel$/\Null$ (where $\Null$ denotes the $\sigma$-ideal of Lebesgue negligible subsets of $\bbR$) is hindered by the fact that, because of the countable chain condition, the gaps produced to witness the fact that a given partial Borel lifting cannot be extended to a Borel lifting are countable. All presently known gap-sealing techniques rely on uncountable combinatorics and the methods going back to~\cite{Hau:Summen},~\cite{Luz:Chastyah}, and~\cite{Ku:Gaps}. It is not known whether (uncountable) gaps can be frozen outside of the context of certain $F_{\sigma\delta}$ ideals on~$\bbN$; see~\cite{Fa:Luzin}. 

\subsubsection{Maximal rigidity}
In Example~\ref{Ex.AlCoMac} we saw that by \cite{truss1996recovering} another quotient group, $S_\infty/H$, (where $H$ denotes the normal subgroup of all finitely supported permutations of $\bbN$) has the property that (provably in $\ZFC$) each of its automorphisms is given as conjugation by an almost permutation, and in particular it has a continuous lifting. One consequence of this strong rigidity result is the following. Fix a pair $f$, $g$ in $S_\infty$, and let $\cM=(S_\infty,f)$ and $\cN=(S_\infty,g)$. Then $\cM/H$ and $\cN/H$ are isomorphic if and only if $f$ and $g$ are conjugate (by an almost permutation) in the quotient, and in particular the question of the isomorphism of these structures is independent from $\ZFC$. 
\begin{question}
What other structures (or categories) have this rigid behaviour, and is it possible to give a general characterisation of such structures? 
 \end{question}

\section{Absoluteness}\label{S.Absoluteness}
\enumtwo 
In this, concluding, section we discuss the necessity of set theory and the relevance of model theory to the rigidity problems discussed in the earlier parts of this survey. We also outline a very general framework for rigidity problems of this sort, extending one that the senior author has been promoting for years (see e.g.,~\cite{Fa:AQ},~\cite{Fa:Rigidity},~\cite{Fa:Liftings}).

\subsubsection{Absoluteness of topologically trivial isomorphisms} A strong evidence of the merit of the notion of a topologically trivial isomorphism (Definition~\ref{Def.Trivial}) is given by an absoluteness argument. By counting quantifiers, one sees that if $\cM/E$ and $\cN/F$ are Borel quotients in the same countable signature and $F\colon \cM\to \cN$ is Borel, then the assertion ``$F$ lifts an isomorphism'' is $\bPi^1_2$ and therefore subject to Shoenfield's Absoluteness Theorem (e.g.,~\cite{Kana:Book}). Similarly, the assertion ``$F$ lifts an embedding'' is $\bPi^1_1$, and therefore the assertion ``There is an embedding of $\cM/E$ into $\cN/F$ with a Borel lifting'' is absolute between transitive models of $\ZFC$ that contain all countable ordinals. The assertion ``There is an isomorphism between $\cM/E$ and $\cN/F$ with a Borel lifting'' is $\bSigma^1_3$, and therefore absolute between forcing extensions, at least once one assumes the existence of class many measurable cardinals.

\subsection{The general rigidity problem (again)} The problem introduced in \S\ref{S.Intro} can be construed as follows. Given a Borel structure $\cM$ (in this paper we mostly considered Boolean algebras and \cstar-algebras, only because these are the contexts in which we had something to say), consider the space of congruences of $\cM$. In our case, congruences were given by ideals, but depending on the category one can consider normal subgroups or arbitrary equivalence relations that are congruences with respect to the algebraic and relational structure of $\cM$. For a pair of Borel structures $\cM$ and $\cN$ of the same signature and congruences $\cI,\cJ$ of $\cM$ and $\cN$, respectively, consider the following two questions. 
\begin{enumerate}
\item\label{10.abs.1} Is the assertion that $\cM/\cI\cong \cN/\cJ$ relatively consistent with $\ZFC$? 
\item\label{10.abs.2} Is the assertion that $\cM/\cI\cong \cN/\cJ$ provable in $\ZFC$? 
\end{enumerate}
We start with \ref{10.abs.1}, deferring the discussion of (a modified form of) \ref{10.abs.2} to \S\ref{S.Non-isomorphism}. 
If $\cI$ and $\cJ$ are Borel, or even projective, then the assertion that $\cM/\cI\cong \cN/\cJ$ is a $\Sigma^2_1$ statement: 
\begin{quote}
There exists $f\colon \cM\to \cN$ such that $\Phi_f([a]_\cI)=[f(a)]_\cJ$ is well-defined and an isomorphism between $\cM/\cI$ and~$\cN/\cJ$. 
\end{quote}
By Woodin's $\Sigma^2_1$-absoluteness theorem (see e.g.,~\cite{Lar:Stationary}), if~\ref{10.abs.1} can be forced to hold and there are class many measurable Woodin cardinals in the universe, then $\cM/\cI\cong \cN/\cJ$ in every forcing extension that satisfies $\CH$. For metamathematical reasons, this is strictly weaker than the assertion that if $\cM/\cI\cong \cN/\cJ$ is relatively consistent with $\ZFC$ then $\CH$ implies it. Nevertheless, this observation usually gives a good lead on how to approach a problem at hand, shifting focus from set-theoretic methods to model-theoretic analysis of the theories of $\cM/\cI$ and $\cN/\cJ$, as well as the saturation properties of these structures.\footnote{As we have seen in \S\ref{S.Cohomology}, in the case of \cstar-algebras model theory has to be supplemented by other methods.}

\subsubsection{Absoluteness of model-theoretic properties} The theory of a given structure is absolute between transitive models of $\ZFC$. However, if the structure in question is Borel, then it is re-evaluated from its Borel code in every model. This process results in not necessarily isomorphic structures. 

\begin{proposition}\label{P.10.1} Suppose that $\cM$ is a projective structure (of classical, discrete, logic) and $\cI$ is a Borel congruence on $\cM$. For every formula $\varphi(\bar x)$ of the signature of $\cM$ and every tuple $\bar a$ in $\cM/\cI$ of the appropriate sort, the assertion
\[
\cM/\cI\models \varphi(\bar a)
\]
is a projective statement. 
\end{proposition}

We will sketch the proof of a slightly more precise statement in case $\cI$ is Borel. If $\varphi$ is quantifier-free, then the assertion is clearly $\Delta^1_1$. By induction on complexity of $\varphi$ one shows that if $\varphi(\bar x)$ is $\Sigma_n$ ($\Pi_n$) then $\cM/\cI\models \varphi(\bar a)$ is $\Sigma^1_n$ ($\Pi^1_n$). The proof in the case when $\cI$ is projective is similar.

An analogous proof gives Proposition~\ref{P.10.2} below. The (hopefully inoffensive) phrase `Borel metric structure' is short for `a structure with a metric signature which is also Borel'. 

\begin{proposition} \label{P.10.2} Suppose that $\cM$ is a Borel metric structure and $\cI$ is a Borel congruence on $\cM$. For every formula $\varphi(\bar x)$ of the signature of $\cM$, every $r\in \bbR$, and every tuple $\bar a$ in $\cM/\cI$ of the appropriate sort, the assertion
\[
\cM/\cI\models \varphi(\bar a)\leq r
\]
is a projective statement. More precisely, if $\cI$ is Borel and $\varphi$ is $\Sigma_n$ ($\Pi_n$) then this assertion is $\Sigma^1_n$ ($\Pi^1_n$). 
\end{proposition}

The assumption of Corollary~\ref{C.10.3} is, unlike the (to an untrained ear, similar-sounding) assumption of Woodin's $\Sigma^2_1$ absoluteness theorem, within reach of inner model theory (\cite{Nee:Inner}) and therefore fairly innocuous.\footnote{By $\Th(A)$ we denote the theory of a structure $A$.} 

\begin{corollary} \label{C.10.3} Suppose that there exist class many measurable Woodin cardinals. If $\cM$ is a Borel structure (discrete or metric) and $\cI$ is a projective congruence on $\cM$, then neither $\Th(\cM/\cI)$ nor the truth of the assertion `$\cM/\cI$ is countably saturated' can be changed by forcing. 
\end{corollary}

Note that the theory of the standard model of second-order arithmetic, which is a Borel structure (see e.g.,~\cite{Kana:Book}) can be changed by forcing if no large cardinals are assumed. 

For the theory of $\cM/\cI$ this is a consequence of Proposition~\ref{P.10.1}, Proposition~\ref{P.10.2}, and the fact that if there are class many Woodin cardinals then the truth in $L(\bbR)$ cannot be changed by forcing (e.g.,~\cite{Kana:Book}). 

Countable saturation of $\cM/\cI$ is not a projective statement since there is no guarantee that all formulas of the type in question will be $\Sigma_n$ for a fixed $n$. Nevertheless, it is clearly a statement of $L(\bbR)$, and Corollary~\ref{C.10.3} follows.

Under what conditions is the quotient $\cM/\cI$ as in Corollary~\ref{C.10.3} countably saturated? As we have seen in \S\ref{S.6aii}, under general conditions one can only expect a restricted form of saturation, such as the degree-1 saturation in the category of \cstar-algebras. It is tempting to conjecture that if $\cI$ is sufficiently ergodic---in the sense that each equivalence class $[a]_\cI$ is dense in the sort of $a$ as interpreted in $\cM$---then $\cM/\cI$ has some (nontrivial and useful) form of countable saturation. 

\subsubsection{BDF}
The Brown--Douglas--Filmore question (\S\ref{BDF}), whether there exists an automorphism of the Calkin algebra that sends the unilateral shift to its adjoint, can be considered within the framework of~\ref{10.abs.1}. Considering $\cM=(\cB(H),s)$ and $\cN=(\cB(H),s^*)$ as structures in the language of \cstar-algebras expanded by a constant interpreted by the shift or by its adjoint, the question is whether $\cM/\cK(H)$ is isomorphic to $\cN/\cK(H)$. The lamentable status of this question can be summarised in two remarks. First, it is not even known whether these two structures are elementarily equivalent. Second, even if they were, the Calkin algebra---and therefore each of these two structures---fails even to be countably homogeneous (see \S\ref{S.6aii}). Nevertheless, Woodin's $\Sigma^2_1$ absoluteness theorem suggests that assuming the Continuum Hypothesis may help. In addition, forcing $\CH$ without adding reals preserves $\Sigma^2_1$ statements and therefore assuming $\CH$ does no harm in this context.

\subsubsection{Non-isomorphism}\label{S.Non-isomorphism}
In order to avoid metamathematical detours, instead of question~\ref{10.abs.2} we consider a question that is, given the current state of the art in set theory, easier to analyse: When does $\cM/\cI\not\cong \cM/\cJ$ hold in a forcing extension? 

\begin{conjecture} \label{10.C.1} Suppose that there exist class many supercompact cardinals.\footnote{The lavish large cardinal assumption assures that every set-forcing extension has a further set-forcing extension that satisfies Martin’s Maximum.} Suppose that $\cM$ and $\cN$ are Borel structures (discrete or metric) of the same signature and that $E$ and $F$ are Borel congruences on $\cM$ and $\cN$, respectively. If $\cM/E$ and $\cN/F$ are non-isomorphic in some forcing extension, then they are non-isomorphic in every forcing extension that satisfies Martin's Maximum (MM).
\end{conjecture}

In this context, the choice of MM, provably the strongest forcing axiom,\footnote{Sort of---see~\cite{viale2015category}.} may appear to be arbitrary, but the recent unification of forcing axioms with Woodin's axiom (*) in~\cite{aspero2021martin} gives additional weight to the view that forcing axioms provide a coherent alternative to the $\CH$. In all the (many!) instances in which Conjecture~\ref{10.C.1} has been confirmed (\S\ref{6bii.PFAAVi}), $\OCAsharp$ and MA (see \S\ref{S.OCAsharp}) suffice. As we have seen in \S\ref{6bi} and \S\ref{6bii.PFAAVi}, in the case of Boolean algebras $\cP(\bbN)/\cI$ for a Borel ideal $\cI$, a confirmation of Conjecture~\ref{10.C.1} relies on whether gaps in the corresponding Boolean algebra can be frozen (i.e., can be made indestructible by an $\aleph_1$-preserving forcing) in $\cP(\bbN)/\cI$. This extends, using $\OCAsharp$ in place of $\OCA$, to reduced products of fields, linear orders, trees, and sufficiently random graphs (\cite{de2023trivial}, see \S\ref{S.fields-etc}). As noted in \S\ref{6bi}, in~\cite{FaSh:Trivial} and~\cite{Gha:FDD} it was proved that there exists forcing extensions of the universe in which every isomorphism 
\[
\Phi\colon \cP(\bbN)/\cI\to \cP(\bbN)/\cJ
\]
has a continuous lifting. Since counting the quantifiers shows that the existence of such isomorphism is a $\Sigma^1_3$ statement, granted a mild large cardinal axiom (class many measurable cardinals, see e.g.,~\cite{Kana:Book}), the existence of a continuous lifting of $\Phi$ is absolute between set-forcing extensions. Therefore in these models two quotients of the form $\cP(\bbN)/\cI$ are isomorphic if and only if they are isomorphic in some forcing extension. It is however not clear whether this conclusion follows from forcing axioms. In the case when $\cI$ is an $F_{\sigma\delta}$ ideal, the results of~\cite{Fa:Luzin} come close to confirming Conjecture~\ref{10.C.1} (see Question~\ref{Q.countablydetermined}). The most frustrating instance of this conjecture is when $\cI$ is the Fubini product $\Fin\otimes\Fin=\{A\subseteq \bbN^2\mid (\forall^\infty m)(\forall^\infty n) (m,n)\notin A\}$ (as common, for convenience $\cP(\bbN)$ is replaced with the isomorphic $\cP(\bbN^2)$.)

\bibliographystyle{plain}
\bibliography{ifmainbib}
\end{document}